\documentclass[11pt]{article}

\date{}
\usepackage[utf8]{inputenc} % allow utf-8 input
\usepackage[T1]{fontenc}    % use 8-bit T1 fonts
\usepackage[colorlinks=true,linkcolor=blue,allcolors=blue]{hyperref}       % hyperlinks
\usepackage{url}            % simple URL typesetting
\usepackage{booktabs}       % professional-quality tables
\usepackage{amsfonts}       % blackboard math symbols
\usepackage{nicefrac}       % compact symbols for 1/2, etc.
\usepackage{microtype,nicefrac}      % microtypography
\usepackage{xspace}
\usepackage{natbib}
\usepackage[toc]{appendix}
\usepackage{minitoc}
\usepackage{subcaption}

\usepackage{comment}
\usepackage{amsmath,amsthm}% \usepackage{amsthm}
\usepackage{graphicx}
\usepackage{rotating}

\usepackage{amssymb}
\usepackage{autobreak}
\usepackage{mathtools}
\usepackage[dvipsnames]{xcolor}
\usepackage{bbm}
\usepackage[ruled,vlined, linesnumbered]{algorithm2e}
\usepackage{algorithmic}
\usepackage[parfill]{parskip}

\usepackage[toc]{appendix}
\usepackage{minitoc}

\usepackage{bigints}
\usepackage{mathabx}

\makeatletter
\DeclareMathOperator*{\argmax}{arg\,max}
\DeclareMathOperator*{\argmin}{arg\,min}

\DeclarePairedDelimiterXPP{\exof}[1]{\ex}{[}{]}{}{%					% for conditional expectations
 #1}

\newcommand*{\defeq}{\mathrel{\vcenter{\baselineskip0.5ex \lineskiplimit0pt
                     \hbox{\scriptsize.}\hbox{\scriptsize.}}}%
                     =}

\newcommand{\Real}{\mathbb{R}}

\newcommand{\X}{\mathcal X}

\newcommand{\Y}{\mathcal Y}

\renewcommand{\S}{\mathcal S}

\newcommand{\N}{\mathcal N}

\newcommand{\R}{\mathcal{R}}

\newcommand{\cgo}{\textnormal{\textsl{CGO}}\xspace}
\newcommand{\ocgo}{\textnormal{\textsl{oCGO}}\xspace}
\newcommand{\cgd}{\textnormal{\textsl{CGD}}\xspace}
\newcommand{\gda}{\textnormal{\textsl{GDA}}\xspace}
\newcommand{\omda}{\textnormal{\textsl{OMDA}}\xspace}
\newcommand{\mda}{\textnormal{\textsl{MDA}}\xspace}
\newcommand{\agda}{\textnormal{\textsl{AGDA}}\xspace}

\newcommand{\Reals}{\mathbb{R}}

\usepackage{float}

\newtheorem{lemma}{Lemma}[section]
\newtheorem{example}{Example}[section]

\newtheorem{theorem*}{Theorem}[section]

\newtheorem{definition}{Definition}[section]
\newtheorem{theorem}{Theorem}%[section]
\newtheorem{proofsk*}{Proof Sketch}%[section]

\title{Competitive Gradient Optimization}

\author{%
Abhijeet Vyas\\
  Purdue University\\
  \texttt{vyas26@purdue.edu} 
\and
Kamyar Azizzadenesheli\\
  Purdue University\\
  \texttt{kamyar@purdue.edu} \\
}

\begin{document}
\doparttoc % Tell to minitoc to generate a toc for the parts
\faketableofcontents % Run a fake tableofcontents command for the partocs

\maketitle

\begin{abstract}
We study the problem of convergence to a stationary point in zero-sum games. We propose competitive gradient optimization (\cgo), a gradient-based method that incorporates the interactions between the two players in zero-sum games for optimization updates. We provide continuous-time analysis of \cgo and its convergence properties while showing that in the continuous limit, \cgo predecessors degenerate to their gradient descent ascent (\gda) variants. We provide a rate of convergence to stationary points and further propose a generalized class of $\alpha$-coherent function for which we provide convergence analysis. We show that for strictly $\alpha$-coherent functions, our algorithm convergences to a saddle point. Moreover, we propose optimistic \cgo (\ocgo), an optimistic variant, for which we show convergence rate to saddle points in $\alpha$-coherent class of functions.

% \newline
% \newline
% \input{Sections/0-Abstract}
\end{abstract}

\section{Introduction}

We study the zero-sum simultaneous two-player optimization problem of the following form,
\begin{align}
\min_{x\in \mathcal{X}} f(x,y),~~~\max_{y\in \mathcal{Y}} f(x,y)
\label{eq:minimax}
\end{align}

where $x$ and $y$ are players moves with $\mathcal{X}\subseteq \Real^m,\mathcal{Y}\subseteq \Real^n$ and $f$ is a scalar value map from $\mathcal{X}\times \mathcal{Y} \rightarrow \Real$. Such an optimization problem, also known as minimax optimization problems, has numerous applications in machine learning and decision theory, some examples include competitive Markov decision processes~\citep{comdp}, e.g., game of StarCraft and Go~\citep{starcraft,go}, adversarial learning and robustness learning~\citep{adv1,namkoong2016stochastic,adv2}, generative adversarial networks (GAN)~\citep{goodfellow2014generative,radford2015unsupervised,arjovsky2017wasserstein}, and risk assessment~\citep{artzner1999coherent}.

Gradient descent ascent (\gda) is the standard first-order method to approach the minimax optimization problem in Eq.~\eqref{eq:minimax} and is known to converge for  \textit{strictly}-coherent functions~\citep{mertikopoulos2018optimistic} which subsumes the  \textit{strictly} convex-concave function class~\citep{facchinei2003finite}. Yet, \gda cycles or diverges on simple functions with interactive terms between the players, e.g., a function like $f(x,y) = y^\top x$~\citep{mertikopoulos2018optimistic}. To tackle this issue, \citet{schafer2019competitive} proposes competitive gradient descent (\cgd) which includes the bi-linear approximation of the function as opposed to only the linear approximation used in \gda to formulate the local update. In this approach, despite being bilinear, the game approximation per player is linear. With this update, \cgd is able to utilize the interaction terms to guarantee convergence in some non-convex concave problems rather than be impeded by them.

\citet{daskalakis2018training} proposes to extend the online learning algorithm optimistic mirror descent ascent (\omda) \citep{rakhlin2013online} to two player games and shows convergence of the method for all bi-linear games of the form $f(x,y) = y^\top A x$ (thereby for $y^\top x$). \citet{mertikopoulos2018optimistic} uses the extra-gradient version of \omda to show convergence for all coherent saddle points which includes the saddle points in bi-linear-games of the form $f(x,y) = y^\top A x$ . However, we show, \cgd and \omda (as defined in \citep{mertikopoulos2018optimistic}) reduce to \gda and mirror descent ascent (\mda) respectively in the continuous-time limit (gradient-flow). The continuous-time regime has given insights into the behavior of single-player optimization algorithms~\citep{wilson1611lyapunov,lee2016gradient} and has been used to study games in~\citep{mazumdar2020gradient}.

We propose competitive gradient optimization (\cgo), an optimization method that incorporates players' interaction in order to come up with gradient updates. \cgo considers local linear approximation of the game and introduces the interaction terms in the linear model. At an iteration point $(x,y)$, the \cgo update is as follows,
\begin{align}\label{eqn:localgame0}
    \begin{split}
    \argmin_{\Delta x \in \mathcal{X}} \Delta x^{\top} \nabla_x f &+\frac{\alpha}{\eta} \Delta x^{\top} \nabla_{xy}^2 f\Delta y + \Delta y^{\top} \nabla_y f + \frac{1}{2\eta} \Delta x^{\top} \Delta x \\
    \argmax_{\Delta y \in \mathcal{Y}}  \Delta y^{\top} \nabla_y f &+\frac{\alpha}{\eta} \Delta y^{\top} \nabla_{yx}^2 f \Delta x + \Delta x^{\top} \nabla_x f - \frac{1}{2\eta} \Delta y^{\top} \Delta y.
    \end{split}
\end{align}
where the first term in the update is a local linear approximation of the game. The second term is the interaction term between players which is scaled with $\alpha$ to represent the importance of incorporating the interaction in the update. This scaling analogs to scaling in Newton methods with varying learning rates. This approximation results in a local bi-linear approximation of the game. And finally, $\eta$ is the learning rate appearing in the penalty term. It is important noting that fixing the other player, the optimization for each player is a linear approximation of the game. Since the game approximation for each player is linear in its action, we consider this update yet a linear update. The solution to the \cgo update is the following,
\begin{align}
\begin{bmatrix}
\Delta x\\
\Delta y
\end{bmatrix} &=-\eta g_\alpha := -\eta \begin{bmatrix}
I & \alpha \nabla_{xy}\\
-\alpha \nabla_{yx}f & I
\end{bmatrix}^{-1}\begin{bmatrix}
\nabla_x f\\
-\nabla_y f 
\end{bmatrix} 
\end{align}
where $g_\alpha$ is the gradient update at point $(x,y)$. \cgo is a generalization of its predecessors, in the sense that, setting $\alpha=0$ recovers \gda, and setting $\alpha=\eta$ recovers \cgd. \cgo gives greater flexibility for the updates in the hyper-parameters and gives rise to a distinct algorithm in continuous-time. In large-scale practical and deep learning settings, this update can be efficiently and directly computed using an optimized implementation of conjugate gradient and Hessian vector products.

Further, we introduce generalized versions of the Stampacchia and Minty variational inequality~\citep{facchinei2003finite} and extend the definition of coherent saddle points~\citep{mertikopoulos2018optimistic} to $\alpha$-coherent saddle points and show the convergence of \cgo  under $\alpha$-coherence. Finally, we propose optimistic \cgo which converges to the saddle points for $\alpha$-coherent saddle point problems which are not  \textit{strictly} $\alpha$-coherent. 

Our main contributions are as follows:

\begin{itemize}
    \item We propose \cgo that utilizes bi-linear approximation of the game in Eq.~\eqref{eq:minimax} and accordingly weights the interaction terms between agents in the updates.
    \item In order to study whether \cgo provides a fundamentally new component, we study \cgo's and its predecessors' behaviors in continuous-time. We observe that in the limit of the learning rate approaching zero, i.e., continuous-time regime, the \cgd and \omda \textit{reduce} to their \gda and \mda counterparts and \cgo gives rise to a \textbf{distinct update in the continuous-time.}
    \item Using the standard Lyapunov analysis machinery, we show that \cgd and \gda convergence for \textit{strictly convex-concave functions}, and \cgo allows for \textbf{arbitrary negative eigen-values} in the pure hessian of the minimizer and \textbf{arbitrary positive values} for the maximizer.
    \item We extend the definition of coherence function class~\citep{mertikopoulos2018optimistic} to $\alpha$-coherent functions for which we show the optimistic variant of \cgo, optimistic competitive gradient optimization \ocgo converges to saddle points with a desired rate while\textit{ \cgo} converges to the saddle points which satisfy the \textit{strict $\alpha$-coherence} condition. We provide families of functions satisfying the above.
\end{itemize}

\section{Related works}

Largely the algorithms proposed to solve the minimax optimization problem can be divide into 2-parts, those containing simultaneous update which solve a simultaneous game locally at each iteration and those containing sequential updates. While our work focuses on the simultaneous updates, sequential updates are relevant due to their close proximity and the fact that they often time gives rise to relevant solutions. We discuss the work done in the 2 aforementioned categories below:  

\paragraph{Sequential updates}
A sequential version \gda in an alternating form is alternating gradient descent ascent (\agda) is often time shown to be more stable than its simultaneous counter-part~\citep{gidel2019negative,bailey2020finite}. \citet{yang2020global} introduces the 2-sided Polyak-Lojasiewicz (PL)-inequality. The PL-inequality was first introduced by \citet{polyak1963gradient} as a sufficient condition for gradient descent to achieve a linear convergence rate, \citet{yang2020global} shows that the same can be extended to achieve convergence of \agda to saddle points, which are the only stationary points for the said functions. Yet, \agda also cycles in several problems including bi-linear functions showing the persist difficulty of cycling behavior for any \gda algorithm. To solve this problem, 2-time scale gradient descent ascent is proposed~\citep{heusel2017gans,goodfellow2014generative,metz2016unrolled,prasad2015two} which use different learning rates for the descent and ascent. \citet{heusel2017gans} proves its convergence to local nash-equilibrium (saddle points). \citet{jin2020local} discusses the limit points of 2-time scale \gda by defining local minimax points, analogs of the local nash-equilibrium in the sequential game setting and shows that for vanishing learning rate for the descent, 2-time scale \gda provably converges to local mini-max points. Another line of work concerns itself with finding stationary points of the function $F(x) = \max_y f(x,y)$, \citep{lin2020gradient,rafique1810non,nouiehed2019solving,jin2019minmax}. The 2-time scale approaches mainly rely on the convergence of one player per update step of the other player, which makes these updates generally slow to converge.

\paragraph{Simultaneous updates}

Simultaneous update methods preserve the simultaneous nature of the game at each step, such methods include \omda \citep{daskalakis2018training}, its extra-gradient version~\citep{mertikopoulos2018optimistic}, ConOpt~\citep{mescheder2017numerics}, \cgd~\citep{schafer2019competitive}, LOLA~\citep{foerster2017learning}, predictive update~\citep{yadav2017stabilizing} and symplectic gradient adjustment~\citep{balduzzi2018mechanics}. Of the above, \citep{daskalakis2018training,mertikopoulos2018optimistic,foerster2017learning} are inspired from no-regret strategies formulated in \citep{rakhlin2013online,jadbabaie2015online} based on follow the leader \citep{shalev2006convex,grnarova2017online} for online learning. \citep{schafer2019competitive} uses the cross-term of the Hessian, while \citep{mescheder2017numerics} uses the pure terms to come up with a second order update. \citep{balduzzi2018mechanics} proposes an update based on the asymmetric part of the game Hessian obtained from its Helmholtz decomposition. Some of these algorithms converge to stationary points that need not correspond to saddle points, \citet{daskalakis2018limit} shows that \gda, as well as optimistic \gda, may converge to stationary points which are not saddle points. ConOpt is shown to converge to stationary points which are not local Nash equilibrium in the experiments~\citep{schafer2019competitive}.

\section{Preliminaries}
In this section, we describe the simultaneous minimax optimization problem and notations to express the properties of functions we use in the analysis. We discuss the class of $\alpha$-coherent functions which extends the definition of coherence in \cite{mertikopoulos2018optimistic} and for different versions of which \cgo and  \ocgo converge to the saddle point.

Throughout the paper we often denote the concatenation of the arguments $x$ and $y$ to be $z := (x,y)$.

\begin{definition}[First order stationary point]
A point $z^*=(x^*,y^*)\in \mathcal{X}\times\mathcal{Y}$ is a stationary point of the optimization Eq.~\eqref{eq:minimax} if it satisfies the following,
\begin{equation}
\begin{aligned}
\nabla_x f(x^*,y^*) = \textbf{0}, \nabla_y f(x^*,y^*) = \textbf{0}
\end{aligned}
\end{equation}
\end{definition}

% We now define some properties of the function $f$ that we use to show convergence of \cgo and Optimistic \cgo.

% \paragraph{Function $f$ is Lipschitz Continuous}
We say a function $f$ is $L$ Lipschitz continuous if for any two points $z_1:=(x_1,y_1)\in  \mathcal{X}\times\mathcal{Y}$ and $z_2=:(x_2,y_2) \in  \mathcal{X}\times\mathcal{Y}$, it satisfies
$$
| f(z_1)- f(z_2)| \leq L \|z_1-z_2\|_2 \\
$$
where $|\cdot|$ denote the absolute value and $\|\cdot\|_2$ denote the corresponding $2$-norm in the product space $\mathcal{X}\times\mathcal{Y}$. Similarly, for a given function $f$, we say it has $L'$-Lipschitz continuous gradient if for any two points $z_1:=(x_1,y_1)\in  \mathcal{X}\times\mathcal{Y}$ and $z_2=:(x_2,y_2) \in  \mathcal{X}\times\mathcal{Y}$, it satisfies
$$
\|\nabla  f(z_1)-\nabla f(z_2)\|_2 \leq L' \|z_1-z_2\|_2
$$
And finally, we say a function has $(L_{xx},L_{yy},L_{xy})$-Lipschitz continuous Hessian, if similarly, for any two points $z_1:=(x_1,y_1) \in  \mathcal{X}\times\mathcal{Y}$ and $z_2:=(x_2,y_2) \in \mathcal{X}\times\mathcal{Y}$ the followings hold,
\begin{align*}
 \|\nabla_{xx}^2  f(z_1)-\nabla_{xx}^2 f(z_2)\|_2 \leq L_{xx} \|z_1-z_2\|_2\\
\|\nabla_{yy}^2  f(z_1)-\nabla_{yy}^2 f(z_2)\|_2 \leq L_{yy} \|z_1-z_2\|_2\\
\|\nabla_{xy}^2  f(z_1)-\nabla_{xy}^2 f(z_2)\|_2 \leq L_{xy} \|z_1-z_2\|_2   
\end{align*}
where all the norms are $2$-norms with respect to their corresponding suitable definition of native spaces.\\

We present the notation for the minimum and maximum value of matrices derived from the Hessian of $f$. The extremums are evaluated over the complete domain of $f$.
\begin{table}[H]
  \caption{Eigenvalue notations for matrices derived from 2\textsuperscript{nd} derivates to simplify notation}
  \label{sample-table}
  \centering
  \begin{tabular}{lcc}
    \toprule
    Matrix     & Minimum Eigenvalue       & Maximum Eigenvalue  \\
    $\nabla_{xx}^2f$ & $\underline{\lambda_{xx}}$ & $\overline{\lambda_{xx}}$     \\
    \midrule
    $\nabla_{yy}^2f$ & $\underline{\lambda_{yy}}$ & $\overline{\lambda_{yy}}$     \\
    \midrule
    $\nabla_{xy}^2f\nabla_{yx}^{2} f$     & $\underline{\lambda_{xy}}$       &
    $\overline{\lambda_{xy}}$  \\
    \midrule
    $\nabla_{yx}^{2}f\nabla_{xy}^2 f$     & $\underline{\lambda_{yx}}$       & $\overline{\lambda_{yx}}$  \\    \bottomrule
  \end{tabular}
\end{table}
Further we define $\overline{\lambda_{1}}=\max(\overline{\lambda_{xx}},-\underline{\lambda_{yy}}),\overline{\lambda_{2}}=\max(\overline{\lambda_{xx}},\overline{\lambda_{yy}})$. We also have $\underline{\lambda_{xy}},\underline{\lambda_{yx}}\geq 0$ since $\nabla_{xy}^2f\nabla_{yx}^{2} f,\nabla_{yx}^2f\nabla_{xy}^{2} f$ are positive semi-definite.

\begin{definition}[Bregman Divergence]
The Bregman divergence with a strongly convex and differentiable potential function h is defined as $$\mathcal{B}_h(x,y) = h(x)-h(y)-\langle x-y,\nabla h(y) \rangle$$
\end{definition}

\paragraph{Saddle point (SP)} We define the solutions of the following problems to be $\min-\max$ and $\max-\min$ saddle points respectively,
\begin{equation}\label{Saddle Point1}
  \hspace{-5cm} \bullet~~\textit{$\min-\max$ saddle point:}~~~~~~~~~\min_{x\in \mathcal{X}} \max_{y\in \mathcal{Y}} f(x,y)
\end{equation}
\begin{equation}\label{Saddle Point2}
  \hspace{-5cm} \bullet~~\textit{$\max-\min$ saddle point:}~~~~~~~~~\max_{x\in \mathcal{X}} \min_{y\in \mathcal{Y}}  f(x,y)
\end{equation}
We now introduce modified forms of the Stampacchia and Minty variational inequalities and present the definition of $\alpha$-coherent saddle point problems.
\begin{definition}[$\alpha$-Variational inequalities]
$\alpha$-coherence generalizes the definition of coherent saddle points in \cite{mertikopoulos2018optimistic} which sets $\alpha$ to zero. The definition of $\alpha$-coherence hinges on the following two variational inequalities ($g_{\alpha}$ as in Eq.~\eqref{eqn:localgame0}),

\begin{itemize}
    \item $\alpha-MVI$ : $g_{\alpha}(x,y)^\top(z-z^*) \geq 0$ for all  $z : (x,y) \in \mathcal{X}\times\mathcal{Y}$
    \item $\alpha-SVI$ : $g_{\alpha}(x^*,y^*)^\top(z-z^*) \geq 0$ for all  $z: (x,y) \in \mathcal{X}\times\mathcal{Y}$
\end{itemize}
\end{definition}
 
\begin{definition}[$\alpha$-coherence]\label{coherencedef}
We say that $\min-\max$ SP problem is $\alpha$-coherent if, 
\begin{itemize}
\item  Every solution of $\alpha-SVI$ is also a $\min-\max$ SP.
\item  There exists a $\min-\max$ SP, p that satisfies $\alpha-MVI$
\item  Every $\min-\max$ SP, $(x^*,y^*)$ satisfies $\alpha-MVI$ locally, i.e., for all $(x,y)$ sufficiently close to $(x^*,y^*)$
\end{itemize}
The $\alpha$-coherent $\max-\min$ SP problem is defined similarly.

\end{definition}

In the above, if $\alpha-MVI$ holds as a strict inequality whenever x is not a solution thereof, SP problem will be
called \textit{strictly} $\alpha$-coherent; by contrast, if $\alpha-MVI$ holds as an equality for all $(x,y) \in \mathcal{X}\times\mathcal{Y}$, we will say that the SP problem is null $\alpha$-coherent.

\section{Motivation}\label{Section4}

In this section, we present the main motivations of our approach. The first is the popularity of the damped Newton method~(Algorithm 9.5, \citep{boyd2004convex}) which scales the second-order term in the Taylor-series expansion of the function to come up with the local update. The second is the observation that both \cgd and \omda reduce to \gda and \mda in the continuous-time limit which calls for a new algorithm that is distinct from \gda and \mda in continuous-time. The third is the observation that several functions give rise to (SP) problems which are \textit{strictly} $\alpha$-coherent$,\quad\forall \alpha>0$ but not \textit{strictly}-coherent as defined in \cite{mertikopoulos2018optimistic} which coincides with $\alpha$-coherence when we set $\alpha=0$
% a \textit{strictly}-coherent )  $f(x,y)=x^{\top}Ay$ is a \textit{strictly} $\alpha-coherent$ SP problem for which \cgd\textsuperscript{+} converges to the saddle point and 

\subsection{Adjustable learning rate Newton method}
The celebrated Newton method in one player optimization gives rise to an update which is the solution to the following local optimization problem; For a function $f:\mathcal{X}\rightarrow \Real,x\in \mathcal{X},\mathcal{X} \subseteq \Real^m$, we have, 
\begin{equation}\label{newton}
    \begin{aligned}
          \min_{\Delta x\in \mathcal{X}} \nabla f^\top \Delta x+ \frac{1}{2}\Delta x^\top \nabla_{xx}^2 f\Delta x
    \end{aligned}
\end{equation}
which does not have the notion of learning rate. However, prior work provides strong learning and regret guarantees, even in adversarial cases, for the adjusted Newton method where the Newton term is replaced with its weighted version $\frac{\alpha}{2}\Delta x^\top \nabla_{xx}^2 f\Delta x$~\citep{hazan2007logarithmic}. This scaling allows for different learning updates that adjust how much the update weighs the second term. This is a similar approach taken in \cgo update.

\subsection{Continuous-time version of \cgd and \omda}
In this section, we analyze the continuous-time versions of \cgd, \omda, and \cgo. We show that while, in continuous-time, \cgd and \omda reduce to their \gda and \mda counterparts, \cgo gives rise to a distinct update. 

\paragraph{CGD}
Following the discussion in the introduction, \cgd can be obtained by setting $\alpha=\eta$ in the \cgo update rule. Doing so in Eq.~\eqref{eqn:localgame0} we obtain,
\begin{align}
    \label{eqn:cgd}
    &\Delta x = -\eta \left( I+ \eta^2 \nabla_{xy}^2f \nabla_{yx}^2 f \right)^{-1}  
                \left( \nabla_{x} f +\eta \nabla_{xy}^2f  \nabla_{y} f \right) \\
    &\Delta y = -\eta \left( I +\eta^2 \nabla_{yx}^2f \nabla_{xy}^2 f \right)^{-1}  
                \left( -\nabla_{y} f +\eta \nabla_{yx}^2f  \nabla_{x} f \right),
    \end{align}
For the continuous-time analysis, the learning rate $\eta$ corresponds to the time discretization $\Delta t$ with scaling factor $\beta$, i.e., $\eta = \beta\Delta t$. The ratios of the changes in $x$ := $\Delta x$ and $y$ := $\Delta y$ to $\eta$ then become the time derivative of $x$ and $y$ in the limit $\eta \rightarrow 0$. Ergo, for \cgd update we obtain,
\begin{align}
    \label{eqn:contcgd}
    &\dot x = - \beta\nabla_{x} f\\
    &\dot y = \beta\nabla_{y} f
    \end{align}
Where $\dot x = \frac{\mathrm{d}x}{\mathrm{d}t}, \dot y = \frac{\mathrm{d}y}{\mathrm{d}t}$ are the time derivatives of $x$ and $y$. This is the \textit{same} as the update rule for \gda and the interaction information is lost in continuous-time.

% Now we show that the OMDA algorithm~\citep{mertikopoulos2018optimistic} 
\paragraph{OMDA}
To present the updates for \omda we first define the proximal map,
\begin{equation}
\label{eq:prox}
P_{z}(p)
	= \arg\min_{z'\in \mathcal{X}\times \mathcal{Y}} \{\langle p,z - z'\rangle + \mathcal{B}_h(z', z)\}
\end{equation}
Where $\mathcal{B}_h$ is the Bregmann Divergence with the potential function $h$. The \omda update rule is then given by:
\begin{align}
    \label{omd}
    &z_{n+\frac{1}{2}} = P_{z_n}(-\eta g_n) = \nabla h^{-1}(\nabla h(X_n)-\eta g_n) = z_n-\eta \nabla (\nabla h^{-1})^{\top}g_n+o(\eta g_n)\\
    &z_{n+1} = P_{z_{n}}(-\eta g_{n+\frac{1}{2}}) = \nabla h^{-1}(\nabla h(z_{n})-\eta g_{n+\frac{1}{2}}) = z_n-\eta \nabla (\nabla h^{-1})^{\top}g_{n+\frac{1}{2}}+o(\eta g_{n+\frac{1}{2}})
    \end{align}
Where $g_n,g_{n+\frac{1}{2}}$ are the vector $(\nabla_x f(x,y), -\nabla_y f(x,y))$ evaluated at $z_n,z_{n+1}$ respectively. We now analyze the updates of \omda in continuous-time. In the limit $\eta \rightarrow 0$ we have $\frac{\partial z}{\partial t} =- \beta(\nabla(\nabla h)^{-1}(z))^{\top} \nabla f(z) $ which is the same as the update rule of \mda and the effect of half time stepping vanishes in continuous-time. 

\paragraph{CGO}
Taking the continuous-time limit $\eta \rightarrow 0$ of the \cgo updates Eq.~\eqref{eqn:localgame0} we obtain,
\begin{align}
    \label{eqn:conttime}
    \begin{split}

    &\dot x = - \beta\left( I+ \alpha^2 \nabla_{xy}^2f \nabla_{yx}^2 f \right)^{-1}  
                \left( \nabla_{x} f +\alpha \nabla_{xy}^2f  \nabla_{y} f \right) \\
    &\dot y = -\beta\left( I +\alpha^2 \nabla_{yx}^2f \nabla_{xy}^2 f \right)^{-1}  
                \left( -\nabla_{y} f +\alpha \nabla_{yx}^2f  \nabla_{x} f \right),
        
    \end{split}
\end{align}

which is a distinct update from \gda and the interaction information is preserved in continuous-time. We simulat\footnote{The code for all simulations can be found at \href{https://github.com/AbhijeetiitmVyas/CompetitiveGradientOptim}{Link to code}} the continuous-time setting by using a very small learning rate and observe that while \cgd cycles around the origin (Figure~\eqref{figure:cgdxy}), \cgo is able to take a somewhat direct path to the saddle point solution (Figure \eqref{figure:cgoxy}). This is an encouraging experiment, validating our hypothesis on the importance of \cgo update.

\begin{figure}
\centering
\subcaptionbox{\cgd\label{figure:cgdxy}}
[.36 \textwidth]{\includegraphics[width=0.8\linewidth]{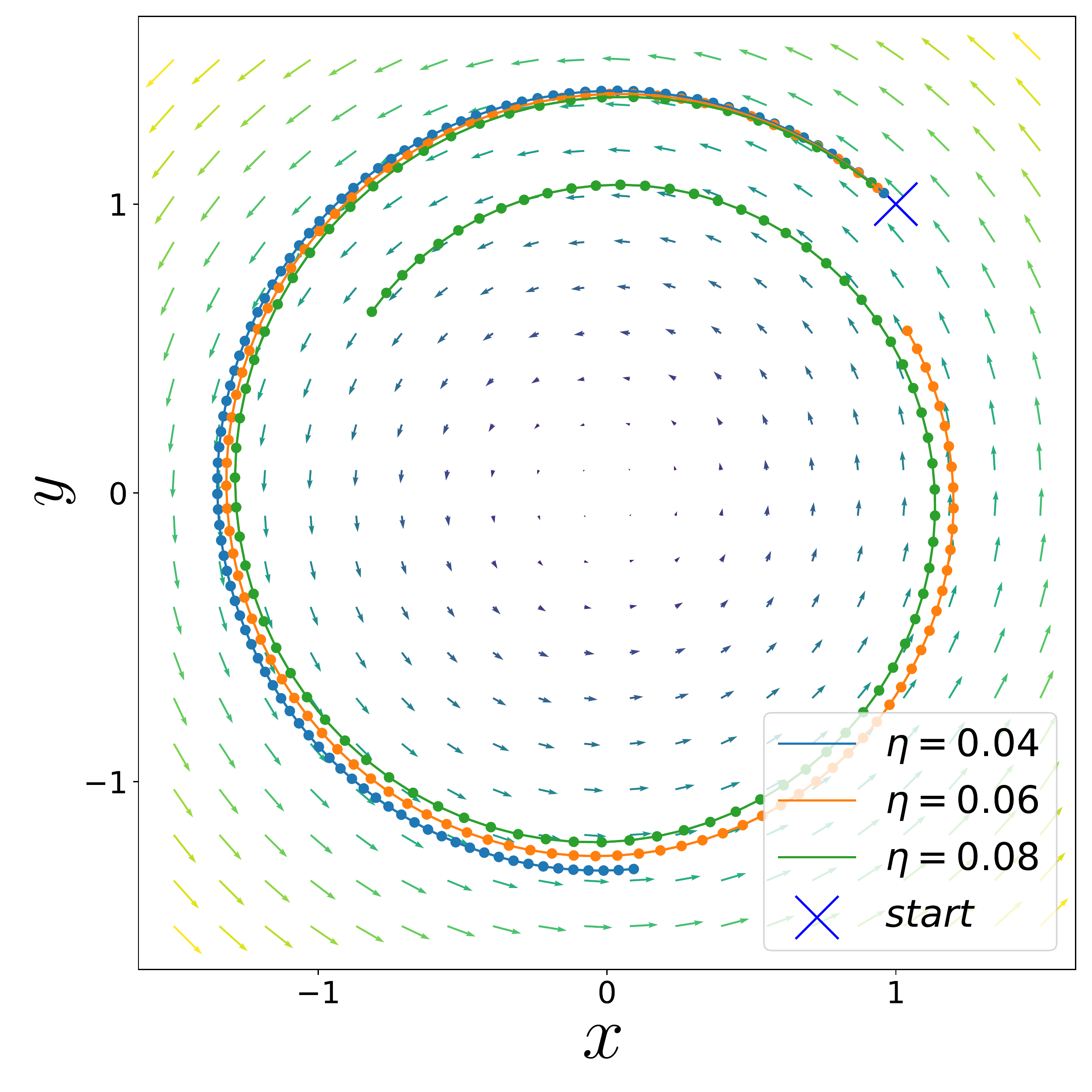}}
\subcaptionbox{\cgo\label{figure:cgoxy}}
[.36 \textwidth]{\includegraphics[width=0.8\linewidth]{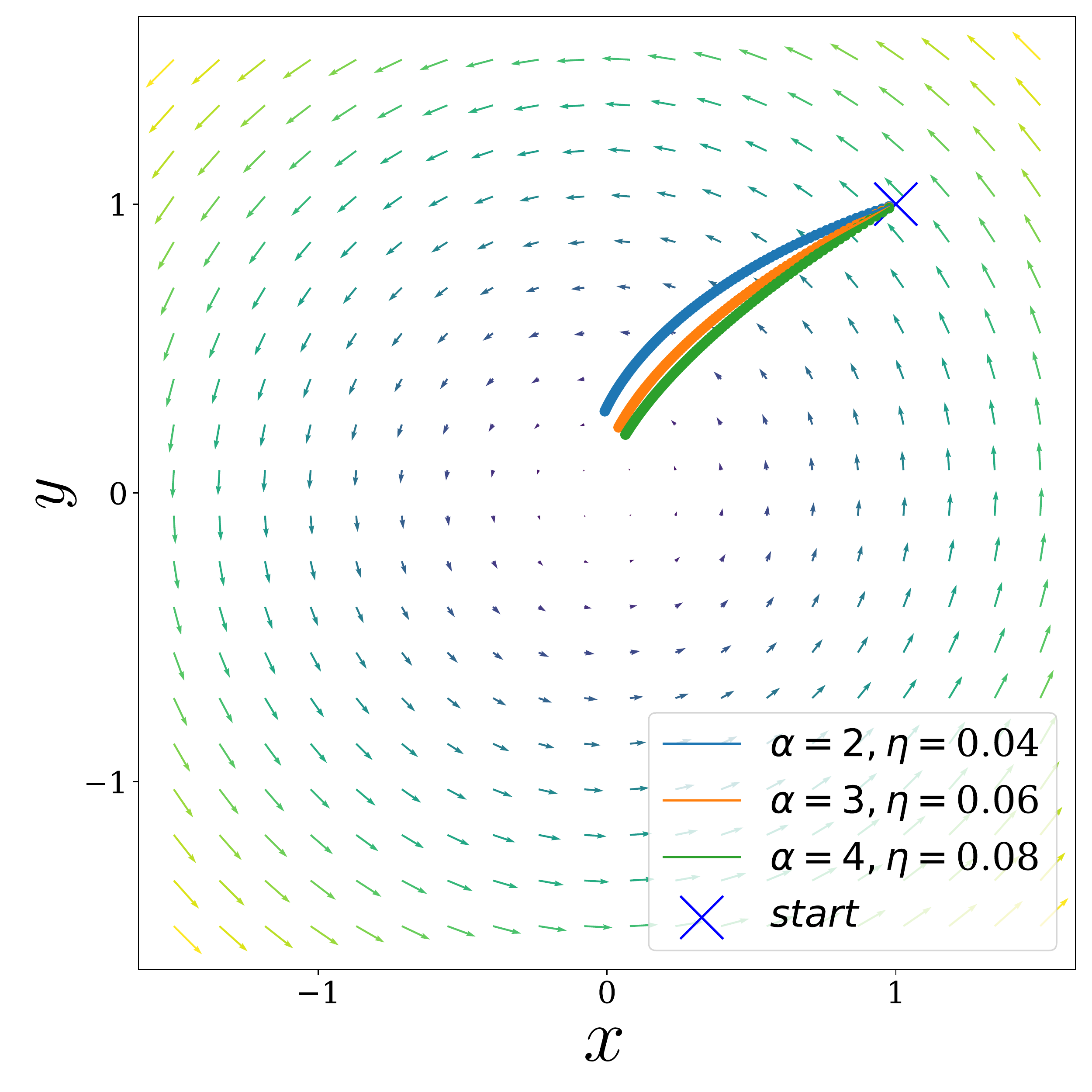}}
\caption{Modeling of the continuous-time regime : \cgd cycles while \cgo takes a direct path}
\end{figure}

\subsection{Families of functions which give rise to $\alpha$-coherent SP}
\label{subseq:example_functions}

The following examples establish a few families of $\alpha$-coherent functions. First we present the important result that all bi-linear games $f=x^\top A y$, are \textit{strictly} $\alpha$-coherent. 

\begin{example}\label{eg1}
All functions of the form $f(x,y) = x^{\top}Ay,A\in \Real^{m\times n}$, give rise to \textit{strictly} $\alpha$-coherent $\min-\max$ SP problems $\forall \alpha>0$ and are null coherent for $\alpha=0$.
\end{example}
\begin{proof}[Proof Sketch]
The origin is the only saddle point of the above function, we evaluate SVI and $\alpha$-SVI at the origin,\\

i) We have $\langle g_0,z \rangle$, $g_0 = (Ay,-A^\top x)$.
Hence, $\langle g_0,z \rangle = x^\top Ay-y^\top A^\top x = 0,\quad\forall ~(x,y) \in \mathcal{X}\times \mathcal{Y}$

ii) Also we have:
\begin{equation}
    \langle g_{\alpha},z \rangle \geq \alpha \lambda_{\min}((I+\alpha^2AA^\top)^{-1} A A^\top) \|x\|^2+\alpha\lambda_{\min}((I+\alpha^2A^\top A)^{-1} A^\top Ay)\|y\|^2>0
\end{equation}

Where the final inequality follows from the fact that $\min(\lambda_{\min}(A^\top A),\lambda_{\min}(AA^\top))>0,\quad\forall A\in \Real^{m\times n}$. See \eqref{eg1a} for a detailed proof.

\end{proof}

We present another family of functions parameterized by a scalar $k$. For $k\geq 0$ the functions exhibit a $\min-\max$ saddle point at the origin (a $\max-\min$ saddle point is at ($\infty,-\infty$)), while for $k<0$ the function has a $\max-\min$ saddle point at the origin (a $\min-\max$ saddle point is at ($-\infty,\infty$)). For both cases, origin satisfies the $\alpha$-variational inequalities for $\alpha\geq-k$, strictly for $\alpha>k$.

\begin{example}\label{eg2}
The family of functions $f_{k}(x,y) = \frac{k}{2}(x^2-y^2)+xy$ with $k\geq0$ gives rise to an $\alpha$-coherent $\min-\max$ SP problem for $\alpha = -k$ and a \textit{strictly} $\alpha$-coherent $\min-\max$ SP problem $\forall \alpha >- k$. For $k<0$, it gives rise to an $\alpha$-coherent $\max-\min$ SP problem for $\alpha = -k$ and a \textit{strictly} $\alpha$-coherent $\max-\min$ SP problem $\forall \alpha >- k$
\end{example}
\begin{proof}[Proof Sketch]
We evaluate the variational inequalities at the origin,\\
For $g_{0}$ we have:

\begin{equation}
    \begin{aligned}
        \langle g_0,z \rangle =~ &x(kx+y)-y(-ky+x)
        = k x^2 + k y^2 > 0
    \end{aligned}
\end{equation}
 
For $g_{\alpha}$ we have:

\begin{equation}
    \begin{aligned}
        \langle g_{\alpha},z \rangle = \frac{k+\alpha}{1+\alpha^2}(x^2+y^2)>0,\quad\forall \alpha>-k 
    \end{aligned}
\end{equation}

\end{proof}

\section{Convergence results of \cgo and the \ocgo algorithm}\label{Sec:SV}

In this section we present the convergence results of our \cgo algorithm. We first consider the convergence to stationary points and present the conditions and rate for the continuous-time and discrete-time regimes. Then, we state the convergence results of the \cgo algorithm to \textit{strictly} $\alpha$-coherent saddle points. Then, we introduce the \ocgo updates and present its rate of convergence to $\alpha$-coherent saddle points. Finally we showcase the working of \cgo and \ocgo by simulating them on a few benchmark functions from the families presented in subsection~\eqref{subseq:example_functions}. 

\subsection{Convergence analysis in continuous-time}

We present our first result for convergence of \cgo in continuous-time. We present the proof sketch and refer the readers to \eqref{contgda} for the complete proof. To highlight the difference in convergence rate and condition of \cgo from \gda we also derive the conditions for convergence of \gda using a Lyapunov-style analysis. By carefully choosing the parameter $\alpha$ we show that we can accommodate arbitrary deviation from the strictly convex-concave condition which is required for the convergence of continuous-time \gda.

\begin{theorem}\label{contheorem}
Continuous-time \cgo runs on a twice differentiable function $f$ with parameters $\alpha,\beta$ on functions satisfying $\lambda>0$ where
$$\lambda := \beta \min(2 \underline{\lambda_{xx}}-2  \alpha\overline{\lambda_{xx}}^2+ c\frac{\underline{\lambda_{xy}}}{1+\alpha^2\underline{\lambda_{xy}}
},-2 \overline{\lambda_{yy}}-2  \alpha\overline{\lambda_{yy}}^2+ c\frac{\underline{\lambda_{yx}}}{1+\alpha^2\underline{\lambda_{yx}}
})$$
converges exponentially to a stationary point with rate $\lambda$. Where $c=\beta(\alpha-2\alpha^2\overline{\lambda_1}-2\alpha^3\overline{\lambda_2}^2)$.

\end{theorem}
\begin{proof}[[Proof Sketch]]
We choose $\|g_0\|^2$ to be our Lyapnuov function where, $$g_0:=(\nabla_x(f(x,y),-\nabla_y f(x,y))$$
Evaluating time derivative of $\|g\|^2$ we obtain :
\begin{equation} \label{dhh}
\begin{aligned}
\frac{\mathrm{d}\|g_0\|^2}{\mathrm{d}t}= 2 g_0^{\top} \dot g_0 &=2\begin{bmatrix}
\nabla_x^\top& -\nabla_y ^\top
\end{bmatrix}\begin{bmatrix}
\nabla_{xx} &  \nabla_{xy}\\
-\nabla_{xy}^{\top} & -\nabla_{yy}
\end{bmatrix}
\begin{bmatrix}
\dot x \\
\dot y
\end{bmatrix}
 \\
 &=2 \dot x^{\top} \nabla_{xx} \nabla_x +  2 \nabla_x^{\top} \nabla_{xy} \dot y + 2 \dot y^{\top} \nabla_{yy} \nabla_y +  2 \nabla_y^{\top} \nabla_{xy}^{\top} \dot x\\ 
\end{aligned}
\end{equation}

By plugging in \cgo updates and manipulating we show:
$$ \frac{\mathrm{d}\|g_0\|^2}{\mathrm{d}t} \leq -\lambda \|g_0\|^2$$ where $\lambda$ is as stated in the Theorem. The detailed proof is in \eqref{conttimecgo}.

To compare, we also derive the conditions for \gda Eq.~\eqref{eqn:contcgd} in continuous-time in \eqref{contgda}. We obtain,

\begin{equation}
    \begin{aligned}
    \frac{\mathrm{d}\|g_0\|^2}{\mathrm{d}t} &\le -\|g_0\|^2 \min (\lambda_{\min}( 2 \beta\nabla_{xx}), \lambda_{\min}(-2\beta \nabla_{yy}))\\&=-2\beta\|g_0\|^2  \min (\underline{\lambda_{xx}},- \overline{\lambda_{yy}})
    \end{aligned}
\end{equation}
For convergence, we require $\min (\underline{\lambda_{xx}},- \overline{\lambda_{yy}})\geq0$ which is the convex-concave condition.
\end{proof}

This Theorem implies that in the present of interaction, particularly, when $\frac{\underline{\lambda_{xy}}}{1+\alpha^2\underline{\lambda_{xy}}
}$ and $\frac{\underline{\lambda_{yx}}}{1+\alpha^2\underline{\lambda_{yx}}
}$ are positive, it allows to break free from the convex-concave condition by appropriately setting $\alpha$.

We set $\alpha$ such that $\overline{\lambda_{xx}} \leq \frac{1}{5\alpha};\underline{\lambda_{xx}} \geq -\frac{1}{5\alpha};\underline{\lambda_{yx}},\underline{\lambda_{xy}}\leq \frac{K}{\alpha^2};\underline{\lambda_{yy}} \geq -\frac{1}{5\alpha};\overline{\lambda_{yy}} \leq \frac{1}{5\alpha};K\gg1$ which implies $\overline{\lambda_{1}},\overline{\lambda_{2}}<\frac{1}{5\alpha}$ and we obtain $\lambda_{min} \geq \frac{1}{50\alpha}$. This shows that continuous-time \cgo allows \textbf{arbitrary deviation} of $\underline{\lambda_{xx}},\overline{\lambda_{yy}}$ (from the convex-concave condition i.e. $\underline{\lambda_{xx}}\geq 0,\overline{\lambda_{yy}}\leq 0$), if
$\underline{\lambda_{yx}},\underline{\lambda_{xy}}$ are proportional to the square of the deviation of the pure terms.

\subsection{Convergence analysis in discrete-time}

\paragraph{Convergence to stationary points}

We derive the conditions required for \cgo to converge to a stationary point and show that large singular values of the interaction terms help in convergence. By tuning the hyperparameters we are able to control the influence of this interactive term and obtain faster convergence.
\begin{theorem}\label{disctheorem}
\cgo with parameters $\alpha$ and $\eta$ when initialized in the neighborhood of a first-order stationary point $z^*$ on a Lipschitz-continuous and thrice differentiable function $f$ that has Lipschitz-continuous gradients and Hessian and $1\geq\lambda>0$ where,  

$$
 \lambda :=  \min(\eta(2\underline{\lambda_{xx}}-2\frac{10\eta+8\alpha}{\eta}\overline{\lambda_{xx}}^2)+c\frac{\underline{\lambda_{xy}}}{1+
\alpha^2\underline{\lambda_{xy}}
}),\\-\eta(2\underline{\lambda_{yy}}+2\frac{10\eta+8\alpha}{\eta}\overline{\lambda_{yy}}^2)+c\frac{\underline{\lambda_{yx}}}{1+
\alpha^2\underline{\lambda_{yx}}
}))   
$$ converges exponentially to $z^*$ with rate $r(\lambda) = 1-\lambda$. Where c is a polynomial function of $\eta,\alpha,\overline{\lambda_1},\overline{\lambda_2}$.
\end{theorem}
Similar to the continuous-time setting, the terms $\frac{\underline{\lambda_{xy}}}{1+
\alpha^2\underline{\lambda_{xy}}
}$, $\frac{\underline{\lambda_{yx}}}{1+
\alpha^2\underline{\lambda_{yx}}
}$ are non-negative and appropriately choosing $\alpha$ and $\eta$ allows us to tune $c$ and obtain convergence for functions not satisfying the convex-concave condition.
\cgd restricts the flexibility of $c$ by choosing $\alpha=\eta$ and \cgo utilizes this extra degree of freedom granted by $\alpha$ to allow convergence for a larger class of functions. The proof of the above Theorem is provided in appendix~\eqref{disctimecgo}, for completeness we also we provide the analysis of discrete time \gda in the appendix~\eqref{discretegdasection}.

\paragraph{Convergence to \textit{strictly} $\alpha$-coherent saddle points}
Now we discuss the convergence properties of \cgo for the class of \textit{strictly} $\alpha$-coherent functions, the detailed proof is in the appendix, see \eqref{strictalpha}.

\begin{theorem}\label{strictalphatheorem}
Suppose that a Lipschitz-continuous function $f$ has Lipschitz-continuous gradients
and Hessian and gives rise to a strictly $\alpha$-coherent SP. If \cgo is run with perfect gradient and competitive hessian oracles and parameter $\alpha$ and parameter sequence $\{\eta_n\}$ such that $\sum_{1}^{\infty} \eta_n^{2} < \infty$ and $\sum_{1}^{\infty} \eta_n = \infty$, then the sequence of \cgd iterates $\{z_n\}$, converges to a solution of SP. 
\end{theorem}

\paragraph{Convergence to $\alpha$-coherent saddle points}
For convergence to the saddle points for $\alpha$-coherent functions which are not strictly $\alpha$-coherent, we propose the optimistic \cgo algorithm.

\textbf{Optimistic \cgo}
The update rule is given by:
\begin{align*}
    &z_{n+\frac{1}{2}} =z_n-\eta g_{\alpha,n}\\
    &z_{n+1} = z_n-\eta g_{\alpha,n+\frac{1}{2}}
    \end{align*}
where $g_{\alpha,n},$ is as in \eqref{eqn:localgame0} and $\eta$ is the learning rate.

\begin{theorem}\label{alphatheorem}
Suppose that a $L$-Lipschitz-continuous function $f$ that has $L'$-Lipschitz-continuous gradients and $L_{xy}$ Lipschitz-continuous Hessian gives rise to an $\alpha$-coherent SP. If \ocgo is run with parameter $\alpha$ and parameter sequence $\{\eta_n\}$ such that,
\begin{itemize}
    \item $0<\alpha^2 < \frac{\sqrt{L'^4+4L_{xy}^2L^2}-L'^2}{2L_{xy}^2L^2}$
    \item $0< \eta_n < \frac{\sqrt{\alpha^2 L^2 L_{xy}^2 + L'^2-2\alpha^4L^2 L'^2 L_{xy}^2 - \alpha^2 L'^4 } - \alpha^3 L_{0}^2 L_{xy}^2}{\alpha^2 L^2 L_{xy}^2 + L'^2},\quad\forall n$
\end{itemize}
then the sequence of iterates $z_n$ converges to $z^*$ where $z^*:=(x^*,y^*) \in\mathcal{X\times Y}$ is a saddle point. Moreover, the \ocgo converges with the rate of $\frac{1}{n}$, i.e., for the average of the gradients, we have,
% $\frac{1}{n}\sum_{k=1}^n \|g_{\alpha,k}\|^2$, is of the order $O(\frac{1}{n})$
$$\frac{1}{n}\sum_{k=1}^n \|g_{\alpha,k}\|^2=O\left(\frac{1}{n}\right)$$
\end{theorem}
The details proof of the above Theorem is provided in \eqref{alphatheorem}.

\subsection{Simulation of \cgo and \ocgo on families from section \eqref{Section4}}

% ###rewrite
\begin{figure}[t]
\centering
\subcaptionbox{\cgo \label{bilineara}}
[.24 \textwidth]{\includegraphics[width=\linewidth]{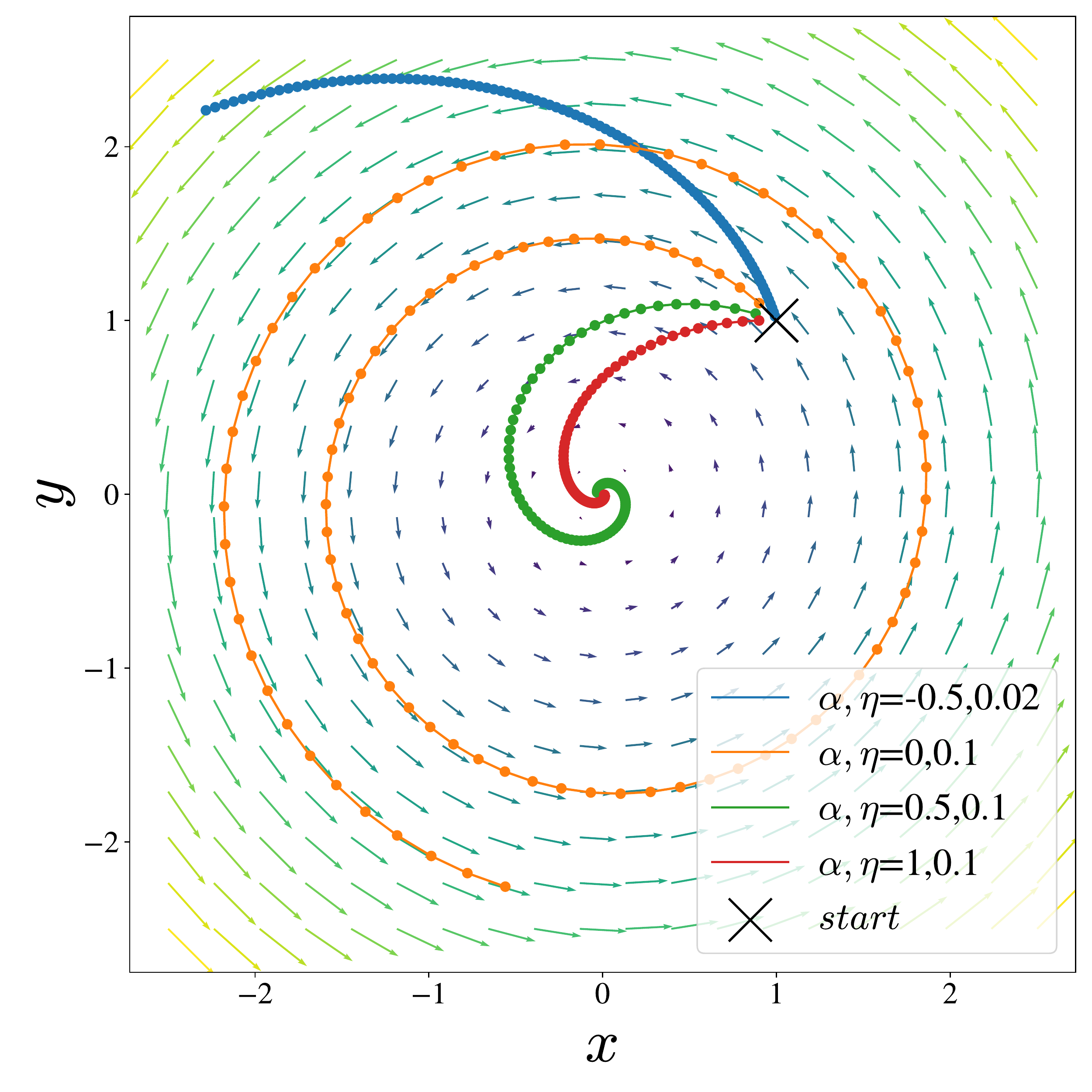}}
\subcaptionbox{optimistic \cgo\label{bilinearb}}
[.24 \textwidth]{\includegraphics[width=\linewidth]{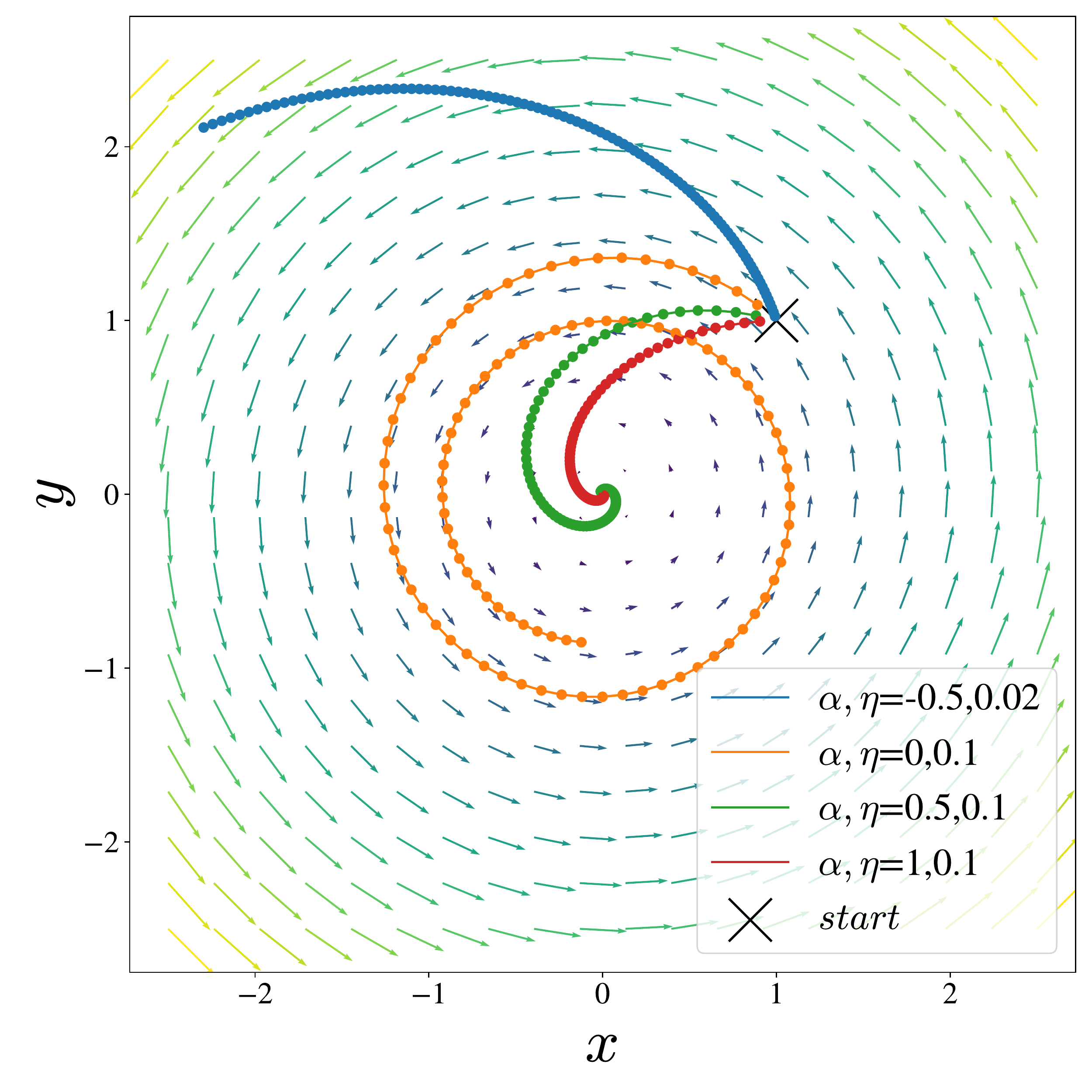}}
\subcaptionbox{\cgo\label{bilinearc}}
[.245 \textwidth]{\includegraphics[width=1.1\linewidth]{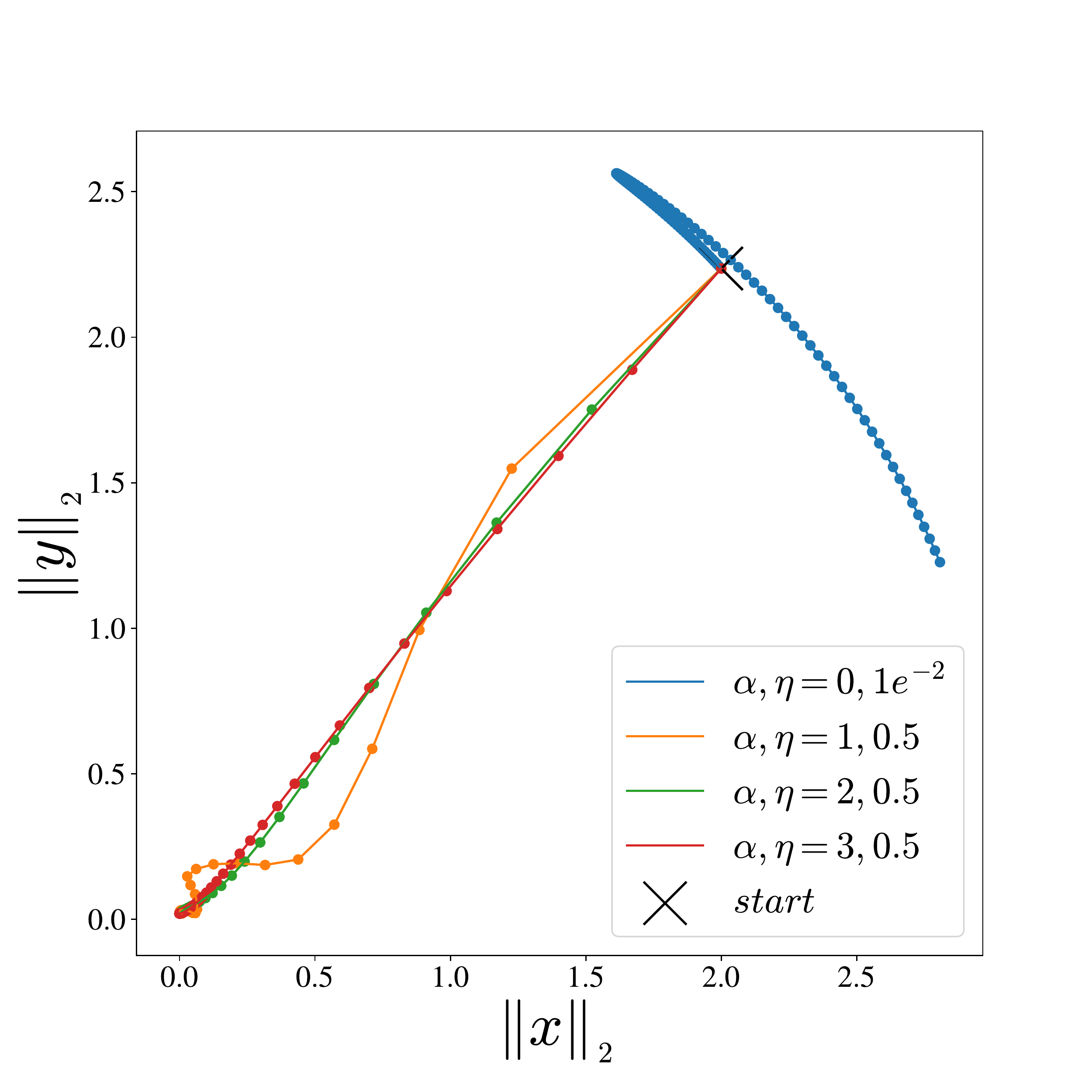}}
\subcaptionbox{optimistic \cgo\label{bilineard}}
[.245 \textwidth]{\includegraphics[width=1.1\linewidth]{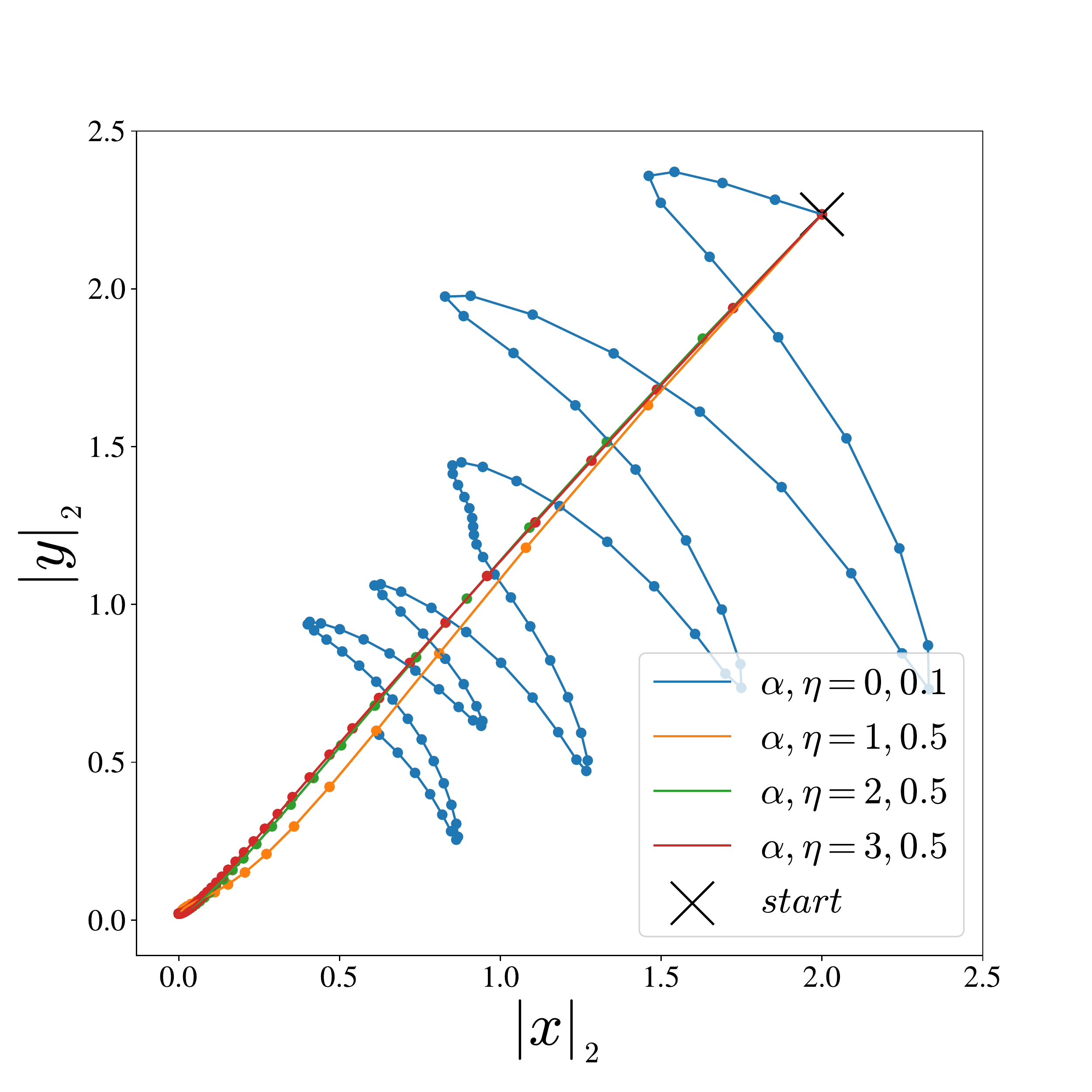}}
\caption{\cgo and optimistic \cgo on bilinear functions $f(x,y)=xy, (x,y)\in \Real^2$ : (a,b) and $f(x,y) = x^\top A y, x\in \Real^4,y\in \Real^5$ : (c,d) for 100 iterations.}
\end{figure}
We now evaluate the performance of \cgo and \ocgo on families discussed in examples \eqref{eg1} and \eqref{eg2}. We first consider a function $f(x,y)=x^{\top}Ay,A\in \Real^{4\times 5}, x\in \Real^4, y\in \Real^5$ . We sample all the entries of A independently from a standard Gaussian, $A=(a_{ij}),a_{ij}\sim \N (0,1)$. We consider the plot of the $L^2$ norm of $x$ vs. that of $y$, since the only saddle point is the origin, the desired solution is $\|x\|_2,\|y\|_2\rightarrow 0$. We plot the iterates of \cgo and \ocgo for different $\alpha$, and observe that \ocgo converges to the saddle point for $\alpha\geq 0$ (at a very slow rate for $\alpha=0$) while \cgo does so for $\alpha >0$. The results at $\alpha=0$ are that of \gda and optimistic \gda. We see similar results for the case where $A$ is the scalar $1$, i.e. $f(x,y)=xy$. This is in accordance with the analysis in example \eqref{eg1}.

We then proceed to perform experiments on the family $f(x,y) = \frac{k}{2}(x^2-y^2)+xy$ for $k=2,-2$. For both values of $k$ we see that \ocgo converges to the origin for $\alpha\geq -k$ and \cgo converges for $\alpha > -k$, following the analysis in example \eqref{eg2}. For $k=2$ the origin is a $\min-\max$ saddle point, while for $k=-2$ it is a $\max-\min$ saddle point. The gradient field, $g_0:=(\nabla_x(f(x,y),-\nabla_y f(x,y))$ is plotted for all 2-dimension cases. All algorithms are run for 100 iterations.

\begin{figure}[t]
\centering
\subcaptionbox{\cgo\label{coherentfama}}
[.24 \textwidth]{\includegraphics[width=\linewidth]{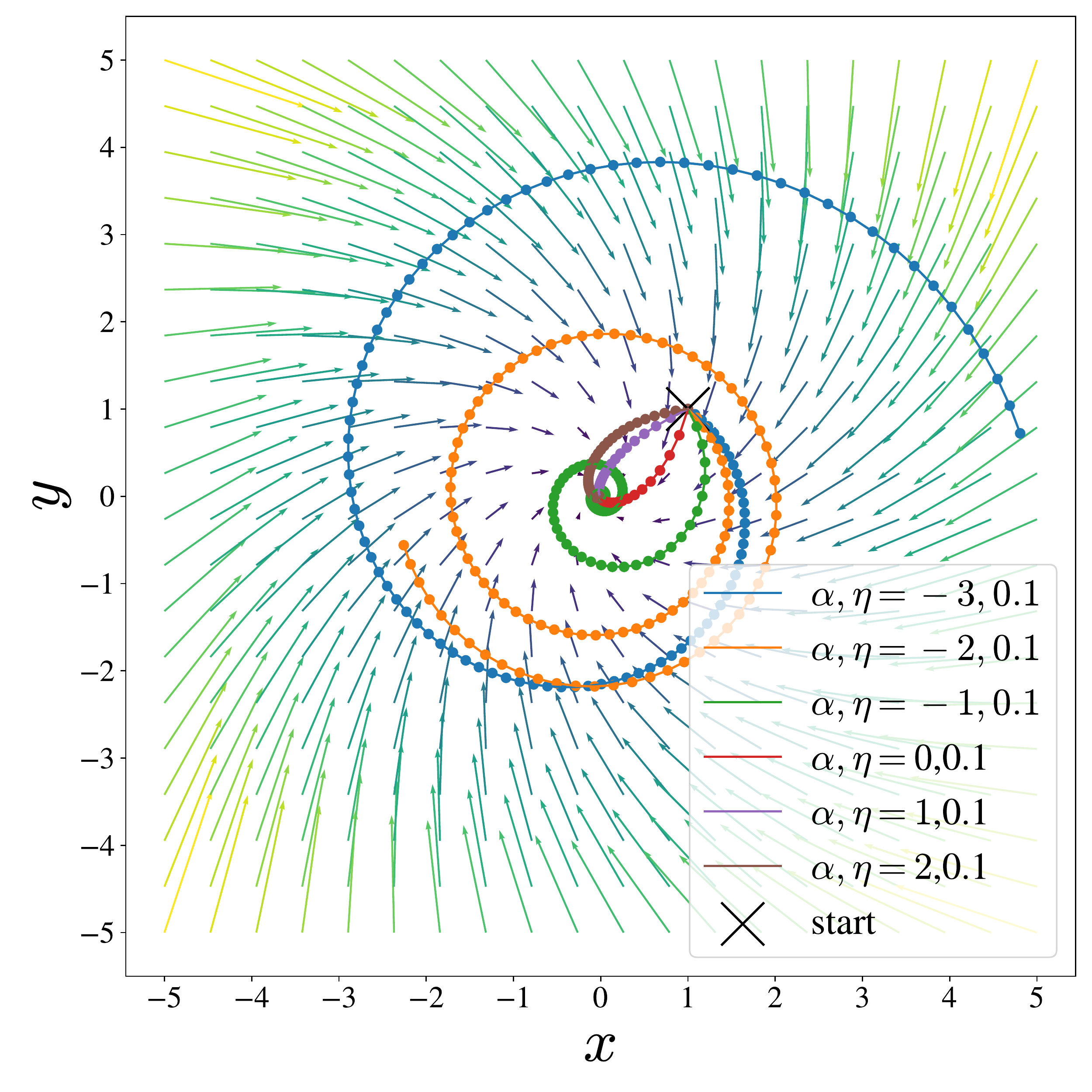}}
\subcaptionbox{optimistic \cgo\label{coherentfamb}}
[.24 \textwidth]{\includegraphics[width=\linewidth]{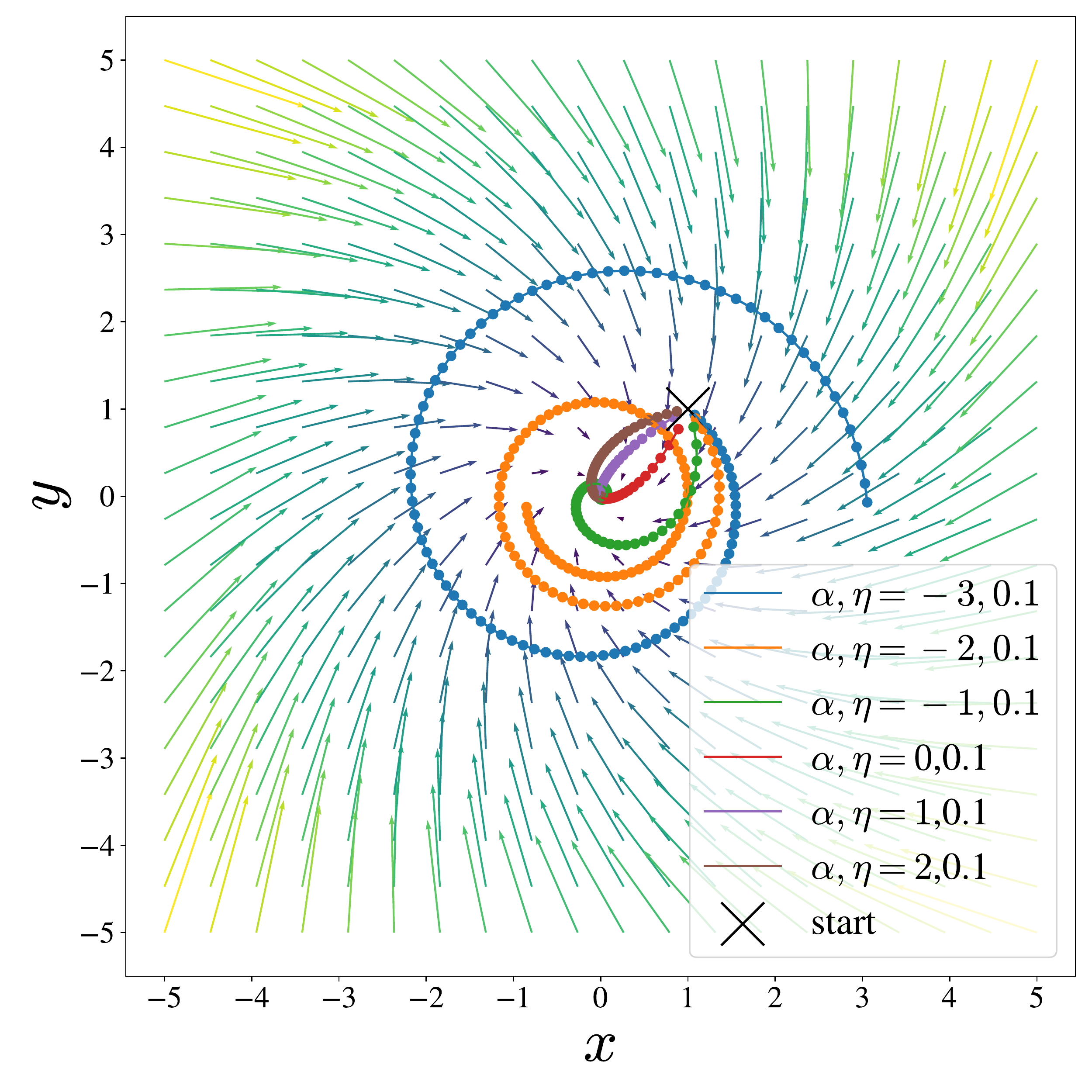}}
\subcaptionbox{\cgo\label{coherentfamc}}
[.24 \textwidth]{\includegraphics[width=\linewidth]{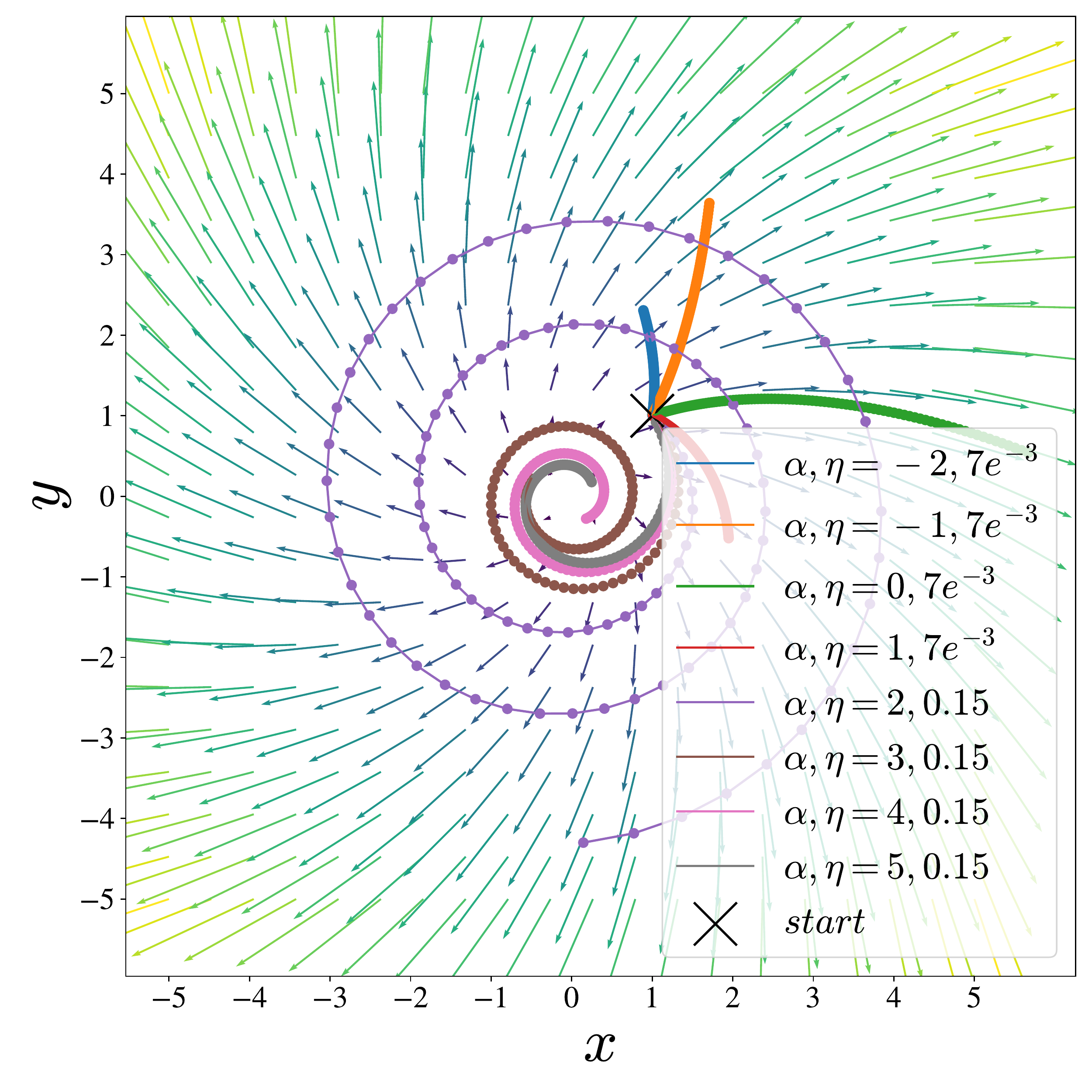}}
\subcaptionbox{optimistic \cgo\label{coherentfamd}}
[.24 \textwidth]{\includegraphics[width=\linewidth]{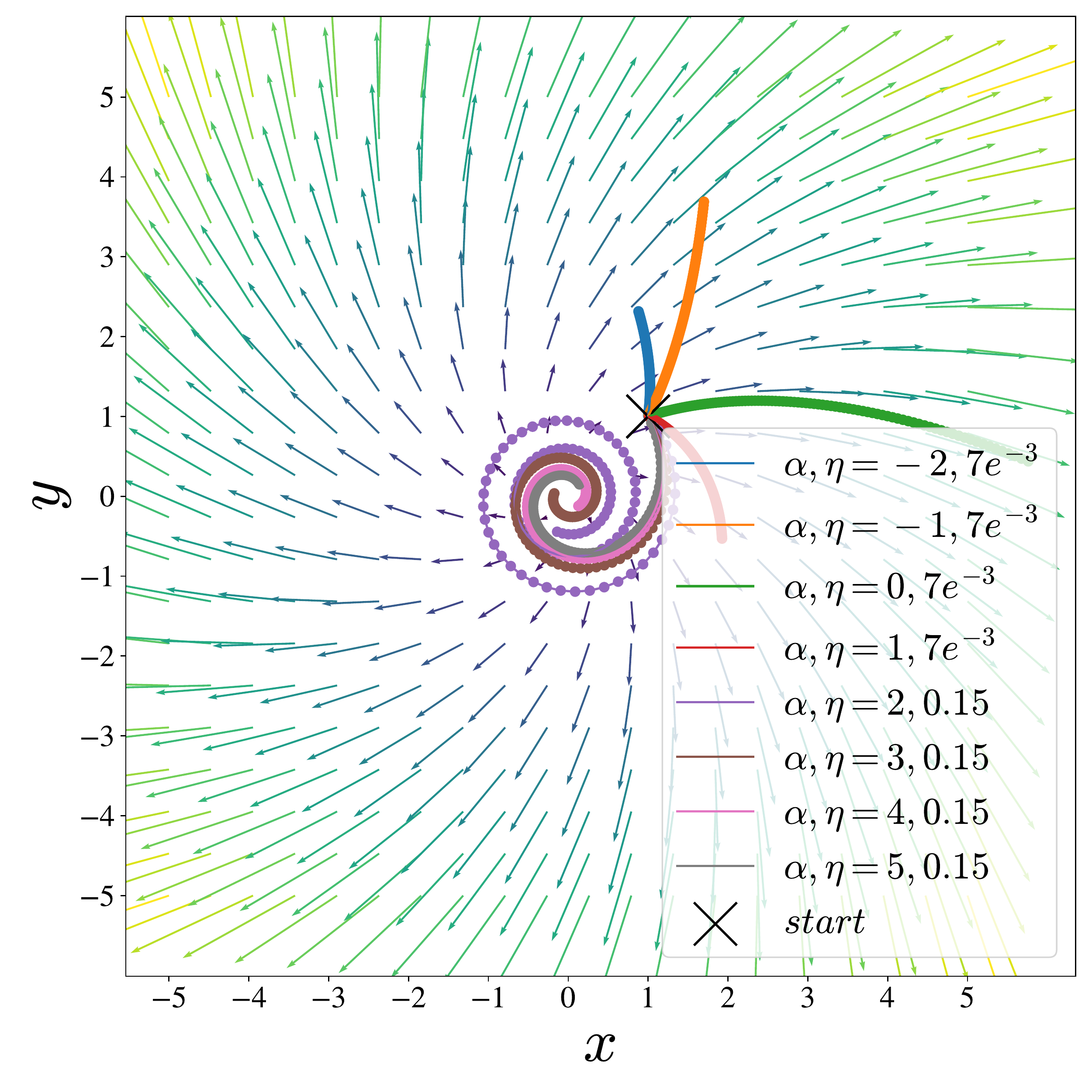}}
\caption{\cgo and optimistic \cgo on functions from the family $f(x,y)=\frac{k}{2}(x^2-y^2)-xy$. k=2 : (a,b) and k=-2 : (c,d) for 100 iterations.}
\end{figure}

\section{Conclusion}
We propose the \cgo algorithm which allows us to control the effect of the cross derivative term in \cgd. This increases the size of the class of functions for which the algorithm converges. In the realm of continuous-time we observe that \cgd reduces to \gda, \cgo on the other hand gives rise to a distinct update which allows for a margin of deviation from the strictly convex-concave convergence condition of \gda. Furthermore, we generalize the definition of coherent saddle point problems defined in \cite{mertikopoulos2018optimistic} to $\alpha$-coherent saddle points for which we prove convergence of Optimistic \cgo and of \cgo in the strict version of $\alpha$-coherence, we show order $O(\frac{1}{n})$ rate of the average gradients for \cgo. Finally we present a short experiment study on some $\alpha$-coherent functions. Future work would involve using \cgo in various machine learning tasks such as GANs, competitive reinforcement learning (RL) and adversarial machine learning.

\newpage

\bibliography{Arxiv}
\bibliographystyle{plainnat}

\newpage

\newpage
% \appendix
\addcontentsline{toc}{section}{Appendix}
\part{Appendices} % Start the appendix part
\parttoc % Insert the appendix TOC
\newpage

In this section we present proofs for statements pertaining to example \eqref{eg1} and \eqref{eg2}.

\section{Proof of example~\eqref{eg1}}\label{eg1a}
For clarity we restate the statement of the example~\eqref{eg1}. All functions of the form $\Delta x^{\top}Ay$ are strictly $\alpha$-coherent $\forall \alpha>0$ and are null coherent for $\alpha=0$.

\begin{proof}[Proof of example~\eqref{eg1}]

In order to show the above mentioned statement, we first note that the origin is the only saddle point of this function. We now evaluate $\langle g_0,z \rangle$, where $g_0 = (Ay,-A^\top x)$. Hence, $\forall ~(x,y) \in \mathcal{X}\times \mathcal{Y}$. we have,
$$\langle g_0,z \rangle = x^\top Ay-y^\top A^\top x = 0.$$
ergo, the function $\Delta x^{\top}Ay$ is null-coherent.

Similarly, we evaluate the $\alpha$-SVI. We observe for the function $x^\top A y$,
$$g_{\alpha} = ((I+\alpha^2AA^\top)^{-1}(Ay+\alpha A A^\top x),(I+\alpha^2A^\top A)^{-1}(-A^\top x+\alpha A^\top A y))$$
Hence for $\langle g_{\alpha},z \rangle $ we have,
\begin{align} \label{eq:eg1_alpha}
    \langle g_{\alpha},z \rangle = x^\top(I+\alpha^2AA^\top)^{-1}(Ay+\alpha A A^\top x)+y^\top (I+\alpha^2A^\top A)^{-1}(-A^\top x+\alpha A^\top Ay)
\end{align}
We further observe that, following the statement of Lemma~\eqref{lemma:inverse-transpose}, we have
$$(I+\alpha^2AA^\top)^{-1}A = A(I+\alpha^2A^\top A)^{-1},$$
and therefore, incorporating it in to the Eq.~\eqref{eq:eg1_alpha}, we have, 
$$ x^\top(I+\alpha^2AA^\top)^{-1}Ay = x^\top A(I+\alpha^2A^\top A)^{-1}y = y^\top (I+\alpha^2A^\top A)^{-1}A^\top x.$$
Thus, for $\langle g_{\alpha},z \rangle $ we have,
\begin{align*}
    \langle g_{\alpha},z \rangle = &x^\top(I+\alpha^2AA^\top)^{-1}\alpha A A^\top x+y^\top (I+\alpha^2A^\top A)^{-1}\alpha A^\top Ay\\
     \geq &\alpha \lambda_{min}((I+\alpha^2AA^\top)^{-1} A A^\top) \|\Delta x\|^2+\alpha\lambda_{min}((I+\alpha^2A^\top A)^{-1} A^\top Ay)\|\Delta y\|^2
\end{align*}
Finally, observing that $\min(\lambda_{min}(A^\top A),\lambda_{min}(AA^\top))\geq0$ for any $A$, and following the statement in the Lemma~\eqref{lemma:eigen-properties1} we also have,  
$$\alpha \lambda_{min}((I+\alpha^2AA^\top)^{-1} A A^\top) \|\Delta x\|^2+\alpha\lambda_{min}((I+\alpha^2A^\top A)^{-1} A^\top Ay)\|\Delta y\|^2>0,\quad\forall \alpha>0,$$
and hence $\langle g_{\alpha},z \rangle >0,\quad\forall \alpha>0$. Ergo, the function $\Delta x^{\top}Ay$ is \textit{strictly} $\alpha$ coherent.
\end{proof}

\section{Proof of example~\eqref{eg2}}\label{eg2a}
Now, we restate the statement of the example \eqref{eg2}. The family ofunctions $f_{k}(x,y) = \frac{k}{2}(x^2-y^2)+xy$ for $k\geq 0$ gives rise to
\begin{itemize}
    \item  $\min-\max$ $\alpha$-coherent SP problem when $\alpha = -k$,
    \item $\min-\max$ \textit{strictly} $\alpha$-coherent SP problem when $ \alpha >- k$.
\end{itemize}

and for $k<0$ the family gives rise to,
\begin{itemize}
    \item $\max-\min$ $\alpha$-coherent SP problem when $\alpha = -k$,
    \item $\max-\min$ \textit{strictly} $\alpha$-coherent SP problem when $ \alpha >- k$.
\end{itemize}

\begin{proof}[Proof of example~\eqref{eg2}]

We first note that the origin is the only saddle point of the above family. Further, the origin is a $\min-\max$ saddle point when $k\geq 0$ and a $\max-\min$ saddle point when $k<0$.

For this family we evaluate $\langle g_{\alpha},z \rangle$,
\begin{align*}
        \langle g_{\alpha},z \rangle = &x((1+\alpha^2)^{-1}(k x-y-\alpha (-x-k y)))+y((1+\alpha^2)^{-1}(x+k y-\alpha(k x-y)))\\
        = & (1+\alpha^2)^{-1}( k x^2-xy+\alpha x^2 +\alpha k xy+xy+k y^2 -\alpha kxy+\alpha y^2) \end{align*}
\end{proof}

Simplifying this expression for $\alpha>-k $ we obtain,
\begin{align*}
        \langle g_{\alpha},z \rangle = \frac{k+\alpha}{1+\alpha^2}(x^2+y^2)>0,~\forall \alpha>-k 
\end{align*}

Ergo, the above mentioned function class is \textit{strictly} $\alpha$-coherent when $\alpha>-k$. Furthermore, when $\alpha=-k$
we have $\langle g_{\alpha},z \rangle=0$, ergo the class is null $\alpha$-coherent for $\alpha=-k$.

% \part{Proofs for continuous time}
% % \faketableofcontents
% \parttoc 

\section{Continuous time \gda}\label{contgda}
In this section, we state the update rule for \gda and derive sufficient convergence conditions using Lyapunov analysis. The update rule of \gda is computed through the following optimization problem,

\begin{align}
    \begin{split}
    \label{eqn:localgame}
    \min_{\delta x \in \Reals^m} \delta x^{\top} \nabla_x f  &+\delta y^{\top} \nabla_y f + \frac{1}{2\eta} \delta x^{\top}\delta x \\
    \max_{\delta y \in \Reals^n}\delta y^{\top} \nabla_y f &+ \delta x^{\top} \nabla_x f  - \frac{1}{2\eta}\delta y^{\top}\delta y.
    \end{split}
\end{align}

Which gives the following closed form update,

\begin{align} \label{cont1}
\begin{bmatrix}
 \Delta x\\
 \Delta y
\end{bmatrix} &= -\eta \begin{bmatrix}
\nabla_x f \\
-\nabla_y f 
\end{bmatrix} 
\end{align}
where $\eta$ is the learning rate. 
Taking the limit $\eta \rightarrow 0$ and scaling the flow of time with $\beta$ we get the continuous time dynamics as follows,

\begin{align} \label{cont2}
\begin{bmatrix}
\dot x\\
\dot y
\end{bmatrix} &= -\beta \begin{bmatrix} 
\nabla_x f \\
-\nabla_y f 
\end{bmatrix}  = -\beta  g_0
\end{align}

where $g_0=\begin{bmatrix}
\nabla_x f\\
-\nabla_y f 
\end{bmatrix}$ is the concatenation of the gradients. 
Furthermore, for the second order curvature of this dynamics, i.e., the gradient of $g_0$, we have,
\begin{align}\label{g0dot}
\dot g_0=\begin{bmatrix}
\nabla_{xx}^2 f &  \nabla_{xy}^2 f\\
-\nabla_{yx}^2 f & -\nabla_{yy}^2 f
\end{bmatrix}
\begin{bmatrix}
\dot x \\
\dot y
\end{bmatrix}
\end{align}
For the Lyapunov analysis, we now choose $\|g_0\|^2$ as our Lyapunov function and evaluate its time-derivative, i.e., 
\begin{align*}
\dot{\| g_0\|^2}=\frac{\mathrm{d}\|g_0\|^2}{\mathrm{d}t}= 2 g_0^{\top} \dot g_0 &=2\begin{bmatrix}
\nabla_x f^\top& -\nabla_y f ^\top
\end{bmatrix}\begin{bmatrix}
\nabla_{xx}^2 f &  \nabla_{xy}^2\\
-\nabla_{yx}^2 f & -\nabla_{yy}^2 f
\end{bmatrix}
\begin{bmatrix}
\dot x \\
\dot y
\end{bmatrix}
     \\
     &=2 \dot x^{\top} \nabla_{xx}^2 f \nabla_x f +  2 \nabla_x f^{\top} \nabla_{xy}^2 f \dot y 
     + 2 \dot y^{\top} \nabla_{yy}^2 f \nabla_y f +  2 \nabla_y f^{\top} \nabla_{yx}^2 f \dot x
\end{align*}
     
Using the update rule of \gda, i.e., Eq.~\eqref{cont2}, we substitute $\dot x$ and $\dot y$ in the above equation and have, 
\begin{align}\label{gdot} 
 \dot{ \| g_0\|^2 }&= -2\beta \nabla_x f^{\top} \nabla_{xx}^2 f \nabla_x f +  2\beta\nabla_y f^{\top}  \nabla_{yy}^2 f \nabla_y f\nonumber\\
 &\qquad\qquad\qquad-2\beta \nabla_x f^{\top} \nabla_{xy}^2 f \nabla_y f +  2\beta\nabla_y f^{\top}  \nabla_{yx}^2 f \nabla_x f \nonumber\\
& = -   2\beta \nabla_x f^{\top} \nabla_{xx}^2 f \nabla_x f -  (-2\beta\nabla_y f^{\top}  \nabla_{yy}^2 f \nabla_y f) 
\end{align}

For the right hand side, we know,
\begin{align*}
    %   \lambda_{max}(2\beta\nabla_{xx}^2 f)\|\nabla_x f\|^2+\lambda_{max}(-2\beta\nabla_{yy}^2 f)\|\nabla_y f\|^2 \geq
    2\beta \nabla_x f^{\top} \nabla_{xx}^2 f \nabla_x f +  (-2\beta\nabla_y f^{\top}  \nabla_{yy}^2 f \nabla_y f) \geq       \lambda_{min}(2\beta\nabla_{xx}^2 f)\|\nabla_x f\|^2+\lambda_{min}(-2\beta\nabla_{yy}^2 f)\|\nabla_y f\|^2  \nonumber
\end{align*}

Therefore, following the Eq.~\eqref{gdot}, we have,
\begin{align*}
       -\dot{ \| g_0\|^2 } \geq \lambda_{min}(2\beta\nabla_{xx}^2 f)\|\nabla_x f\|^2+\lambda_{min}(-2\beta\nabla_{yy}^2 f)\|\nabla_y f\|^2  
\end{align*}

Resulting in the following Lyapunov key inequality,

\begin{align*} 
\dot{ \| g_0 \|^2 } &\le -\|g_0\|^2 \min \{\lambda_{\min}( 2\beta \nabla_{xx}^2 f), \lambda_{\min}(-2\beta \nabla_{yy}^2 f)\}
\end{align*}

Since, for convex-concave functions, $\min \{\lambda_{\min}( 2\beta \nabla_{xx}^2 f), \lambda_{\min}(-2\beta \nabla_{yy}^2 f)\}$ is always non-negative, which guarantees convergence of this dynamical system.

\section{Continuous time \cgo}\label{conttimecgo}
In this section, we first derive the continuous-time update rule of \cgo and then show convergence by choosing the norm squared of the gradient of $f$ as the Lyapunov function. Taking the \cgo update rule,
\begin{align*} 
\begin{bmatrix}
 \Delta x\\
 \Delta y
\end{bmatrix}
&= -\eta \begin{bmatrix}
I & \alpha \nabla_{xy}^2 f\\
-\alpha \nabla_{yx}f & I
\end{bmatrix}^{-1}\begin{bmatrix}
\nabla_x f \\
-\nabla_y f 
\end{bmatrix} 
\end{align*}

and taking the limit $\eta \rightarrow 0$, treating $\eta$ as time, and scaling time with $\beta$, we get,
\begin{align} \label{eq:update-rule}
\begin{bmatrix}
\dot x\\
\dot y
\end{bmatrix}  &= -\beta \begin{bmatrix}
I & \alpha \nabla_{xy}^2 f\\
-\alpha \nabla_{yx}^2 f & I
\end{bmatrix}^{-1}\begin{bmatrix}
\nabla_x f \\
-\nabla_y f 
\end{bmatrix} 
\end{align}

We further simplify Eq.~\eqref{eq:update-rule} by re-arranging the matrix inverse,
\begin{align}\label{contform1} 
\begin{bmatrix}
\dot x +\alpha \nabla_{xy}^2 f  \dot y \\
-\alpha \nabla_{yx}^2 f  \dot x + \dot y
\end{bmatrix}
=
\begin{bmatrix}
-\beta \nabla_x f \\
\beta \nabla_y f  
\end{bmatrix}
\end{align}
The above form will be useful in showing convergence. By solving for variable $\dot x, \dot y$, we get the explicit form,
\begin{align}\label{contform2}
\dot x &= -\beta \left( I + \alpha^2 \nabla_{xy}^2 f \nabla_{yx}^2 f \right)^{-1}  
            \left( \nabla_{x} + \alpha \nabla_{xy}^2 f  \nabla_{y} \right)\nonumber\\
\dot y &= -\beta \left( I + \alpha^2 \nabla_{yx}^2 f \nabla_{xy}^2 f \right)^{-1}  
            \left( \alpha \nabla_{yx}^2 f  \nabla_{x} - \nabla_{y} \right)
\end{align}
We use this construction to prove Theorem \eqref{contheorem}.

\begin{proof}[Proof of Theorem \eqref{contheorem}]
    We choose $\|g_0\|^2$ as our Lyapunov function and evaluate its time derivative to observe,
\begin{align} \label{dhh1}
\dot{\| g_0\|^2}&=\frac{\mathrm{d}\|g_0\|^2}{\mathrm{d}t}\nonumber\\
&= 2 g_0^{\top} \dot g_0 \nonumber\\
&=2\begin{bmatrix}
\nabla_x f^\top& -\nabla_y f ^\top
\end{bmatrix}\begin{bmatrix}
\nabla_{xx}^2 f &  \nabla_{xy}^2 f\\
-\nabla_{yx}^2 f & -\nabla_{yy}^2 f
\end{bmatrix}
\begin{bmatrix}
\dot x \\
\dot y
\end{bmatrix}
 \nonumber\\
 &=2 \dot x^{\top} \nabla_{xx}^2 f \nabla_x f +  2 \nabla_x f^{\top} \nabla_{xy}^2 f \dot y  + 2 \dot y^{\top} \nabla_{yy}^2 f \nabla_y f +  2 \nabla_y f^{\top} \nabla_{yx}^2 f \dot x
\end{align}

Ignoring the factor $2$, we expand the terms containing $\nabla_{xy}^2 f$ in Eq.~\eqref{dhh1} by replacing $\dot x$ and $\dot y$ using Eq.~\eqref{contform2} as follows,
\begin{align} \label{crosstermexpansion}
\dot x^{\top} \nabla_{xy}^2 f \nabla_y f + \nabla_x f^{\top} \nabla_{xy}^2 f\dot y &= -\beta \left( \nabla_{x} + \alpha \nabla_{xy}^2 f  \nabla_{y} \right)^{\top}\left( I + \alpha^2 \nabla_{xy}^2 f \nabla_{yx}^2 f \right)^{-1}  
 \nabla_{xy}^2 f \nabla_y f \nonumber\\
&\quad\quad-\nabla_x f^{\top} \nabla_{xy}^2 f\beta \left( I + \alpha^2 \nabla_{yx}^2 f \nabla_{xy}^2 f \right)^{-1}  
\left( \alpha \nabla_{yx}^2 f  \nabla_{x} - \nabla_{y} \right) \nonumber\\
%%%%%
&= -\beta \nabla_x f^{\top} \left( I + \alpha^2 \nabla_{xy}^2 f \nabla_{yx}^2 f \right)^{-1}\nabla_{xy}^2 f\nabla_y f\nonumber\\
&\quad\quad-\alpha \beta \nabla_y f^{\top} \nabla_{yx}^2 f\left( I + \alpha^2 \nabla_{yx}^2 f \nabla_{xy}^2 f \right)^{-1}\nabla_{xy}^2 f\nabla_y f\nonumber\\
%%%%%
&\quad\quad+\beta \nabla_x f^{\top}\nabla_{xy}^2 f \left( I + \alpha^2 \nabla_{yx}^2 f \nabla_{xy}^2 f \right)^{-1}\nabla_y f\nonumber\\
&\quad\quad-\alpha \beta \nabla_x f^{\top} \nabla_{xy}^2 f\left( I + \alpha^2 \nabla_{yx}^2 f \nabla_{xy}^2 f \right)^{-1}\nabla_{yx}^2 f\nabla_y f
\end{align}
Using the equality proven in Lemma \eqref{lemma:inverse-transpose} we have,
\begin{align*}
\dot x^{\top} \nabla_{xy}^2 f \nabla_y f + \nabla_x f^{\top} \nabla_{xy}^2 f\dot y &= -\alpha \beta \nabla_x f^{\top} \left( I + \alpha^2 \nabla_{xy}^2 f \nabla_{yx}^2 f \right)^{-1} \nabla_{xy}^2 f \nabla_{yx}^2 f \nabla_x f \\
&\quad\quad-\alpha \beta \nabla_y f^{\top}\left(I + \alpha^2 \nabla_{yx}^2 f \nabla_{xy}^2 f\right)^{-1} \nabla_{yx}^2 f \nabla_{xy}^2 f \nabla_y f\nonumber
\end{align*}
Using the expanded terms in RHS of Eq.~\eqref{crosstermexpansion} back into Eq.~\eqref{dhh1}, we obtain a unified expression,
\begin{align*} 
\dot{\| g_0\|^2} &= 2 \dot x^{\top} \nabla_{xx}^2 f \nabla_x f
-2 \alpha\beta \nabla_x f^{\top} \left(I + \alpha^2 \nabla_{xy}^2 f \nabla_{yx}^2 f\right)^{-1} \nabla_{xy}^2 f \nabla_{yx}^2 f \nabla_x f\\
&\quad\quad+ 2 \dot y^{\top} \nabla_{yy}^2 f \nabla_y f
- 2 \alpha \beta \nabla_y f^{\top} \left(I + \alpha^2 \nabla_{yx}^2 f \nabla_{xy}^2 f\right)^{-1} \nabla_{yx}^2 f \nabla_{xy}^2 f \nabla_y f 
\end{align*}

We now observe that $ \alpha \nabla_{xy}^2 f \dot y +  \beta \nabla_x f  = -\dot x$ and $ \alpha \nabla_{yx}^2 f\dot x + \beta  \nabla_y f = \dot y$, yielding in,
\begin{align}
\dot{ \| g_0\|^2 }
&= - 2 \beta  \nabla_x f^{\top} \nabla_{xx}^2 f \nabla_x f  - 2 \alpha\beta \nabla_x f^{\top} \left(I + \alpha^2 \nabla_{xy}^2 f \nabla_{yx}^2 f\right)^{-1} \nabla_{xy}^2 f \nabla_{xy}^2 f \nabla_x f\nonumber\\
&\quad\quad\quad\quad-2\alpha \dot y^{\top} \nabla_{yx}^2 f\nabla_{xx}^2 f\nabla_{x}\label{plugina}\\
&\quad\quad+ 2 \beta  \nabla_y f^{\top} \nabla_{yy}^2 f \nabla_y f  - 2 \alpha\beta \nabla_y f^{\top} \left(I + \alpha^2 \nabla_{yx}^2 f \nabla_{xy}^2 f\right)^{-1} \nabla_{yx}^2 f \nabla_{xy}^2 f \nabla_y f\nonumber\\
&\quad\quad\quad\quad+ 2\alpha\dot x^{\top} \nabla_{xy}^2 f\nabla_{yy}^2 f\nabla_y f\label{pluginb}
\end{align}
Substituting $\nabla_x f$ and $\nabla_y f$ in lines \eqref{plugina} and \eqref{pluginb} with their equivalences in Eq.~\eqref{contform1}, we get,
\begin{align}
\dot{ \| g_0\|^2 } 
&=- 2 \beta  \nabla_x f^{\top} \nabla_{xx}^2 f \nabla_x f  - 2 \alpha\beta \nabla_x f^{\top} \left(I + \alpha^2 \nabla_{xy}^2 f \nabla_{yx}^2 f\right)^{-1} \nabla_{xy}^2 f \nabla_{xy}^2 f \nabla_x f\nonumber\\
&\quad\quad-2\alpha \dot y^{\top} \nabla_{yx}^2 f\nabla_{xx}^2 f(-\frac{\dot x +\alpha \nabla_{xy}^2 f\dot y}{\beta})\label{transposea}\\
&\quad\quad+ 2 \beta  \nabla_y f^{\top} \nabla_{yy}^2 f \nabla_y f  - 2 \alpha\beta \nabla_y f^{\top} \left(I + \alpha^2 \nabla_{yx}^2 f \nabla_{xy}^2 f\right)^{-1} \nabla_{yx}^2 f \nabla_{xy}^2 f \nabla_y f \nonumber\\
&\quad\quad+ 2\alpha\dot x^{\top} \nabla_{xy}^2 f\nabla_{yy}^2 f\frac{\dot y-\alpha\nabla_{yx}^2 f\dot x}{\beta}\label{transposeb}
\end{align}
Taking transpose of the final terms in lines \eqref{transposea} and \eqref{transposeb}, we obtain,
\begin{align}\label{eq:gdot_cgo_intermediate}
\dot{ \| g_0\|^2 } &=- 2 \beta  \nabla_x f^{\top} \nabla_{xx}^2 f \nabla_x f  - 2 \alpha\beta \nabla_x f^{\top} \left(I + \alpha^2 \nabla_{xy}^2 f \nabla_{yx}^2 f\right)^{-1} \nabla_{xy}^2 f \nabla_{xy}^2 f \nabla_x f^{\top}\nonumber\\
&\quad\quad+ \frac{2}{\beta} \left(\alpha \dot x + \alpha^2 \nabla_{xy}^2 f \dot y \right)^{\top}  \nabla_{xx}^2 f \nabla_{xy}^2 f \dot y\nonumber\\
&\quad\quad+ 2 \beta  \nabla_y f^{\top} \nabla_{yy}^2 f \nabla_y f  - 2 \alpha\beta \nabla_y f^{\top} \left(I + \alpha^2 \nabla_{yx}^2 f \nabla_{xy}^2 f\right)^{-1} \nabla_{yx}^2 f \nabla_{xy}^2 f \nabla_y f^{\top}\nonumber\\
&\quad\quad+ \frac{2}{\beta} \left(\alpha \dot y - \alpha^2 \nabla_{yx}^2 f \dot x \right)^{\top} \nabla_{yy}^2 f \nabla_{yx}^2 f \dot x
\end{align}

We utilize the Peter-Paul inequality to further expand $\nabla_{xx}^2 f$ and $\nabla_{yy}^2 f$ terms in Eq.~\eqref{eq:gdot_cgo_intermediate}. In particular, we derive the following inequalities, 
$$2\dot x^{\top}\nabla_{xx}^2 f\nabla_{xy}^2 f\dot y \leq \|\dot x^{\top}\nabla_{xx}^2 f\|^2 +\|\nabla_{xy}^2 f\dot y\|^2 $$
and $$2\dot x^{\top}\nabla_{xx}^2 f\nabla_{xy}^2 f\dot y \leq \|\dot x^{\top}\nabla_{xx}^2 f\|^2 +\|\nabla_{xy}^2 f\dot y\|^2. $$
Using these inequalities in Eq.~\eqref{eq:gdot_cgo_intermediate}, we have,
\begin{align*}
\| \dot g_0\|^2 &\leq - 2\beta  \nabla_x f^{\top} \nabla_{xx}^2 f \nabla_x f 
-2 \alpha \beta \nabla_x f^{\top} \left(I + \alpha^2 \nabla_{xy}^2 f \nabla_{yx}^2 f\right)^{-1} \nabla_{xy}^2 f \nabla_{xy}^2 f \nabla_x f\\
&\quad\quad+ \frac{1}{\beta}\dot y^{\top} \nabla_{yx}^2 f \left(\alpha I + 2\alpha^2 \nabla_{xx}^2 f \right) \nabla_{xy}^2 f \dot y + 2 \beta \nabla_y f^{\top} \nabla_{yy}^2 f \nabla_y f \\
&\quad\quad  - 2 \alpha\beta \nabla_y f^{\top} \left(I + \alpha^2 \nabla_{yx}^2 f \nabla_{xy}^2 f\right)^{-1} \nabla_{yx}^2 f \nabla_{xy}^2 f \nabla_y f +\frac{1}{\beta}\dot x^{\top} \nabla_{xy}^2 f \left(\alpha I - 2\alpha^2 \nabla_{yy}^2 f \right) \nabla_{yx}^2 f \dot x \\ 
&\quad\quad+ \frac{\alpha}{\beta} \dot y^{\top} \nabla_{yy}^2 f \nabla_{yy}^2 f \dot y+\frac{\alpha}{\beta} \dot x^{\top} \nabla_{xx}^2 f \nabla_{xx}^2 f \dot x
\end{align*}

Considering that $\nabla_{xx}^2 f$ and $\nabla_{yy}^2 f$ are symmetric matrices, we have,
\begin{align}\label{main}
\dot{ \| g_0\|^2 } &\leq 
- 2\beta \nabla_x f^{\top} \nabla_{xx}^2 f \nabla_x f - 2\alpha\beta \nabla_x f^{\top} \left(I + \alpha^2 \nabla_{xy}^2 f \nabla_{yx}^2 f\right)^{-1} \nabla_{xy}^2 f \nabla_{xy}^2 f \nabla_x f \nonumber\\
&\quad\quad+ \frac{\alpha}{\beta} \dot x^{\top} \nabla_{xx}^2 f \nabla_{xx}^2 f \dot x +  \frac{1}{\beta} \left(\alpha  + 2\alpha^2\overline{\lambda_{xx}}\right) \|\nabla_{xy}^2 f \dot y\|^2\nonumber\\
&\quad\quad+ 2 \beta \nabla_y f^{\top} \nabla_{yy}^2 f \nabla_y f  -2 \alpha \beta \nabla_y f^{\top} \left(I + \alpha^2 \nabla_{yx}^2 f \nabla_{xy}^2 f\right)^{-1} \nabla_{yx}^2 f \nabla_{xy}^2 f \nabla_y f \nonumber\\
&\quad\quad+ \frac{\alpha}{\beta} \dot y^{\top} \nabla_{yy}^2 f \nabla_{yy}^2 f \dot y +  \frac{1}{\beta} \left(\alpha  - 2\alpha^2\underline{\lambda_{yy}} \right)\| \nabla_{yx}^2 f \dot x\|^2 
\end{align}
Setting $\overline{\lambda_1} = \max(\overline{\lambda_{xx}},-\underline{\lambda_{yy}})$ we obtain,
\begin{align}\label{lam1}
\dot{ \| g_0\|^2 } &\leq 
- 2\beta \nabla_x f^{\top} \nabla_{xx}^2 f \nabla_x f - 2\alpha\beta \nabla_x f^{\top} \left(I + \alpha^2 \nabla_{xy}^2 f \nabla_{yx}^2 f\right)^{-1} \nabla_{xy}^2 f \nabla_{xy}^2 f \nabla_x f\nonumber \\
&\quad\quad+ \frac{\alpha}{\beta} \dot x^{\top} \nabla_{xx}^2 f \nabla_{xx}^2 f \dot x +  \frac{1}{\beta} \left(\alpha  + 2\alpha^2\overline{\lambda_1}\right) \|\nabla_{xy}^2 f \dot y\|^2\nonumber\\
&\quad\quad+ 2 \beta \nabla_y f^{\top} \nabla_{yy}^2 f \nabla_y f  -2 \alpha \beta \nabla_y f^{\top} \left(I + \alpha^2 \nabla_{yx}^2 f \nabla_{xy}^2 f\right)^{-1} \nabla_{yx}^2 f \nabla_{xy}^2 f \nabla_y f \nonumber\\
&\quad\quad+ \frac{\alpha}{\beta} \dot y^{\top} \nabla_{yy}^2 f \nabla_{yy}^2 f \dot y +  \frac{1}{\beta} \left(\alpha  + 2\alpha^2\overline{\lambda_1} \right)\| \nabla_{yx}^2 f \dot x\|^2 
\end{align}    
Using the update rule in Eq.~\eqref{contform2}, we compute,
\begin{align*}
\left\|\nabla_{yx}^2 f \dot x \right\|^2 &= \beta^2 \left(\nabla_x f + \alpha \nabla_{xy}^2 f \nabla_y f\right)^{\top} \left( I + \alpha^2 \nabla_{xy}^2 f \nabla_{yx}^2 f \right)^{-2} \nabla_{xy}^2 f \nabla_{yx}^2 f \left(\nabla_x f +  \alpha \nabla_{xy}^2 f \nabla_y f\right)\\ 
\left\|\nabla_{xy}^2 f \dot y \right\|^2 &= \beta^2 \left(-\nabla_y f + \alpha \nabla_{yx}^2 f \nabla_x f\right)^{\top} \left( I +\alpha^2 \nabla_{yx}^2 f \nabla_{xy}^2 f \right)^{-2} \nabla_{yx}^2 f \nabla_{xy}^2 f \left(-\nabla_y f + \alpha \nabla_{yx}^2 f \nabla_x f\right).
\end{align*}

by adding up the two equalities above, we obtain, 
\begin{align} \label{eq:sum_of_norms}
\left\|\nabla_{yx}^2 f \dot x\right\|^2 + \left\|\nabla_{xy}^2 f \dot y\right\|^2 &=
\beta^2 \nabla_x f^{\top} \left( I +\alpha^2 \nabla_{xy}^2 f \nabla_{yx}^2 f \right)^{-2} \nabla_{xy}^2 f \nabla_{yx}^2 f \nabla_x f \nonumber\\
&\quad\quad+\beta^2 \nabla_y f^{\top} \left(I +\alpha^2 \nabla_{yx}^2 f \nabla_{xy}^2 f \right)^{-2} \nabla_{yx}^2 f \nabla_{xy}^2 f \nabla_y f\nonumber\\
&\quad\quad+\alpha\beta^2 \nabla_x f^{\top} \left( I +\alpha^2 \nabla_{xy}^2 f \nabla_{yx}^2 f \right)^{-2} \nabla_{xy}^2 f \nabla_{yx}^2 f \nabla_{xy}^2 f\nabla_y f \nonumber\\
&\quad\quad+\alpha\beta^2 \nabla_y f^{\top} \nabla_{yx}^2 f\left(I +\alpha^2 \nabla_{yx}^2 f \nabla_{xy}^2 f \right)^{-2} \nabla_{xy}^2 f \nabla_{yx}^2 f \nabla_x f\nonumber\\
&\quad\quad-\underbrace{\alpha\beta^2 \nabla_x f^{\top} \nabla_{xy}^2 f \left( I +\alpha^2 \nabla_{xy}^2 f \nabla_{yx}^2 f \right)^{-2} \nabla_{yx}^2 f \nabla_{xy}^2 f \nabla_y f}_{{(i)}}\nonumber \\
&\quad\quad-\underbrace{\alpha\beta^2 \nabla_y f^{\top} \left(I +\alpha^2 \nabla_{yx}^2 f \nabla_{xy}^2 f \right)^{-2} \nabla_{yx}^2 f \nabla_{xy}^2 f \nabla_{yx}^2 f \nabla_x f}_{{(ii)}}\nonumber\\
&\quad\quad+\underbrace{\alpha^2 \beta^2 \nabla_x f^{\top} \nabla_{xy}^2 f\left( I +\alpha^2 \nabla_{xy}^2 f \nabla_{yx}^2 f \right)^{-2} \nabla_{yx}^2 f\nabla_{xy}^2 f \nabla_{yx}^2 f \nabla_x f}_{{(iii)}}\nonumber \\
&\quad\quad+\underbrace{\alpha^2 \beta^2 \nabla_y f^{\top}\nabla_{yx}^2 f \left(I +\alpha^2 \nabla_{yx}^2 f \nabla_{xy}^2 f \right)^{-2} \nabla_{xy}^2 f \nabla_{yx}^2 f \nabla_{xy}^2 f \nabla_y f}_{{(iv)}}
\end{align}        
We further analyze the last four terms of the Eq.~\eqref{eq:sum_of_norms}. In particular, we utilize the statement of Lemma \eqref{lemma:inverse-transpose} and for the term $(i)$ in the above equality, we have, 
\begin{align*} 
&\alpha\beta^2 \nabla_x f^{\top} \left( I +\alpha^2 \nabla_{xy}^2 f \nabla_{yx}^2 f \right)^{-2} \nabla_{xy}^2 f \nabla_{yx}^2 f \nabla_{xy}^2 f\nabla_y f \\
&\quad\quad= \alpha\beta^2 \nabla_x f^{\top} 
\nabla_{xy}^2 f \left( I +\alpha^2 \nabla_{xy}^2 f \nabla_{yx}^2 f \right)^{-2} \nabla_{yx}^2 f \nabla_{xy}^2 f \nabla_y f
\end{align*}
correspondingly, for the term $(ii)$, we have,
\begin{align*} 
&\alpha\beta^2 \nabla_y f^{\top} \nabla_{yx}^2 f\left(I +\alpha^2 \nabla_{yx}^2 f \nabla_{xy}^2 f \right)^{-2} \nabla_{xy}^2 f \nabla_{yx}^2 f \nabla_x f\\
&\quad\quad= \alpha\beta^2 \nabla_y f^{\top} \left(I +\alpha^2 \nabla_{yx}^2 f \nabla_{xy}^2 f \right)^{-2} \nabla_{yx}^2 f \nabla_{xy}^2 f \nabla_{yx}^2 f \nabla_x f 
\end{align*}
for the term $(iii)$, we have,
\begin{align*} 
&\alpha^2 \beta^2 \nabla_x f^{\top} \nabla_{xy}^2 f\left( I +\alpha^2 \nabla_{xy}^2 f \nabla_{yx}^2 f \right)^{-2} \nabla_{yx}^2 f\nabla_{xy}^2 f \nabla_{yx}^2 f \nabla_x f\\
&\quad\quad= \alpha^2 \beta^2 \nabla_x f^{\top} \left( I +\alpha^2 \nabla_{xy}^2 f \nabla_{yx}^2 f \right)^{-2}\nabla_{xy}^2 f \nabla_{yx}^2 f\nabla_{xy}^2 f \nabla_{yx}^2 f \nabla_x f
\end{align*}
correspondingly, for the term $(iv)$, we have,
\begin{align*} 
&\alpha^2 \beta^2 \nabla_y f^{\top}\nabla_{yx}^2 f \left(I +\alpha^2 \nabla_{yx}^2 f \nabla_{xy}^2 f \right)^{-2} \nabla_{xy}^2 f \nabla_{yx}^2 f \nabla_{xy}^2 f \nabla_y f\\ 
&\quad\quad= \alpha^2 \beta^2 \nabla_y f^{\top} \left(I +\alpha^2 \nabla_{yx}^2 f \nabla_{xy}^2 f \right)^{-2} \nabla_{yx}^2 f \nabla_{xy}^2 f \nabla_{yx}^2 f \nabla_{xy}^2 f \nabla_y f
\end{align*}
Putting these equalities together in Eq.~\eqref{eq:sum_of_norms}, we have,

\begin{align} \label{P1}
\left\|\nabla_{yx}^2 f \dot x\right\|^2 & + \left\|\nabla_{xy}^2 f \dot y\right\|^2 \nonumber\\
&= \beta^2 \nabla_x f^{\top} \left( I +\alpha^2 \nabla_{xy}^2 f \nabla_{yx}^2 f \right)^{-2} \left(\nabla_{xy}^2 f \nabla_{yx}^2 f +\alpha^2 \nabla_{xy}^2 f \nabla_{yx}^2 f \nabla_{xy}^2 f \nabla_{yx}^2 f \right) \nabla_x f \nonumber \\        
        &\quad\quad+ \beta^2 \nabla_y f^{\top} \left( I + \alpha^2\nabla_{yx}^2 f \nabla_{xy}^2 f \right)^{-2} \left(\nabla_{yx}^2 f \nabla_{xy}^2 f +\alpha^2 \nabla_{yx}^2 f \nabla_{xy}^2 f \nabla_{yx}^2 f \nabla_{xy}^2 f \right) \nabla_y f\nonumber\\
        &=  \beta^2\nabla_x f^{\top} \left(I + \alpha^2 \nabla_{xy}^2 f \nabla_{yx}^2 f\right)^{-1} \nabla_{xy}^2 f \nabla_{yx}^2 f \nabla_x f \nonumber\\
        &\quad\quad+ \beta^2 \nabla_y f^{\top} \left(I +\alpha^2 \nabla_{yx}^2 f \nabla_{xy}^2 f \right)^{-1} \nabla_{yx}^2 f \nabla_{xy}^2 f \nabla_y f
\end{align}        
Plugging this into Eq.~\eqref{lam1} we obtain,
\begin{align}\label{main2}
\dot{ \| g_0\|^2 }
\leq &- 2 \beta \nabla_x f^{\top} \nabla_{xx}^2 f \nabla_x f + \beta(2\alpha^2\overline{\lambda_1}-\alpha)\nabla_x f^{\top} \left(I + \alpha^2 \nabla_{xy}^2 f \nabla_{yx}^2 f\right)^{-1} \nabla_{xy}^2 f \nabla_{xy}^2 f \nabla_x f \nonumber\\
&\quad\quad+ \frac{\alpha}{\beta} \dot x^{\top} \nabla_{xx}^2 f \nabla_{xx}^2 f \dot x+\frac{\alpha}{\beta} \dot y^{\top} \nabla_{yy}^2 f \nabla_{yy}^2 f \dot y \nonumber\\
&\quad\quad+ 2  \beta \nabla_y f^{\top} \nabla_{yy}^2 f \nabla_y f  \nonumber\\
&\quad\quad+  \beta(2\alpha^2\overline{\lambda_1}-\alpha) \nabla_y f^{\top} \left(I + \alpha^2 \nabla_{yx}^2 f \nabla_{xy}^2 f\right)^{-1} \nabla_{yx}^2 f \nabla_{xy}^2 f \nabla_y f 
\end{align}
 We now do the following set of computations,
\begin{align*} 
    \dot x^{\top} \nabla_{xx}^2 f \nabla_{xx}^2 f \dot x+\dot y^{\top} \nabla_{yy}^2 f \nabla_{yy}^2 f \dot y
    &\stackrel{(a)}{=} (-\alpha \nabla_{xy}^2 f\dot y-\beta \nabla_x f)^{\top}  \nabla_{xx}^2 f \nabla_{xx}^2 f  (-\alpha \nabla_{xy}^2 f\dot y-\beta \nabla_x f)\\
    &\quad\quad+ (\alpha \nabla_{yx}^2 f\dot x+ \beta \nabla_y f)^{\top}\nabla_{yy}^2 f \nabla_{yy}^2 f (\alpha \nabla_{yx}^2 f \dot x+ \beta \nabla_y f) \\
    &\stackrel{(b)}{=} \|(\alpha \nabla_{xy}^2 f \dot y+ \beta \nabla_x f)^{\top}\nabla_{xx}^2 f\|^2\\
    & \quad\quad+\|(\alpha \nabla_{yx}^2 f\dot x+\beta \nabla_y f)^{\top}\nabla_{yy}^2 f\|^2 \\
    &\stackrel{(c)}{\leq} 2\alpha^2 \|\nabla_{xy}^2 f \dot y\|^2  \|\nabla_{xx}^2 f\|^2 + 2\alpha^2   \|\nabla_{yx}^2 f \dot x\|^2 \|\nabla_{yy}^2 f \|^2\\
    & \quad\quad+ 2\beta^2\|\nabla_x f \nabla_{xx}^2 f\|^2+2\beta^2\|\nabla_y f
     \nabla_{yy}^2 f\|^2\\
    &\stackrel{(d)}{\leq} 2\beta^2\alpha^2\overline{\lambda_{xx}}^2 \|\nabla_{xy}^2 f \dot y\|^2+2\beta^2  \alpha^2\overline{\lambda_{yy}}^2 \|\nabla_{yx}^2 f \dot x\|^2\\
    & \quad\quad+2 \beta^2 \overline{\lambda_{xx}}^2\nabla_x f^{\top}\nabla_x f + 2\beta^2 \overline{\lambda_{yy}}^2 \nabla_y f^{\top} \nabla_y f\\
    &\stackrel{(e)}{\leq} 2\beta^2\alpha^2\overline{\lambda_2}^2 (\|\nabla_{xy}^2 f \dot y\|^2+\|\nabla_{yx}^2 f \dot x\|^2)\\
    & \quad\quad+2 \beta^2 \overline{\lambda_{xx}}^2\nabla_x f^{\top}\nabla_x f + 2\beta^2 \overline{\lambda_{yy}}^2 \nabla_y f^{\top} \nabla_y f\\
    &\stackrel{(f)}{\leq} 2\beta^2\alpha^2\overline{\lambda_2}^2 \nabla_x f^{\top} \left(I + \alpha^2 \nabla_{xy}^2 f\nabla_{yx}^2 f \right)^{-1} \nabla_{xy}^2 f\nabla_{yx}^2 f \nabla_x f\\
    & \quad\quad+2\beta^2  \alpha^2\overline{\lambda_2}^2 \nabla_y f^{\top} \left(I +\alpha^2 \nabla_{yx}^2 f\nabla_{xy}^2 f  \right)^{-1} \nabla_{yx}^2 f\nabla_{xy}^2 f  \nabla_y f\\
    & \quad\quad+2 \beta^2 \overline{\lambda_{xx}}^2\nabla_x f^{\top}\nabla_x f + 2\beta^2 \overline{\lambda_{yy}}^2 \nabla_y f^{\top} \nabla_y f 
\end{align*}
% tag explain each of the above inequalities.
    Where for $(a)$ we use Eq.~\eqref{contform1} to substitute $\Delta x$ and $\Delta y$, in $(b)$ we re-write the terms as norms, in $(c)$ we use the inequality $\|a+b\|^2\leq2\|a\|^2+2\|b\|^2$, in $(d)$ we bound the terms using the maximum eigenvalues, in $(e)$ we set $\overline{\lambda_2} = \max(\overline{\lambda_{xx}},\overline{\lambda_{yy}})$ and finally for $(f)$ we use Eq.~\eqref{P1}.
    
    Using the above inequality in Eq.~\eqref{main2}, we have,
    \begin{align*}
    \dot{ \| g_0\|^2 } &\leq -   \nabla_x f^{\top} (2\beta \nabla_{xx}^2 f-2  \beta\alpha\overline{\lambda_{xx}}^2I) \nabla_x f +  \nabla_y f^{\top} (2\beta \nabla_{yy}^2 f+2  \beta\alpha\overline{\lambda_{yy}}^2I)  \nabla_y f \\
    &\quad\quad- \beta(\alpha-2\alpha^2\overline{\lambda_{1}}-2\alpha^3\overline{\lambda_2}^2) \nabla_y f^{\top} \left(I + \alpha^2 \nabla_{yx}^2 f \nabla_{xy}^2 f\right)^{-1} \nabla_{yx}^2 f \nabla_{xy}^2 f \nabla_y f \\
    &\quad\quad- \beta(\alpha-2\alpha^2\overline{\lambda_{1}}-2\alpha^3\overline{\lambda_2}^2)\nabla_x f^{\top} \left(I + \alpha^2 \nabla_{xy}^2 f \nabla_{yx}^2 f\right)^{-1} \nabla_{xy}^2 f \nabla_{xy}^2 f \nabla_x f
    \end{align*}
    
By rearranging the above inequality, we get,
    
    \begin{align*}
    \dot{ \| g_0 \|^2 } &
    \leq - \nabla_x f^{\top} \left( (2\beta \nabla_{xx}^2 f-2  \beta\alpha\overline{\lambda_{xx}}^2I)+ \beta(\alpha-2\alpha^2\overline{\lambda_1}-2\alpha^3\overline{\lambda_2}^2) \left(I + \alpha^2 \nabla_{xy}^2 f \nabla_{yx}^2 f\right)^{-1} \nabla_{xy}^2 f \nabla_{yx}^2 f \right)\nabla_x f\\
    & \quad\quad - \nabla_y f^{\top} \left(- (2\beta \nabla_{yy}^2 f+2  \beta\alpha\overline{\lambda_{yy}}^2I) + \beta(\alpha-2\alpha^2\underline{\lambda_1}-2\alpha^3\overline{\lambda_2}^2) \left(I + \alpha^2 \nabla_{yx}^2 f \nabla_{xy}^2 f\right)^{-1} \nabla_{yx}^2 f \nabla_{xy}^2 f \right) \nabla_y f    \\
    &\leq -\|g_0\|^2 \min \{\lambda_{\min}((2\beta \nabla_{xx}^2 f-2  \beta\alpha\overline{\lambda_{xx}}^2I)+ \beta(\alpha-2\alpha^2\overline{\lambda_1}-2\alpha^3\overline{\lambda_2}^2) \left(I + \alpha^2 \nabla_{xy}^2 f \nabla_{yx}^2 f\right)^{-1} \nabla_{xy}^2 f \nabla_{yx}^2 f),\\
    &\quad\quad\lambda_{\min}(- (2\beta \nabla_{yy}^2 f+2  \beta\alpha\overline{\lambda_{yy}}^2I) + \beta(\alpha-2\alpha^2\underline{\lambda_1}-2\alpha^3\overline{\lambda_2}^2) \left(I + \alpha^2 \nabla_{yx}^2 f \nabla_{xy}^2 f\right)^{-1} \nabla_{yx}^2 f \nabla_{xy}^2 f)\}
    \end{align*}
  
which is the key Lyapunov inequality. 
Thus, under the conditions expressed in the statement of the main Theorem, i.e., $\lambda$, as defined in the following is positive,
    \begin{align}
    \lambda &:= \min \{\lambda_{\min}((2\beta \nabla_{xx}^2 f-2  \beta\alpha\overline{\lambda_{xx}}^2I)+ \beta(\alpha-2\alpha^2\overline{\lambda_1}-2\alpha^3\overline{\lambda_2}^2) \left(I + \alpha^2 \nabla_{xy}^2 f \nabla_{yx}^2 f\right)^{-1} \nabla_{xy}^2 f \nabla_{yx}^2 f),\label{keyinequalityresulta}\\
    &\quad\quad\lambda_{\min}(- (2\beta \nabla_{yy}^2 f+2  \beta\alpha\overline{\lambda_{yy}}^2I) + \beta(\alpha-2\alpha^2\underline{\lambda_1}-2\alpha^3\overline{\lambda_2}^2) \left(I + \alpha^2 \nabla_{yx}^2 f \nabla_{xy}^2 f\right)^{-1} \nabla_{yx}^2 f \nabla_{xy}^2 f)\label{keyinequalityresultb}\}
    \end{align}    
the quantity $\|g_0\|^2$ converges to zero exponentially fast with the rate at least $\lambda$.
    \end{proof}
    
Now, we simplify the above expression of the rate using Lemmas \eqref{lemma:eigen-properties} and \eqref{lemma:eigen-properties1} to address the $1^{st}$ and $2^{nd}$ terms respectively in lines \eqref{keyinequalityresulta} and \eqref{keyinequalityresultb},
\begin{align*}
    \lambda_{min} &\geq \beta \min\{2 \underline{\lambda_{xx}}-2  \alpha\overline{\lambda_{xx}}^2+ \beta(\alpha-2\alpha^2\overline{\lambda_1}-2\alpha^3\overline{\lambda_2}^2)\frac{\underline{\lambda_{xy}}}{1+\alpha^2\underline{\lambda_{xy}}
},\\
&\quad\quad-2 \overline{\lambda_{yy}}-2  \alpha\overline{\lambda_{yy}}^2+ \beta(\alpha-2\alpha^2\overline{\lambda_1}-2\alpha^3\overline{\lambda_2}^2)\frac{\underline{\lambda_{yx}}}{1+\alpha^2\underline{\lambda_{yx}}
}\}
\end{align*}

To better understand the above results, we set some relations between the quantities in the above expression. If we set $\alpha$ such that $\overline{\lambda_{xx}} \leq \frac{1}{5\alpha};\underline{\lambda_{xx}} \geq -\frac{1}{5\alpha};\underline{\lambda_{yx}},\underline{\lambda_{xy}}\leq \frac{K}{\alpha^2};\underline{\lambda_{yy}} \geq -\frac{1}{5\alpha};\overline{\lambda_{yy}} \leq \frac{1}{5\alpha};K\gg1$. We have $\overline{\lambda_1},\overline{\lambda_1}\leq \frac{1}{5\alpha}$ and we obtain $\lambda_{min} \geq \frac{1}{50\alpha}$.\\ This shows that as long as the interaction terms $\underline{\lambda_{yx}},\underline{\lambda_{xy}}$ are of the order of the square of the deviation of the pure terms $\underline{\lambda_{xx}},\overline{\lambda_{yy}}$ (from the convex-concave condition i.e. $\underline{\lambda_{xx}}\geq 0,\overline{\lambda_{yy}}\leq 0$), we can guarantee convergence for $\cgo$

Statements and proofs of the Lemmas used in the above derivation are provided below,
\begin{lemma} \label{lemma:inverse-transpose} The following equality holds, 
$$\nabla_{yx}^2 f(I + \alpha^2 \nabla_{xy}^2 f \nabla_{yx}^2 f)^{-1} = (I + \alpha^2 \nabla_{yx}^2 f \nabla_{xy}^2 f)^{-1}\nabla_{yx}^2 f.$$
\end{lemma}
\begin{proof}
To prove this equality statement, we write,

\begin{align*}
\nabla_{yx}^2 f  + \alpha^2 \nabla_{yx}^2 f \nabla_{xy}^2 f \nabla_{yx}^2 f &= (I + \alpha^2 \nabla_{yx}^2 f \nabla_{xy}^2 f) \nabla_{yx}^2 f
\end{align*}
and at the same time, 
\begin{align*}
\nabla_{yx}^2 f  + \alpha^2 \nabla_{yx}^2 f \nabla_{xy}^2 f \nabla_{yx}^2 f &=
\nabla_{yx}^2 f  (I + \alpha^2 \nabla_{xy}^2 f \nabla_{yx}^2 f)
\end{align*}
therefore, we have,
\begin{align*}
(I + \alpha^2 \nabla_{yx}^2 f \nabla_{xy}^2 f) \nabla_{yx}^2 f = \nabla_{yx}^2 f  (I + \alpha^2 \nabla_{xy}^2 f \nabla_{yx}^2 f)
\end{align*}
Multiplying both sides with the inverse of $(I + \alpha^2 \nabla_{yx}^2 f \nabla_{xy}^2 f)$ from the left, and the inverse of $(I + \alpha^2 \nabla_{xy}^2 f \nabla_{yx}^2 f)$ from the right results in,
\begin{align*}
\nabla_{yx}^2 f(I + \alpha^2 \nabla_{xy}^2 f \nabla_{yx}^2 f)^{-1} &= (I + \alpha^2 \nabla_{yx}^2 f \nabla_{xy}^2 f)^{-1}\nabla_{yx}^2 f
\end{align*}
which is the statement of the Lemma.
\end{proof}

\begin{lemma} \label{lemma:eigen-properties}
The following inequality holds, 
$\lambda_{\min}(A+B) \geq \lambda_{\min}(A)+\lambda_{\min}(B),\quad\forall A,B\in \S^{n}_+$.
\end{lemma}
\begin{proof}
We know,
\begin{align*}
\|\Delta x\|^2\lambda_{min}(A+B) \geq x^\top (A+B) x = x^\top A x+ x^\top B x , \quad\forall x
\end{align*}

the following also holds,

$$x^\top A x+ x^\top B x  \geq \|\Delta x\|^2 \lambda_{min}(A)+\|\Delta x\|^2\lambda_{min}(B),\quad\forall x$$
Choosing $x$ not equal to zero we complete the proof.
\end{proof}

\begin{lemma} \label{lemma:eigen-properties1}
Let $B\in \S^n_+$, if (I+B) is invertible, $\lambda_{min}((I+B)^{-1}B)\geq \frac{\underline{\lambda_b}}{1+\underline{\lambda_b}}$, where $\underline{\lambda_b}=\lambda_{min}(B)$
\end{lemma}
\begin{proof}
We can write the following,
\begin{align*}
(I+B)^{-1}B = (I+B)^{-1}B = I-(I+B)^{-1}
\end{align*}

From the statement of Lemma~\eqref{lemma:eigen-properties} we can write,
\begin{align*}
\lambda_{min}((I+B)^{-1}B) \geq \lambda_{min}(I)+\lambda_{min}(-(I+B)^{-1})
\end{align*}
Hence we have,
\begin{align*}
\lambda_{min}((I+B)^{-1}B) \geq 1+(-\frac{1}{1+\underline{\lambda_b}}) = \frac{\underline{\lambda_b}}{1+\underline{\lambda_b}}
\end{align*}

which is the statement of the Lemma.
\end{proof}

% \part{Proofs for discrete time}
% % \faketableofcontents
% \parttoc 

\section{Discrete time \gda}\label{discretegdasection}
In this section, we present the analysis of the discrete time \gda algorithm for completeness. We first present the optimization problem and then derive \gda convergence conditions and convergence rate.

To come up with the update rule, we solve the below optimization problem,
\begin{align}
    \begin{split}
    \label{eqn:localgame1}
    \min_{\delta x \in \Reals^m} \delta  x^{\top} \nabla_x f &+ \delta y^{\top} \nabla_y f + \frac{1}{2\eta} \delta x^{\top} \delta x \\
    \max_{\delta y \in \Reals^n} \delta  y^{\top} \nabla_y f &+ \delta x^{\top} \nabla_x f - \frac{1}{2\eta} \delta y^{\top} \delta y.
    \end{split}
\end{align}
Which gives,
\begin{align} \label{eq:update-rule2}
\begin{bmatrix}
 \Delta x\\
 \Delta y
\end{bmatrix} &= -\eta \begin{bmatrix}
\nabla_x f \\
-\nabla_y f 
\end{bmatrix}
\end{align}

We now write the Taylor expansion of $\nabla_x f,\nabla_y f$ around the $(x,y)$,
    \begin{align*}   
        & \nabla_x f (\Delta x + x, \Delta y + y) 
        = \nabla_x f(x,y) + \nabla_{xx}^2 f \Delta x + \nabla_{xy}^2 f \Delta y 
        + \R_x(\Delta x,\Delta y)\\
        & \nabla_y f(\Delta x + x, \Delta y + y) = \nabla_y f(x,y) + \nabla_{yy}^2 f \Delta y + \nabla_{yx}^2 f \Delta x + \R_y(\Delta x,\Delta y)
    \end{align*}   
    where the remainder terms $\R_x$ and $\R_y$ are defined as,
    \begin{align}\label{gdaremaindersa}
    &\R_x(\Delta x,\Delta y) \defeq \int \limits_0^1 \left( \left(\nabla_{xx}^2 f(t \Delta x + x, t \Delta y + y) - \nabla_{xx}^2 f\right) \Delta x + \left(\nabla_{xy}^2 f(t \Delta x + x, t \Delta y + y) - \nabla_{xy}^2 f\right) \Delta y \right) \mathrm{ d}t \\
    &\R_y(\Delta x,\Delta y) \defeq \int \limits_0^1 \left( \left(\nabla_{yy}^2 f(t \Delta x + x, t \Delta y + y) - \nabla_{yy}^2 f\right) \Delta y + \left(\nabla_{yx}^2 f(t \Delta x + x, t \Delta y + y) - \nabla_{yx}^2 f\right) \Delta x \right) \mathrm{d}t \nonumber \end{align}   
    
    Using this equality, we obtain,
    \begin{align*}   
    &\left\| \nabla_x f \left(\Delta x + x, \Delta y + y\right) \right\|^2 + \left\| \nabla_y f\left(\Delta x + x, \Delta y + y\right) \right\|^2 
    -\|\nabla_x f (x,y)\|^2 - \|\nabla_y f(x,y)\|^2 \\
    &\quad\quad= 2 \Delta x^{\top} \nabla_{xx}^2 f \nabla_x f (x,y) +  2 \nabla_x f (x,y) ^{\top} \nabla_{xy}^2 f \Delta y + \Delta x^{\top} \nabla_{xx}^2 f \nabla_{xx}^2 f \Delta x \\
    &\quad\quad\quad\quad+ 2 \Delta y^{\top} \nabla_{yy}^2 f  \nabla_y f(x,y) +  2 \nabla_y f(x,y) ^{\top} \nabla_{yx}^2 f \Delta x + \Delta y^{\top} \nabla_{yy}^2 f \nabla_{yy}^2 f \Delta y + \\
    &\quad\quad\quad\quad+ \Delta y^{\top} \nabla_{yx}^2 f \nabla_{xy}^2 f \Delta y + 
    \Delta x^{\top} \nabla_{xy}^2 f \nabla_{yx}^2 f \Delta x+ 2 \Delta x^{\top} \nabla_{xx}^2 f \nabla_{xy}^2 f \Delta y+2 \Delta y^{\top} \nabla_{yy}^2 f \nabla_{yx}^2 f \Delta x    \\
    &\quad\quad\quad\quad+ 2\nabla_x f (x,y) ^{\top} \R_x(\Delta x,\Delta y)  + 2\Delta x^{\top} \nabla_{xx}^2 f \R_x(\Delta x,\Delta y) + 2 \Delta y^{\top} \nabla_{yx}^2 f \R_x(\Delta x,\Delta y) + \|\R_x(\Delta x,\Delta y)\|^2\\
    &\quad\quad\quad\quad+ 2\nabla_y f(x,y) ^{\top} \R_y(\Delta x,\Delta y)  + 2\Delta y^{\top} \nabla_{yy}^2 f \R_y(\Delta x,\Delta y) + 2 \Delta x^{\top} \nabla_{xy}^2 f \R_y(\Delta x,\Delta y) + \|\R_y(\Delta x,\Delta y)\|^2
    \end{align*}   
    
    Substituting $\Delta x = -\eta \nabla_x f \left(x, y\right)$ and $\Delta  y = \eta \nabla_y f\left(x, y\right)$ we obtain,
    
    \begin{align*}   
    &\left\| \nabla_x f \left(\Delta x + x, \Delta y + y\right) \right\|^2 + \left\| \nabla_y f\left(\Delta x + x, \Delta y + y\right) \right\|^2 
    -\|\nabla_x f (x,y)\|^2 - \|\nabla_y f(x,y)\|^2 \\
    &\quad\quad= -2\eta \nabla_x f (x,y)^{\top} \nabla_{xx}^2 f \nabla_x f (x,y) + 2 \eta \nabla_y f(x,y)^{\top} \nabla_{yy}^2 f  \nabla_y f(x,y) \\
    &\quad\quad\quad\quad+ 2 \eta^2 \nabla_y f\left(x, y\right)^{\top} \nabla_{yy}^2 f \nabla_{yx}^2 f \nabla_x f \left(x, y\right) - 2\eta^2  \nabla_x f \left(x, y\right)^{\top} \nabla_{xx}^2 f \nabla_{xy}^2 f \nabla_y f\left(x, y\right) \\ 
    &\quad\quad\quad\quad+ \underbrace{2 \eta \nabla_x f (x,y) ^{\top} \nabla_{xy}^2 f \nabla_y f(x,y)}_{{(i)}}- \underbrace{2\eta \nabla_y f(x,y) ^{\top} \nabla_{yx}^2 f \nabla_x f \left(x, y\right)}_{{(ii)}} \\ 
%pure terms        
    &\quad\quad\quad\quad+ \eta^2 \nabla_y f\left(x, y\right)^{\top} \nabla_{yy}^2 f \nabla_{yy}^2 f \nabla_y f\left(x, y\right)  + \eta^2 \nabla_x f (x,y)^{\top} \nabla_{xx}^2 f \nabla_{xx}^2 f \nabla_x f (x,y)\\
%Cross pure terms    
    &\quad\quad\quad\quad+  \eta^2 \nabla_y f\left(x, y\right)^{\top} \nabla_{yx}^2 f \nabla_{xy}^2 f \nabla_y f\left(x, y\right) + \eta^2 \nabla_x f \left(x, y\right)^{\top} \nabla_{xy}^2 f \nabla_{yx}^2 f \nabla_x f \left(x, y\right)\\
%Remainder terms
    &\quad\quad\quad\quad+ 2\nabla_x f (x,y) ^{\top} \R_x(\Delta x,\Delta y)  + 2\Delta x^{\top} \nabla_{xx}^2 f \R_x(\Delta x,\Delta y) + 2 \Delta y^{\top} \nabla_{yx}^2 f \R_x(\Delta x,\Delta y) + \|\R_x(\Delta x,\Delta y)\|^2\\
    &\quad\quad\quad\quad+ 2\nabla_y f(x,y) ^{\top} \R_y(\Delta x,\Delta y)  + 2\Delta y^{\top} \nabla_{yy}^2 f \R_y(\Delta x,\Delta y) + 2 \Delta x^{\top} \nabla_{xy}^2 f \R_y(\Delta x,\Delta y) + \|\R_y(\Delta x,\Delta y)\|^2
    \end{align*}   
     The terms $(i)$ and $(ii)$ in the RHS cancel out. Using the Cauchy-Schwarz inequality we obtain,
     
     \begin{align}\label{gdaboundremainder}
    2\nabla_x f(x,y) ^{\top} \R_x(\Delta x,\Delta y)&\leq 2\|\nabla_x f (x,y)\|\| \R_x(\Delta x,\Delta y)\|\nonumber\\
    2\nabla_y f(x,y) ^{\top} \R_y(\Delta x,\Delta y)&\leq
    2\|\nabla_y f (x,y)\|\| \R_y(\Delta x,\Delta y)\|
     \end{align}
    
    Using the upper bounds on $2\nabla_x f(x,y) ^{\top} \R_x(\Delta x,\Delta y)$ and $2\nabla_y f(x,y) ^{\top} \R_y(\Delta x,\Delta y)$ derived in Eq.~\eqref{gdaboundremainder} we obtain,
    \begin{align*}   
    % LHS
    &\left\| \nabla_x f \left(\Delta x + x, \Delta y + y\right) \right\|^2 + \left\| \nabla_y f\left(\Delta x + x, \Delta y + y\right) \right\|^2 
    -\|\nabla_x f (x,y)\|^2 - \|\nabla_y f(x,y)\|^2 \\
    % RHS
    % Linear Pure Terms
    &\quad\quad= -2\eta \nabla_x f (x,y)^{\top} \nabla_{xx}^2 f \nabla_x f (x,y) + 2 \eta \nabla_y f(x,y)^{\top} \nabla_{yy}^2 f  \nabla_y f(x,y) \\
    %Cross terms
    &\quad\quad\quad\quad+ 2 \eta^2 \nabla_y f\left(x, y\right)^{\top} \nabla_{yy}^2 f \nabla_{yx}^2 f \nabla_x f \left(x, y\right) - 2\eta^2  \nabla_x f \left(x, y\right)^{\top} \nabla_{xx}^2 f \nabla_{xy}^2 f \nabla_y f\left(x, y\right) \\ 
%pure terms        
    &\quad\quad\quad\quad+ 2\eta^2 \nabla_y f\left(x, y\right)^{\top} \nabla_{yy}^2 f \nabla_{yy}^2 f \nabla_y f\left(x, y\right)  + 2\eta^2 \nabla_x f (x,y)^{\top} \nabla_{xx}^2 f \nabla_{xx}^2 f \nabla_x f (x,y)\\
%Cross pure terms    
    &\quad\quad\quad\quad+  \eta^2 \nabla_y f\left(x, y\right)^{\top} \nabla_{yx}^2 f \nabla_{xy}^2 f \nabla_y f\left(x, y\right) + \eta^2 \nabla_x f \left(x, y\right)^{\top} \nabla_{xy}^2 f \nabla_{yx}^2 f \nabla_x f \left(x, y\right)\\
%Remainder terms
    &\quad\quad\quad\quad+2\|\nabla_x f \left(x, y\right)\|\|\R_x(\Delta x,\Delta y)\|+2\|\nabla_y f\left(x, y\right)\|\|\R_y(\Delta x,\Delta y)\|\\
    &\quad\quad\quad\quad+ 4\|\R_x(\Delta x,\Delta y)\|^2+ 4\|\R_y(\Delta x,\Delta y)\|^2
    \end{align*}   
Rearranging we obtain,

    \begin{align*}   
    % LHS
    &\left\| \nabla_x f \left(\Delta x + x, \Delta y + y\right) \right\|^2 + \left\| \nabla_y f\left(\Delta x + x, \Delta y + y\right) \right\|^2 
    -\|\nabla_x f (x,y)\|^2 - \|\nabla_y f(x,y)\|^2 \\
    % RHS
    % Linear Pure Terms
    &\quad\quad= \nabla_x f (x,y)^{\top}\left(\eta^2 \nabla_{xx}^2 f^2 -2\eta \nabla_{xx}^2 f+\eta^2 \nabla_{xy}^2 f \nabla_{yx}^2 f\right)\nabla_x f (x,y)\\
    &\quad\quad\quad\quad+ \nabla_y f(x,y)^{\top}\left(\eta^2 \nabla_{yy}^2 f^2 +2\eta \nabla_{yy}^2 f+\eta^2 \nabla_{yx}^2 f \nabla_{xy}^2 f\right)\nabla_y f(x,y)\\ 
    &\quad\quad\quad\quad+  2\eta^2 \nabla_x f \left(x, y\right)^{\top}\left(\nabla_{xy}^2 f \nabla_{yy}^2 f-\nabla_{xx}^2 f\nabla_{xy}^2 f \right)\nabla_y f\left(x, y\right) \\
    &\quad\quad\quad\quad+ 4\|\R_x(\Delta x,\Delta y)\|^2+ 4\|\R_y(\Delta x,\Delta y)\|^2+ 2\|\nabla_x f \left(x, y\right)\|\|\R_x(\Delta x,\Delta y)\|\\
    &\quad\quad\quad\quad+2\|\nabla_y f\left(x, y\right)\|\|\R_y(\Delta x,\Delta y)\|
    \end{align*}   
    
    To conclude, we need to bound the $\R$ terms.
    Using the Lipschitz-continuity of the Hessian and Eq.~\eqref{gdaremaindersa}, we can bound the remainder terms as,
    
    \begin{equation} 
        \|\R_x(\Delta x,\Delta y)\|, \|\R_y(\Delta x,\Delta y)\| \leq L_{xy}(\|\Delta x\| +\|\Delta y\|)^2
    \end{equation}
    Using Eq.~\eqref{eq:update-rule2}, we get,
    \begin{align*} 
        &\|\Delta x\|^2 +\|\Delta y\|^2 =  \eta^2(\|\nabla_x f (x,y)\|^2 + \|\nabla_y f(x,y)\|^2)
    \end{align*}    
    
    Hence we have,
    \begin{align*} 
        &L_{xy}(\|\Delta x\| +\|\Delta y\|)^2 \leq 2L_{xy}(\|\Delta x\|^2+\|\Delta y\|^2) \leq 2\eta^2L_{xy} (\|\nabla_x f (x,y)\|^2 + \|\nabla_y f(x,y)\|^2)
    \end{align*}    

    Thus,
    \begin{align*}   
    % LHS
    &\left\| \nabla_x f \left(\Delta x + x, \Delta y + y\right) \right\|^2 + \left\| \nabla_y f\left(\Delta x + x, \Delta y + y\right) \right\|^2 
    -\|\nabla_x f (x,y)\|^2 - \|\nabla_y f(x,y)\|^2 \\
    % RHS
    % Linear Pure Terms
    &\quad\quad\leq \nabla_x f (x,y)^{\top}\Big(\eta^2 \nabla_{xx}^2 f^2 -2\eta \nabla_{xx}^2 f+2\eta^2 \nabla_{xy}^2 f \nabla_{yx}^2 f \\
    &\quad\quad\quad\quad\quad\quad+4\eta^2L_{xy} (\|\nabla_x f (x,y)\| + \|\nabla_y f(x,y)\|)\Big)\nabla_x f (x,y)\\
    %%%%%%
    &\quad\quad\quad\quad+ \nabla_y f(x,y)^{\top}\Big(\eta^2 \nabla_{yy}^2 f^2 +2\eta \nabla_{yy}^2 f+2\eta^2 \nabla_{yx}^2 f \nabla_{xy}^2 f\\
    &\quad\quad\quad\quad\quad\quad+4\eta^2L_{xy} (\|\nabla_x f (x,y)\| + \|\nabla_y f(x,y)\|) \Big)\nabla_y f(x,y)\\ 
    %%%%%%
    &\quad\quad\quad\quad+  \underbrace{2\eta^2 \nabla_x f \left(x, y\right)^{\top}\left(\nabla_{xy}^2 f \nabla_{yy}^2 f-\nabla_{xx}^2 f\nabla_{xy}^2 f \right)\nabla_y f\left(x, y\right)}_{{(i)}} \\
    &\quad\quad\quad\quad+ 8\eta^2 L_{xy} (\|\nabla_x f (x,y)\|^2 + \|\nabla_y f(x,y)\|^2)
    \end{align*}   
    
    We further use the following inequality,
    $$a^{\top}Ab = \frac{1}{2}a^{\top}Ab+\frac{1}{2}b^{\top}A^\top a \stackrel{(a)}{\leq}\frac{1}{4}( a^{\top}(AA^{\top}+I)a+ b^{\top}(A^{\top}A+I)b)$$
    (where in $(a)$ we use the Peter-Paul inequality on both the terms) to bound the term $(i)$ in the above inequality. We obtain,
    \begin{align*}   
    % LHS
    &\left\| \nabla_x f \left(\Delta x + x, \Delta y + y\right) \right\|^2 + \left\| \nabla_y f\left(\Delta x + x, \Delta y + y\right) \right\|^2 
    -\|\nabla_x f (x,y)\|^2 - \|\nabla_y f(x,y)\|^2\\
    % RHS
    % Linear Pure Terms
    &\quad\quad=\nabla_x f (x,y)^{\top}\Big((\eta^2 \nabla_{xx}^2 f^2 -2\eta \nabla_{xx}^2 f+2\eta^2 \nabla_{xy}^2 f \nabla_{yx}^2 f)\\ &\quad\quad\quad\quad\quad\quad+I(8\eta^2L_{xy}+4\eta^2L_{xy} (\|\nabla_x f (x,y)\| + \|\nabla_y f(x,y)\|))\Big)\nabla_x f (x,y)\\
    %%%%%%%
    &\quad\quad\quad\quad+ \nabla_y f (x,y)^{\top}\Big((\eta^2 \nabla_{yy}^2 f^2 +2\eta \nabla_{yy}^2 f+2\eta^2 \nabla_{yx}^2 f \nabla_{xy}^2 f)\\
    &\quad\quad\quad\quad\quad\quad+I(8\eta^2L_{xy}+4\eta^2L_{xy} (\|\nabla_x f (x,y)\| + \|\nabla_y f(x,y)\|)) \Big)\nabla_y f(x,y)\\
    %%%%%%%
    &\quad\quad\quad\quad+\eta^2\nabla_x f (x,y)^{\top}\left((\nabla_{xy}^2 f \nabla_{yy}^2 f-\nabla_{xx}^2 f\nabla_{xy}^2 f )(\nabla_{xy}^2 f \nabla_{yy}^2 f-\nabla_{xx}^2 f\nabla_{xy}^2 f)^{\top} \right) \nabla_x f (x,y)\\
    &\quad\quad\quad\quad+\eta^2\nabla_y f (x,y)^{\top}\left((\nabla_{xy}^2 f \nabla_{yy}^2 f-\nabla_{xx}^2 f\nabla_{xy}^2 f )^{\top}(\nabla_{xy}^2 f \nabla_{yy}^2 f-\nabla_{xx}^2 f\nabla_{xy}^2 f) \right) \nabla_y f(x,y)\\
    &\quad\quad\quad\quad+\eta^2/2(\|\nabla_x f (x,y)\|^2 + \|\nabla_y f(x,y)\|^2)
    \end{align*}   
    
    This gives,
    \begin{align*}   
    % LHS
    &\left\| \nabla_x f \left(\Delta x + x, \Delta y + y\right) \right\|^2 + \left\| \nabla_y f\left(\Delta x + x, \Delta y + y\right) \right\|^2\leq(1-\lambda_{min}) (\|\nabla_x f (x,y)\|^2 + \|\nabla_y f(x,y)\|^2)
    \end{align*}       
%Tag fix parenthesis in the following
    Where,
    \begin{align*}
    \lambda_{min} &= \eta ~\min\Big\{\lambda_{min}(2 \nabla_{xx}^2 f- \eta(\nabla_{xx}^2 f^2 -2 \nabla_{xy}^2 f \nabla_{yx}^2 f-I(8L_{xy}+4L_{xy} (\|\nabla_x f \| + \|\nabla_y f\|)+\frac{1}{2})\\
    &\quad\quad\quad\quad-(\nabla_{xy}^2 f \nabla_{yy}^2 f-\nabla_{xx}^2 f\nabla_{xy}^2 f )(\nabla_{xy}^2 f \nabla_{yy}^2 f-\nabla_{xx}^2 f\nabla_{xy}^2 f)^{\top})),\\
    %%%%%
    &\quad\quad\lambda_{min}(-2 \nabla_{yy}^2 f- \eta(\nabla_{yy}^2 f^2 -2\nabla_{yx}^2 f \nabla_{xy}^2 f-I(8L_{xy}+4L_{xy} (\|\nabla_x f \| + \|\nabla_y f\|)+\frac{1}{2})\\
    &\quad\quad\quad\quad-(\nabla_{xy}^2 f \nabla_{yy}^2 f-\nabla_{xx}^2 f\nabla_{xy}^2 f )^{\top}(\nabla_{xy}^2 f \nabla_{yy}^2 f-\nabla_{xx}^2 f\nabla_{xy}^2 f)))\Big\}    
    \end{align*}
    Hence for $1 \geq \lambda_{min} > 0$ we have exponentially fast convergence. For sufficiently small $\eta$, we have convergence for all strongly convex-concave functions with rate $1-\lambda_{min}$ where,
    $$\lambda_{min} = \eta (\min\{\lambda_{min}(2 \nabla_{xx}^2 f),\lambda_{min}(-2 \nabla_{yy}^2 f)\})$$
%
%%%%%%%%%%%%%%%%%%%%%%%%%%%%%%
%%%%%%%%%%%%%%%%%%%%%%%%%%%%%%
%%%%%%%%%%%%%%%%%%%%%%%%%%%%%%

\section{Discrete time \cgo}\label{disctimecgo}
    In this section, we restate the update rule for the \cgo algorithm and then derive its convergence rate and a condition for convergence.
    Recall the update rule for \cgo,
    \begin{align*} 
    \begin{bmatrix}
    \Delta x\\
    \Delta y
    \end{bmatrix} &= -\eta \begin{bmatrix}
    I & \alpha \nabla_{xy}^2 f\\
    -\alpha \nabla_{yx}^2 f & I
    \end{bmatrix}^{-1}\begin{bmatrix}
    \nabla_x f \\
    -\nabla_y f 
    \end{bmatrix} 
    \end{align*}

    The following form of the above equation will be useful in the proof,
    \begin{align} \label{form1}
    \Delta x  = -\eta \nabla_x f -\alpha \nabla_{xy}^2 f\Delta y\nonumber\\
    \Delta y  = \eta \nabla_y f +\alpha \nabla_{yx}^2 f \Delta x 
    \end{align}

    Finally, writing the updates explicitly,

\begin{align}\label{form2}
&\Delta x = -\eta \left( I + \alpha^2 \nabla_{xy}^2 f \nabla_{yx}^2 f  \right)^{-1}  
            \left( \nabla_{x} f + \alpha \nabla_{xy}^2 f  \nabla_{y} f \right)\nonumber \\
&\Delta y = \eta \left( I + \alpha^2 \nabla_{yx}^2 f \nabla_{xy}^2 f \right)^{-1}  
            \left( \nabla_{y} f - \alpha \nabla_{yx}^2 f  \nabla_{x} f \right),
\end{align}

\begin{proof}[Proof of Theorem \eqref{disctheorem}]
Using the Taylor expansion of $(\nabla_x f,\nabla_y f)$, around the point $(x,y)$ we obtain, 
    
    \begin{align*}   
        & \nabla_x f (\Delta x + x, \Delta y + y) 
        = \nabla_x f ( x, y) + \nabla_{xx}^2 f \Delta x + \nabla_{xy}^2 f \Delta y 
        + \R_x(\Delta x,\Delta y)\\
        & \nabla_y f(\Delta x + x, \Delta y + y) = \nabla_y f( x,  y) + \nabla_{yy}^2 f\Delta y + \nabla_{yx}^2 f \Delta x + \R_y(\Delta x,\Delta y)
    \end{align*}   
    where the remainder terms $\R_x$ and $\R_y$ are defined as,
    \begin{align}
    &\R_x(\Delta x,\Delta y) \defeq \int \limits_0^1 \left( \left(\nabla_{xx}^2 f(t \Delta x + x, t \Delta y + y) - \nabla_{xx}^2 f\right) \Delta x + \left(\nabla_{xy}^2 f(t \Delta x + x, t \Delta y + y) - \nabla_{xy}^2 f\right) \Delta y \right) \mathrm{ d}t \label{cgoremaindersa}\\
    &\R_y(\Delta x,\Delta y) \defeq \int \limits_0^1 \left( \left(\nabla_{yy}^2 f(t \Delta x + x, t \Delta y + y) - \nabla_{yy}^2 f\right) \Delta y + \left(\nabla_{yx}^2 f(t \Delta x + x, t \Delta y + y) - \nabla_{yx}^2 f\right) \Delta x \right) \mathrm{d}t \label{cgoremaindersb} \end{align}   
    Using these equalities, we obtain the value of the difference between norm of the vector $(\nabla_x f,\nabla_y f)$ at points $(x,y)$ and updated ones, $(\Delta x+x_{k},\Delta y+y_{k})$.
    
    \begin{align}\label{maincomp}
    &\left\| \nabla_x f \left(\Delta x + x, \Delta y + y\right) \right\|^2 + \left\| \nabla_y f\left(\Delta x + x, \Delta y + y\right) \right\|^2 
    -\|\nabla_x f (x,y)\|^2 - \|\nabla_y f(x,y)\|^2 \nonumber\\
    &\quad\quad= 2 \Delta x^{\top} \nabla_{xx}^2 f \nabla_x f (x,y) +  2 \nabla_x f (x,y) ^{\top} \nabla_{xy}^2 f \Delta y + \Delta x^{\top} \nabla_{xx}^2 f \nabla_{xx}^2 f \Delta x + 2 \Delta x^{\top} \nabla_{xx}^2 f \nabla_{xy}^2 f \Delta y \nonumber\\
    &\quad\quad\quad\quad+ 2 \Delta y^{\top} \nabla_{yy}^2 f  \nabla_y f(x,y) +  2 \nabla_y f(x,y) ^{\top} \nabla_{yx}^2 f \Delta x + \Delta y^{\top} \nabla_{yy}^2 f \nabla_{yy}^2 f \Delta y + 2 \Delta y^{\top} \nabla_{yy}^2 f \nabla_{yx}^2 f \Delta x  \nonumber\\ 
    &\quad\quad\quad\quad+ \Delta y^{\top} \nabla_{yx}^2 f \nabla_{xy}^2 f \Delta y + 
\Delta x^{\top} \nabla_{xy}^2 f \nabla_{yx}^2 f \Delta x\nonumber\\
    &\quad\quad\quad\quad+ 2\nabla_x f (x,y) ^{\top} \R_x(\Delta x,\Delta y)  + 2\Delta x^{\top} \nabla_{xx}^2 f \R_x(\Delta x,\Delta y) + 2 \Delta y^{\top} \nabla_{yx}^2 f \R_x(\Delta x,\Delta y) + \|\R_x(\Delta x,\Delta y)\|^2\nonumber\\
    &\quad\quad\quad\quad+ 2\nabla_y f(x,y) ^{\top} \R_y(\Delta x,\Delta y)  + 2\Delta y^{\top} \nabla_{yy}^2 f \R_y(\Delta x,\Delta y) + 2 \Delta x^{\top} \nabla_{xy}^2 f \R_y(\Delta x,\Delta y) + \|\R_y(\Delta x,\Delta y)\|^2
    \end{align}
    We now observe using Eq.~\eqref{form1} that,
    
    \begin{align}  
    \Delta y^{\top} \nabla_{yx}^2 f \nabla_{xy}^2 f \Delta y &= -\Delta y^{\top} \nabla_{yx}^2 f \frac{\Delta x + \eta  \nabla_x f (x,y)}{\alpha}\label{sub1}\\ 
    \Delta x^{\top} \nabla_{xy}^2 f \nabla_{yx}^2 f \Delta x &= \Delta x^{\top} \nabla_{xy}^2 f \frac{\Delta y - \eta  \nabla_y f(x,y) }{\alpha}\label{sub2}
    \end{align}  
    Adding up Eq.~\eqref{sub1} and Eq.~\eqref{sub2} we obtain,
    \begin{equation*}
       \Delta x^{\top} \nabla_{xy}^2 f \nabla_{yx}^2 f \Delta x + \Delta y^{\top} \nabla_{yx}^2 f \nabla_{xy}^2 f \Delta y = - \frac{\eta}{\alpha} (\Delta y^{\top} \nabla_{yx}^2 f  \nabla_x f (x,y) + \Delta x^{\top} \nabla_{xy}^2 f  \nabla_y f(x,y))
    \end{equation*}
    Substituting this into Eq.~\eqref{maincomp} yields, 
    
    \begin{align*}  
    &\left\| \nabla_x f \left(\Delta x + x, \Delta y + y\right) \right\|^2 + \left\| \nabla_y f\left(\Delta x + x, \Delta y + y\right) \right\|^2
    -\|\nabla_x f (x,y)\|^2 - \|\nabla_y f(x,y)\|^2 \\
    &\quad\quad= 2 \Delta x^{\top} \nabla_{xx}^2 f \nabla_x f (x,y) +  (2-\frac{\eta}{\alpha})\nabla_x f (x,y) ^{\top} \nabla_{xy}^2 f \Delta y + \Delta x^{\top} \nabla_{xx}^2 f \nabla_{xx}^2 f \Delta x + 2 \Delta x^{\top} \nabla_{xx}^2 f \nabla_{xy}^2 f \Delta y   \\
    &\quad\quad\quad\quad+ 2 \Delta y^{\top} \nabla_{yy}^2 f  \nabla_y f(x,y) +  (2-\frac{\eta}{\alpha})\nabla_y f(x,y) ^{\top} \nabla_{yx}^2 f \Delta x + \Delta y^{\top} \nabla_{yy}^2 f \nabla_{yy}^2 f \Delta y + 2 \Delta y^{\top} \nabla_{yy}^2 f \nabla_{yx}^2 f \Delta x \\ 
    &\quad\quad\quad\quad+ 2\nabla_x f (x,y) ^{\top} \R_x(\Delta x,\Delta y)  + 2\Delta x^{\top} \nabla_{xx}^2 f \R_x(\Delta x,\Delta y) + 2 \Delta y^{\top} \nabla_{yx}^2 f \R_x(\Delta x,\Delta y) + \|\R_x(\Delta x,\Delta y)\|^2\\
    &\quad\quad\quad\quad+ 2\nabla_y f(x,y) ^{\top} \R_y(\Delta x,\Delta y)  + 2\Delta y^{\top} \nabla_{yy}^2 f \R_y(\Delta x,\Delta y) + 2 \Delta x^{\top} \nabla_{xy}^2 f \R_y(\Delta x,\Delta y) + \|\R_y(\Delta x,\Delta y)\|^2
    \end{align*}  
    We use the update rule of \cgo, Eq.~\eqref{form2} to substitute $\Delta x$ and $\Delta y$ and observe that $\nabla_{yx}^2 f(I + \nabla_{xy}^2 f\nabla_{yx}^2 f )^{-1} = (I + \nabla_{yx}^2 f \nabla_{xy}^2 f)^{-1}\nabla_{yx}^2 f$ as stated in Lemma~\eqref{lemma:inverse-transpose} to obtain the following equality,
    \begin{align*}   
         &\Delta x^{\top} \nabla_{xy}^2 f  \nabla_y f(x,y) + \nabla_x f (x,y) ^{\top} \nabla_{xy}^2 f \Delta y \\
         &\quad\quad=-\eta\alpha \nabla_x f (x,y) ^{\top} \left( I + \alpha^2 \nabla_{xy}^2 f\nabla_{yx}^2 f  \right)^{-1} \nabla_{xy}^2 f\nabla_{yx}^2 f   \nabla_x f (x,y)\\
         &\quad\quad\quad\quad- \eta \alpha \nabla_y f(x,y) ^{\top}\left(I + \alpha^2 \nabla_{yx}^2 f\nabla_{xy}^2 f\right)^{-1} \nabla_{yx}^2 f\nabla_{xy}^2 f  \nabla_y f(x,y).
    \end{align*}
    Yielding,
    \begin{align*}   
    &\left\| \nabla_x f \left(\Delta x + x, \Delta y + y\right) \right\|^2 + \left\| \nabla_y f\left(\Delta x + x, \Delta y + y\right) \right\|^2 
    -\|\nabla_x f (x,y)\|^2 - \|\nabla_y f(x,y)\|^2 \\
    &= 2 \Delta x^{\top} \nabla_{xx}^2 f \nabla_x f (x,y) -\eta \alpha(2-\frac{\eta}{\alpha})\nabla_x f (x,y) ^{\top} \left(I + \alpha^2 \nabla_{xy}^2 f \nabla_{yx}^2 f \right)^{-1} \nabla_{xy}^2 f \nabla_{yx}^2 f   \nabla_x f (x,y) \\
    &\quad\quad\quad\quad+ 2 \Delta x^{\top} \nabla_{xx}^2 f \nabla_{xy}^2 f \Delta y  + 2 \Delta y^{\top} \nabla_{yy}^2 f \nabla_{yx}^2 f \Delta x +  \Delta x^{\top} \nabla_{xx}^2 f \nabla_{xx}^2 f \Delta x + \Delta y^{\top} \nabla_{yy}^2 f \nabla_{yy}^2 f \Delta y  \\
    &\quad\quad\quad\quad+ 2 \Delta y^{\top} \nabla_{yy}^2 f  \nabla_y f(x,y)  -\eta \alpha(2-\frac{\eta}{\alpha})\nabla_y f(x,y) ^{\top} \left(I + \alpha^2 \nabla_{yx}^2 f\nabla_{xy}^2 f\right)^{-1} \nabla_{yx}^2 f\nabla_{xy}^2 f  \nabla_y f(x,y) \\ 
    &\quad\quad\quad\quad+ 2\nabla_x f (x,y) ^{\top} \R_x(\Delta x,\Delta y)  + 2\Delta x^{\top} \nabla_{xx}^2 f \R_x(\Delta x,\Delta y) + 2 \Delta y^{\top} \nabla_{yx}^2 f \R_x(\Delta x,\Delta y) + \|\R_x(\Delta x,\Delta y)\|^2\\
    &\quad\quad\quad\quad+ 2\nabla_y f(x,y) ^{\top} \R_y(\Delta x,\Delta y)  + 2\Delta y^{\top} \nabla_{yy}^2 f \R_y(\Delta x,\Delta y) + 2 \Delta x^{\top} \nabla_{xy}^2 f \R_y(\Delta x,\Delta y) + \|\R_y(\Delta x,\Delta y)\|^2.
    \end{align*}
    We now substitute $\Delta x$ and $\Delta y$ using Eq.~\eqref{form1} yielding,
    
    \begin{align}   
    &\left\| \nabla_x f \left(\Delta x + x, \Delta y + y \right)\right\|^2 + \left\|\nabla_y f\left(\Delta x + x, \Delta y + y \right)\right\|^2 
    -\|\nabla_x f (x,y)\|^2 - \|\nabla_y f(x,y)\|^2\nonumber \\
    &= - 2 \eta \nabla_x f (x,y) ^{\top} \nabla_{xx}^2 f \nabla_x f (x,y)+ 2 \eta \nabla_y f(x,y) ^{\top} \nabla_{yy}^2 f  \nabla_y f(x,y)\nonumber\\ 
    &\quad\quad-\eta \alpha(2-\frac{\eta}{\alpha})\nabla_x f (x,y) ^{\top} \left(I + \alpha^2 \nabla_{xy}^2 f \nabla_{yx}^2 f \right)^{-1} \nabla_{xy}^2 f \nabla_{yx}^2 f   \nabla_x f (x,y) \nonumber\\
    &\quad\quad+ \Delta x^{\top} \nabla_{xx}^2 f \nabla_{xx}^2 f \Delta x+\frac{2}{\eta} \left(\underbrace{(\alpha+\eta)\Delta x}_{{(i)}} + \alpha^2 \nabla_{xy}^2 f \Delta y \right)^{\top}  \nabla_{xx}^2 f \nabla_{xy}^2 f \Delta y\nonumber\\
    &\quad\quad+ \Delta y^{\top} \nabla_{yy}^2 f \nabla_{yy}^2 f \Delta y+\frac{2}{\eta} \left(\underbrace{(\alpha+\eta)\Delta y}_{{(ii)}} - \alpha^2 \nabla_{yx}^2 f \Delta x \right)^{\top} \nabla_{yy}^2 f \nabla_{yx}^2 f \Delta x   \nonumber\\
    &\quad\quad-\eta \alpha(2-\frac{\eta}{\alpha})\nabla_y f(x,y) ^{\top} \left(I + \alpha^2 \nabla_{yx}^2 f\nabla_{xy}^2 f\right)^{-1} \nabla_{yx}^2 f\nabla_{xy}^2 f  \nabla_y f(x,y)
     \nonumber\\ 
    &\quad\quad+ 2\nabla_x f (x,y) ^{\top} \R_x(\Delta x,\Delta y)  + \underbrace{2\Delta x^{\top} \nabla_{xx}^2 f \R_x(\Delta x,\Delta y)}_{{(iii)}} + 2 \Delta y^{\top} \nabla_{yx}^2 f \R_x(\Delta x,\Delta y) + \|\R_x(\Delta x,\Delta y)\|^2\nonumber\\
    &\quad\quad+ 2\nabla_y f(x,y) ^{\top} \R_y(\Delta x,\Delta y)  + \underbrace{2\Delta y^{\top} \nabla_{yy}^2 f \R_y(\Delta x,\Delta y)}_{{(iv)}} + 2 \Delta x^{\top} \nabla_{xy}^2 f \R_y(\Delta x,\Delta y) + \|\R_y(\Delta x,\Delta y)\|^2
    \end{align}

    Now we use Peter-Paul inequality, and bound the terms $(i)$ and $(ii)$ respectively as follows,
    \begin{align}\label{PP1}
    \frac{2(\alpha+\eta)}{\eta} \Delta x^{\top} \nabla_{xx}^2 f \nabla_{xy}^2 f \Delta y &\leq 8\frac{\alpha+\eta}{\eta}\| \Delta x^{\top}  \nabla_{xx}^2f\|^2+ \frac{\alpha+\eta}{8\eta}\|\nabla_{xy}^2 f \Delta y\|^2\nonumber\\     
    \frac{2(\alpha+\eta)}{\eta} \Delta y^{\top} \nabla_{yy}^2 f \nabla_{yx}^2 f \Delta x &\leq  8\frac{\alpha+\eta}{\eta}\| \Delta y^{\top}  \nabla_{yy}^2f\|^2+ \frac{\alpha+\eta}{8\eta}\|\nabla_{yx}^2 f \Delta x\|^2
    \end{align}
    and terms $(iii)$ and $(iv)$ as,
    
    \begin{align}\label{PP2}
    2 \Delta x^{\top} \nabla_{xy}^2 f \R_y(\Delta x,\Delta y)&\leq  \|\R_y(\Delta x,\Delta y)\|^2+ \|\nabla_{xy}^2 f \Delta y \|^2 \nonumber\\
    2 \Delta y^{\top} \nabla_{yx}^2 f \R_x(\Delta x,\Delta y)&\leq  \|\R_x(\Delta x,\Delta y)\|^2+ \|\nabla_{yx}^2 f \Delta x\|^2
    \end{align}
    Using the bounds obtained in Eqs.~\eqref{PP1} and \eqref{PP2} we get,  
    
    \begin{align*}   
    &\left\| \nabla_x f \left(\Delta x + x, \Delta y + y \right)\right\|^2 + \left\|\nabla_y f\left(\Delta x + x, \Delta y + y \right)\right\|^2 
    -\|\nabla_x f (x,y)\|^2 - \|\nabla_y f(x,y)\|^2 \\
    &\quad\quad\leq - 2 \eta \nabla_x f (x,y) ^{\top} \nabla_{xx}^2 f \nabla_x f (x,y)\\
    &\quad\quad\quad\quad-\eta \alpha(2-\frac{\eta}{\alpha})\nabla_x f (x,y) ^{\top} \left(I + \alpha^2 \nabla_{xy}^2 f \nabla_{yx}^2 f \right)^{-1} \nabla_{xy}^2 f \nabla_{yx}^2 f   \nabla_x f (x,y)  \\
    &\quad\quad\quad\quad+\frac{10\eta+8\alpha}{\eta} \Delta x^{\top} \nabla_{xx}^2 f \nabla_{xx}^2 f \Delta x 
    +\Delta y^{\top} \nabla_{yx}^2 f \left(\frac{\alpha+\eta}{8\eta} I + \frac{2\alpha^2 \nabla_{xx}^2 f}{\eta} \right) \nabla_{xy}^2 f \Delta y   \\
    &\quad\quad\quad\quad+ 2 \eta \nabla_y f(x,y) ^{\top} \nabla_{yy}^2 f  \nabla_y f(x,y)\\
    &\quad\quad\quad\quad-\eta \alpha(2-\frac{\eta}{\alpha})\nabla_y f(x,y) ^{\top} \left(I + \alpha^2 \nabla_{yx}^2 f\nabla_{xy}^2 f\right)^{-1} \nabla_{yx}^2 f\nabla_{xy}^2 f  \nabla_y f(x,y) \\
    &\quad\quad\quad\quad+ \frac{10\eta+8\alpha}{\eta} \Delta y^{\top} \nabla_{yy}^2 f \nabla_{yy}^2 f \Delta y 
    + \Delta x^{\top} \nabla_{xy}^2 f \left(\frac{\alpha+\eta}{8\eta} I - \frac{2\alpha^2 \nabla_{yy}^2 f}{\eta} \right) \nabla_{yx}^2 f \Delta x\\
    &\quad\quad\quad\quad+ 2\nabla_x f (x,y) ^{\top} \R_x(\Delta x,\Delta y) + \underbrace{2 \Delta y^{\top} \nabla_{yx}^2 f \R_x(\Delta x,\Delta y)}_{{(i)}} + 2\|\R_x(\Delta x,\Delta y)\|^2\\
    &\quad\quad\quad\quad+ 2\nabla_y f(x,y) ^{\top} \R_y(\Delta x,\Delta y) + \underbrace{2 \Delta x^{\top} \nabla_{xy}^2 f \R_y(\Delta x,\Delta y)}_{{(ii)}} + 2\|\R_y(\Delta x,\Delta y)\|^2.
    \end{align*}
    
    We use the Peter-Paul inequality to bound the term $(i)$ as,
    
    $$2 \Delta y^{\top} \nabla_{yx}^2 f \R_x(\Delta x,\Delta y)\leq 4 \|\R_x(\Delta x,\Delta y)\|^2+\frac{1}{4} \|\nabla_{yx}^2 f \Delta x\|^2 $$
    and the term $(ii)$ as, 
    
    $$2 \Delta x^{\top} \nabla_{xy}^2 f \R_y(\Delta x,\Delta y)\leq 4 \|\R_y(\Delta x,\Delta y)\|^2+\frac{1}{4} \|\nabla_{xy}^2 f \Delta y \|^2$$
    Substituting the above obtained bounds and noting that $\nabla_{xx}f$ and $\nabla_{yy}f$ are symmetric matrices we obtain,
    
     \begin{align*}   
    &\left\| \nabla_x f \left(\Delta x + x, \Delta y + y \right)\right\|^2 + \left\|\nabla_y f\left(\Delta x + x, \Delta y + y \right)\right\|^2 
    -\|\nabla_x f (x,y)\|^2 - \|\nabla_y f(x,y)\|^2 \\
    &\quad\quad\leq - 2 \eta \nabla_x f (x,y) ^{\top} \nabla_{xx}^2 f \nabla_x f (x,y) \\
    &\quad\quad\quad\quad-\eta \alpha(2-\frac{\eta}{\alpha})\nabla_x f (x,y) ^{\top} \left(I + \alpha^2 \nabla_{xy}^2 f \nabla_{yx}^2 f \right)^{-1} \nabla_{xy}^2 f \nabla_{yx}^2 f   \nabla_x f (x,y) \\
    &\quad\quad\quad\quad+\frac{10\eta+8\alpha}{\eta} \Delta x^{\top} \nabla_{xx}^2 f \nabla_{xx}^2 f \Delta x + \left(\frac{\alpha+3\eta}{8\eta} +  \frac{2\alpha^2\underline{\lambda_{xx}}}{\eta}\right) \|\nabla_{xy}^2 f \Delta y \|^2  \\
    &\quad\quad\quad\quad+ 2 \eta \nabla_y f(x,y) ^{\top} \nabla_{yy}^2 f  \nabla_y f(x,y)  \\
    &\quad\quad\quad\quad-\eta \alpha(2-\frac{\eta}{\alpha})\nabla_y f(x,y) ^{\top} \left(I + \alpha^2 \nabla_{yx}^2 f\nabla_{xy}^2 f\right)^{-1} \nabla_{yx}^2 f\nabla_{xy}^2 f  \nabla_y f(x,y) \\
    &\quad\quad\quad\quad+ \frac{10\eta+8\alpha}{\eta} \Delta y^{\top} \nabla_{yy}^2 f \nabla_{yy}^2 f \Delta y +  \left(\frac{\alpha+3\eta}{8\eta} - \frac{2\alpha^2\underline{\lambda_{yy}}}{\eta} \right) \|\nabla_{yx}^2 f \Delta x\|^2 \\ 
    &\quad\quad\quad\quad+ 2\nabla_x f (x,y) ^{\top} \R_x(\Delta x,\Delta y) + 6\|\R_x(\Delta x,\Delta y)\|^2\\
    &\quad\quad\quad\quad+ 2\nabla_y f(x,y) ^{\top} \R_y(\Delta x,\Delta y) + 6\|\R_y(\Delta x,\Delta y)\|^2
    \end{align*}
    Using Eq.~\eqref{form1} to substitute $\Delta x$ and $\Delta y$ we compute, 
    \begin{align*}   
        \left\|\nabla_{yx}^2 f \Delta x \right\|^2&= \eta^2\left(\nabla_x f (x,y) + \nabla_{xy}^2 f \nabla_y f(x,y)\right)^{\top} \\
        &\quad\quad\quad\quad\left( I + \nabla_{xy}^2 f\nabla_{yx}^2 f  \right)^{-2} \nabla_{xy}^2 f\nabla_{yx}^2 f  \left(\nabla_x f (x,y) + \nabla_{xy}^2 f \nabla_y f(x,y)\right)
    \end{align*}
    And,
    \begin{align*}
        \left\|\nabla_{xy}^2 f \Delta y \right\|^2 &= \eta^2\left(-\nabla_y f(x,y) + \nabla_{yx}^2 f \nabla_x f (x,y)\right)^{\top} \\
        &\quad\quad\quad\quad\left( I + \nabla_{yx}^2 f\nabla_{xy}^2 f  \right)^{-2} \nabla_{yx}^2 f\nabla_{xy}^2 f  \left(-\nabla_y f(x,y) + \nabla_{yx}^2 f \nabla_x f (x,y)\right)
    \end{align*}    
    By adding up the two, we obtain,
    
    \begin{align}   \label{crosssum}
        \left\|\nabla_{yx}^2 f \Delta x\right\|^2 + \left\|\nabla_{xy}^2 f \Delta y\right\|^2 &=
        \eta^2\nabla_x f (x,y) ^{\top} \left( I + \nabla_{xy}^2 f\nabla_{yx}^2 f  \right)^{-2} \left(\nabla_{xy}^2 f\nabla_{yx}^2 f  +\nabla_{xy}^2 f\nabla_{yx}^2 f \right) \nabla_x f (x,y)\nonumber\\
        &\quad\quad+ \eta^2\nabla_y f(x,y) ^{\top} \left( I + \nabla_{yx}^2 f\nabla_{xy}^2 f  \right)^{-2} \left(\nabla_{yx}^2 f\nabla_{xy}^2 f  +\nabla_{yx}^2 f\nabla_{xy}^2 f \right) \nabla_y f(x,y)\nonumber\\
        &= \eta^2 \nabla_x f (x,y) ^{\top} \left(I + \alpha^2 \nabla_{xy}^2 f \nabla_{yx}^2 f \right)^{-1} \nabla_{xy}^2 f \nabla_{yx}^2 f   \nabla_x f (x,y)\nonumber \\ 
        &\quad\quad+\eta^2 \nabla_y f(x,y) ^{\top} \left(I + \alpha^2 \nabla_{yx}^2 f\nabla_{xy}^2 f   \right)^{-1} \nabla_{yx}^2 f\nabla_{xy}^2 f  \nabla_y f(x,y)
    \end{align}
    Setting $\overline{\lambda_1} = \max(\overline{\lambda_{xx}},-\underline{\lambda_{yy}})$ and using Eq.~\eqref{crosssum} to substitute $\|\nabla_{xy}^2 f \Delta y\|^2+\|\nabla_{yx}^2 f\Delta x\|^2$, we have,
    
    \begin{align}\label{maincompute}
    &\left\| \nabla_x f \left(\Delta x + x, \Delta y + y \right)\right\|^2 + \left\|\nabla_y f\left(\Delta x + x, \Delta y + y \right)\right\|^2 
    -\|\nabla_x f (x,y)\|^2 - \|\nabla_y f(x,y)\|^2\nonumber \\
    &\quad\quad\leq - 2 \eta \nabla_x f (x,y) ^{\top} \nabla_{xx}^2 f \nabla_x f (x,y)  +\frac{10\eta+8\alpha}{\eta} \underbrace{\Delta x^{\top} \nabla_{xx}^2 f \nabla_{xx}^2 f \Delta x}_{{(i)}}\nonumber\\
    &\quad\quad\quad\quad+\eta \left(\frac{11\eta+16\alpha^2\overline{\lambda_1}}{8} - \frac{15}{8}\alpha\right)\nabla_x f (x,y) ^{\top} \left(I + \alpha^2 \nabla_{xy}^2 f \nabla_{yx}^2 f \right)^{-1} \nabla_{xy}^2 f \nabla_{yx}^2 f   \nabla_x f (x,y)\nonumber \\
    &\quad\quad\quad\quad+ 2 \eta \nabla_y f(x,y) ^{\top} \nabla_{yy}^2 f  \nabla_y f(x,y)  + \frac{10\eta+8\alpha}{\eta} \underbrace{\Delta y^{\top} \nabla_{yy}^2 f \nabla_{yy}^2 f \Delta y}_{{(ii)}}\nonumber\\
    &\quad\quad\quad\quad+\eta \left(\frac{11\eta+16\alpha^2\overline{\lambda_1}}{8} - \frac{15}{8}\alpha\right)\nabla_y f(x,y) ^{\top} \left(I + \alpha^2 \nabla_{yx}^2 f\nabla_{xy}^2 f\right)^{-1} \nabla_{yx}^2 f\nabla_{xy}^2 f  \nabla_y f(x,y)  \nonumber\\ 
    &\quad\quad\quad\quad+ 2\nabla_x f (x,y) ^{\top} \R_x(\Delta x,\Delta y) + 6\|\R_x(\Delta x,\Delta y)\|^2\nonumber\\
    &\quad\quad\quad\quad+ 2\nabla_y f(x,y) ^{\top} \R_y(\Delta x,\Delta y) + 6\|\R_y(\Delta x,\Delta y)\|^2
    \end{align}
    Substituting $\Delta x$ and $\Delta y$ from Eq.~\eqref{form1} we bound the sum of terms $(i)$ and $(ii)$ as follows,
    
    \begin{align*} 
    \Delta x^{\top} \nabla_{xx}^2 f \nabla_{xx}^2 f \Delta x&+\Delta y^{\top} \nabla_{yy}^2 f \nabla_{yy}^2 f \Delta y \\&= (-\alpha \nabla_{xy}^2 f\Delta y-\eta  \nabla_x f (x,y))^{\top}  \nabla_{xx}^2 f \nabla_{xx}^2 f  (-\alpha \nabla_{xy}^2 f\Delta y-\eta  \nabla_x f (x,y))\\
    &\quad\quad+(\alpha \nabla_{yx}^2 f\Delta x+\eta  \nabla_y f(x,y))^{\top}\nabla_{yy}^2 f \nabla_{yy}^2 f (\alpha \nabla_{yx}^2 f \Delta x+\eta  \nabla_y f(x,y)) \\
    &= \|\nabla_{xx}^2 f(\nabla_{xy}^2 f \Delta y+\eta  \nabla_x f (x,y))\|^2\\
    &\quad\quad+ \|\nabla_{yy}^2 f(\alpha \nabla_{yx}^2 f\Delta x+\eta  \nabla_y f(x,y))\|^2 \\
    &\leq 2\alpha^2 \|\nabla_{xy}^2 f \Delta y\|^2  \|\nabla_{xx}^2 f\|^2 + 2\alpha^2  \|\nabla_{yx}^2 f \Delta x\|^2\|\nabla_{yy}^2 f\|^2\\
    &\quad\quad+ 2\eta^2\|\nabla_{xx}^2 f\nabla_x f (x,y)\|^2+2\eta^2\|\nabla_{yy}^2 f\nabla_y f(x,y) \|^2\\
    &= 2\alpha^2 \|\nabla_{xy}^2 f \Delta y\|^2  \overline{\lambda_{xx}}^2 + 2\alpha^2  \|\nabla_{yx}^2 f \Delta x\|^2\overline{\lambda_{yy}}^2\\
    &\quad\quad+ 2\eta^2\|\nabla_{xx}^2 f\nabla_x f (x,y)\|^2+2\eta^2\|\nabla_{yy}^2 f\nabla_y f(x,y) \|^2\\
    \end{align*}
    Setting $\overline{\lambda_2} = \max(\overline{\lambda_{xx}},\overline{\lambda_{yy}})$ and using Eq.~\eqref{crosssum} to substitute $\|\nabla_{xy}^2 f \Delta y\|^2+\|\nabla_{yx}^2 f\Delta x\|^2$,
    \begin{align*}     
    \Delta x^{\top} \nabla_{xx}^2 f \nabla_{xx}^2 f \Delta x&+\Delta y^{\top} \nabla_{yy}^2 f \nabla_{yy}^2 f \Delta y\\
    &\leq 2\alpha^2 \overline{\lambda_2}^2(\|\nabla_{xy}^2 f \Delta y\|^2   + \|\nabla_{yx}^2 f \Delta x\|^2) \\
    &\quad\quad+ 2\eta^2\|\nabla_{xx}^2 f\nabla_x f (x,y)\|^2+2\eta^2\|\nabla_{yy}^2 f\nabla_y f(x,y) \|^2\\ 
    &\leq 2\alpha^2 \eta^2 \overline{\lambda_2}^2 \nabla_x f (x,y) ^{\top} \left(I + \alpha^2 \nabla_{xy}^2 f\nabla_{yx}^2 f \right)^{-1} \nabla_{xy}^2 f\nabla_{yx}^2 f \nabla_x f (x,y)\\  
    &\quad\quad+2\alpha^2 \eta^2 \overline{\lambda_2}^2 \nabla_y f(x,y) ^{\top} \left(I + \alpha^2 \nabla_{yx}^2 f\nabla_{xy}^2 f   \right)^{-1} \nabla_{yx}^2 f\nabla_{xy}^2 f  \nabla_y f(x,y)\\
    &\quad\quad+2\overline{\lambda_{xx}}^2  \nabla_x f (x,y) ^{\top} \nabla_x f (x,y) \\
    &\quad\quad+2\overline{\lambda_{yy}}^2  \nabla_y f(x,y) ^{\top} \nabla_y f(x,y) 
    \end{align*} 
    
    Substituting the above bound in Eq.~\eqref{maincompute} we obtain,
    \begin{align}\label{maincomputation}   
    &\left\| \nabla_x f \left(\Delta x + x, \Delta y + y \right)\right\|^2 + \left\|\nabla_y f\left(\Delta x + x, \Delta y + y \right)\right\|^2 
    -\|\nabla_x f (x,y)\|^2 - \|\nabla_y f(x,y)\|^2 \nonumber\\
    &\quad\quad\leq -  \eta \nabla_x f (x,y) ^{\top} \left(2\nabla_{xx}^2 f-2\frac{10\eta+8\alpha}{\eta}\overline{\lambda_{xx}}^2 \right) \nabla_x f (x,y)\nonumber\\
    &\quad\quad\quad\quad+ \eta \nabla_y f(x,y) ^{\top} \left(2\nabla_{yy}^2 f + 2\frac{10\eta+8\alpha}{\eta}\overline{\lambda_{yy}}^2\right) \nabla_y f(x,y)\nonumber\\
    &\quad\quad\quad\quad+ \left(\eta (\frac{11\eta+16\alpha^2\overline{\lambda_1}}{8} - \frac{15}{8}\alpha)+2(10\eta+8\alpha)\alpha^2\eta\overline{\lambda_2}^2\right)\nonumber\\
    &\quad\quad\quad\quad\quad\quad\quad\quad(\nabla_y f(x,y) ^{\top} \left(I + \alpha^2 \nabla_{yx}^2 f\nabla_{xy}^2 f\right)^{-1} \nabla_{yx}^2 f\nabla_{xy}^2 f  \nabla_y f(x,y)\nonumber\\
    &\quad\quad\quad\quad+\nabla_x f (x,y) ^{\top} \left(I + \alpha^2 \nabla_{xy}^2 f\nabla_{yx}^2 f \right)^{-1} \nabla_{xy}^2 f\nabla_{yx}^2 f   \nabla_x f (x,y))\nonumber\\ 
    &\quad\quad\quad\quad+ 2\nabla_x f (x,y) ^{\top} \R_x(\Delta x,\Delta y) + 6\|\R_x(\Delta x,\Delta y)\|^2\nonumber\\
    &\quad\quad\quad\quad+ 2\nabla_y f(x,y) ^{\top} \R_y(\Delta x,\Delta y) + 6\|\R_y(\Delta x,\Delta y)\|^2.
    \end{align}
    
    To conclude, we need to bound the $\R$-terms.
    Using the Lipschitz-continuity of the Hessian, and equations Eq.~\eqref{cgoremaindersa} and Eq.~\eqref{cgoremaindersb} we can bound,
    
    \begin{align}\label{remainderineq}
        \|\R_x(\Delta x,\Delta y)\|, \|\R_y(\Delta x,\Delta y)\| \leq L_{xy}(\|\Delta x\| +\|\Delta y\|)^2
    \end{align}

    Using Eq.~\eqref{form1} we get,
    
    \begin{align*} 
        (\|\Delta x\|^2 +\|\Delta y\|^2) =&  \eta^2(\| \nabla_x f (x,y)\|^2+\|\nabla_y f(x,y)\|^2)\\&+2\eta\alpha(\nabla_x f (x,y)^{\top}\nabla_{yx}^2 f\Delta x+ \nabla_y f(x,y)^{\top}\nabla_{xy}^2 f\Delta y) \\&+\alpha^2(\|\nabla_{yx}^2 f\Delta x\|^2+\|\nabla_{xy}^2 f\Delta y\|^2)
    \end{align*}
    
    From Eq.~\eqref{crosssum} we have,
    
    \begin{align}\label{cgoprelimbound}
        \alpha^2(\left\|\nabla_{yx}^2 f \Delta x\right\|^2 + \left\|\nabla_{xy}^2 f \Delta y\right\|^2) =& \eta^2 \nabla_x f (x,y) ^{\top} \left(I + \alpha^2 \nabla_{xy}^2 f \nabla_{yx}^2 f \right)^{-1} \alpha^2\nabla_{xy}^2 f \nabla_{yx}^2 f   \nabla_x f (x,y)\nonumber\\
        &\quad\quad+\eta^2 \nabla_y f(x,y) ^{\top} \left(I + \alpha^2 \nabla_{yx}^2 f\nabla_{xy}^2 f   \right)^{-1}\alpha^2 \nabla_{yx}^2 f\nabla_{xy}^2 f  \nabla_y f(x,y)\nonumber\\ 
        \leq& \eta^2(\|\nabla_x f (x,y)\|^2+\|\nabla_y f(x,y)\|^2)
        \end{align}

     And observe,
     
     \begin{align}\label{intermediaterem}
          \nabla_x f (x,y)^{\top}\nabla_{yx}^2 f\Delta x+\nabla_y f(x,y)^{\top}\nabla_{xy}^2 f\Delta y &= (\nabla_x f (x,y),\nabla_y f(x,y))^{\top}(\nabla_{yx}^2 f\Delta x,\nabla_{xy}^2 f\Delta y)\nonumber\\
          &\stackrel{(c)}{\leq} \|(\nabla_x f (x,y),\nabla_y f(x,y)\|\|(\nabla_{yx}^2 f\Delta x,\nabla_{xy}^2 f\Delta y)\|\nonumber\\
          &\stackrel{(d)}{\leq}\frac{\eta}{\alpha} (\|\nabla_x f (x,y)\|^2+\|\nabla_y f(x,y)\|^2)   
     \end{align}

    Where in $(c)$ we use the Cauchy-Schwarz inequality and in $(d)$ we use the bound derived in Eq.~\eqref{cgoprelimbound}. We then substitute $\Delta x$ and $\Delta y$ using Eq.~\eqref{form2} to obtain,
        \begin{align} \label{remainder} 
        L_{xy}(\|\Delta x\| +\|\Delta y\|)^2 &\leq 2L_{xy}(\|\Delta x\|^2 +\|\Delta y\|^2)\nonumber\\
        &\leq 2L_{xy}(\eta^2(\|\nabla_x f (x,y)\|^2+\|\nabla_y f(x,y)\|^2)\nonumber\\
        &\quad\quad+2\eta\alpha\underbrace{(\nabla_x f (x,y)^{\top}\nabla_{yx}^2 f\Delta x+\nabla_y f(x,y)^{\top}\nabla_{xy}^2 f\Delta y)}_{{(i)}}\nonumber\\
        &\quad\quad+\alpha^2\underbrace{(\|\nabla_{yx}^2 \Delta x\|^2+\|\nabla_{xy}^2 f\Delta y\|^2))}_{{(ii)}}\nonumber\\
        &\stackrel{(e)}{\leq}  8\eta^2 L_{xy} (\|\nabla_x f (x,y)\|^2 + \|\nabla_y f(x,y)\|^2)
    \end{align}
    
    Where in $(e)$ we have used Eq.~\eqref{intermediaterem} to bound term $(i)$ and Eq.~\eqref{cgoprelimbound} to bound $(ii)$. Combining Eq.~\eqref{remainderineq} and Eq.~\eqref{remainder} we obtain,

    \begin{align}\label{cgoremainderbounda}
        \|\R_x(\Delta x,\Delta y)\|, \|\R_y(\Delta x,\Delta y)\| \leq 8\eta^2 L_{xy} (\|\nabla_x f (x,y)\|^2 + \|\nabla_y f(x,y)\|^2)
    \end{align}    
    
    Also we have, 
    
        \begin{align}\label{cgoremainderboundb} 
        &2\nabla_x f (x,y) ^{\top} \R_x(\Delta x,\Delta y) + 2\nabla_y f(x,y) ^{\top} \R_y(\Delta x,\Delta y)\nonumber\\
        &\quad\quad\stackrel{(a)}{\leq}  2(\|\nabla_x f (x,y)\|\|\R_x(\Delta x,\Delta y) \|+\|\nabla_y f(x,y)\|\|\R_y(\Delta x,\Delta y) \|)\nonumber\\
        &\quad\quad\stackrel{(b)}{\leq}  16L_{xy}\eta^2(\|\nabla_x f (x,y)\|+\|\nabla_y f(x,y)\|)(\|\nabla_x f (x,y)\|^2 + \|\nabla_y f(x,y)\|^2)
    \end{align}    
    
    Where we use Cauchy-Schwarz inequality in $(a)$ and Eq.~\eqref{remainderineq} in $(b)$.
    Finally we use the bounds in Eq.~\eqref{cgoremainderbounda} and Eq.~\eqref{cgoremainderboundb} to bound the terms containing $\R_x(\Delta x,\Delta y)$ and $\R_y(\Delta x,\Delta y)$ in Eq.~\eqref{maincomputation} and further set $k = \eta (\frac{11\eta+16\alpha^2\overline{\lambda_1}}{8} - \frac{15}{8}\alpha)+2(10\eta+8\alpha)\alpha^2\eta\overline{\lambda_2}^2$ to obtain, 
    
\begin{align*}
    &\left\| \nabla_x f \left(\Delta x + x, \Delta y + y \right)\right\|^2 + \left\|\nabla_y f\left(\Delta x + x, \Delta y + y \right)\right\|^2 
    -\|\nabla_x f (x,y)\|^2 - \|\nabla_y f(x,y)\|^2 \\
    &\quad\quad\leq - \nabla_x f (x,y) ^{\top} ( \eta\left(2\nabla_{xx}^2 f-2\frac{10\eta+8\alpha}{\eta}\overline{\lambda_{xx}}^2 \right) + k \left(I + \alpha^2 \nabla_{xy}^2 f\nabla_{yx}^2 f \right)^{-1} \nabla_{xy}^2 f\nabla_{yx}^2 f  \\
    &\quad\quad\quad\quad- 16\eta^2 L_{xy}\left(\|\nabla_x f (x,y)\| + \|\nabla_y f(x,y)\|\right)\\
    &\quad\quad\quad\quad- 384 \eta^4 L_{xy}^2(\|\nabla_x f (x,y)\|^2+\|\nabla_x f (x,y)\|^2) ) \nabla_x f (x,y)    \\
    &\quad\quad\quad\quad- \nabla_y f(x,y) ^{\top} ( -\eta \left(2\nabla_{yy}^2 f+2\frac{10\eta+8\alpha}{\eta}\overline{\lambda_{yy}}^2 \right) + k \left(I + \alpha^2 \nabla_{yx}^2 f\nabla_{xy}^2 f\right)^{-1} \nabla_{yx}^2 f\nabla_{xy}^2 f \\
    &\quad\quad\quad\quad- 16\eta^2 L_{xy} \left(\|\nabla_x f (x,y)\| + \|\nabla_y f(x,y)\|\right)\\
    &\quad\quad\quad\quad- 384 \eta^4 L_{xy}^2(\|\nabla_x f (x,y)\|^2+\|\nabla_y f(x,y)\|^2))\nabla_y f(x,y)        
\end{align*}
    Rearranging we obtain, $$\|(\nabla_{x} f(x+x,y+y), \nabla_{y} f(x+x,y+y)\| \leq (1-\lambda_{\min}) \|(\nabla_{x} f(x,y), \nabla_{y} f(x, y)\|$$

    Thus for $1\geq\lambda_{min}>0$ where,
    \begin{align*} 
         \lambda_{min}=&\min\Big\{\lambda_{min}( \eta\left(2\nabla_{xx}^2 f-2\frac{10\eta+8\alpha}{\eta}\overline{\lambda_{xx}}^2 \right) + k \left(I + \alpha^2 \nabla_{xy}^2 f\nabla_{yx}^2 f \right)^{-1} \nabla_{xy}^2 f\nabla_{yx}^2 f  \\
    &\quad\quad\quad\quad\quad\quad - 16\eta^2 L\left(\|\nabla_x f (x,y)\| + \|\nabla_y f(x,y)\|\right) \\
    &\quad\quad\quad\quad\quad\quad- 384 \eta^4 L^2(\|\nabla_x f (x,y)\|^2+\|\nabla_x f (x,y)\|^2)),\\
    %%%%%
    & \quad\quad\quad\lambda_{min}( -\eta \left(2\nabla_{yy}^2 f+2\frac{10\eta+8\alpha}{\eta}\overline{\lambda_{yy}}^2 \right) + k \left(I + \alpha^2 \nabla_{yx}^2 f\nabla_{xy}^2 f\right)^{-1} \nabla_{yx}^2 f\nabla_{xy}^2 f \\
    &\quad\quad\quad\quad\quad\quad- 16\eta^2 L \left(\|\nabla_x f (x,y)\| + \|\nabla_y f(x,y)\|\right)\\
    &\quad\quad\quad\quad\quad\quad- 384 \eta^4 L^2(\|\nabla_x f (x,y)\|^2+\|\nabla_y f(x,y)\|^2)\Big\}   
    \end{align*}
\end{proof}
we have exponential convergence with rate $(1-\lambda_{min})$.

Now, we simplify the above expression using Lemmas \eqref{lemma:eigen-properties} and \eqref{lemma:eigen-properties1} to obtain,
\begin{align*} 
 \lambda_{min}&\geq  \min\Big\{\eta(2\underline{\lambda_{xx}}-2\frac{10\eta+8\alpha}{\eta}\overline{\lambda_{xx}}^2)+k\frac{\underline{\lambda_{xy}}}{1+
\alpha^2\underline{\lambda_{xy}}
}- 16\eta^2 L \left(\|\nabla_x f (x,y)\| + \|\nabla_y f(x,y)\|\right)\\
&\quad\quad\quad\quad\quad\quad- 384 \eta^4 L^2(\|\nabla_x f (x,y)\|^2+\|\nabla_y f(x,y)\|^2),\\
&\quad\quad\quad\quad-\eta(2\underline{\lambda_{yy}}+2\frac{10\eta+8\alpha}{\eta}\overline{\lambda_{yy}}^2)+k\frac{\underline{\lambda_{yx}}}{1+
\alpha^2\underline{\lambda_{yx}}
}- 16\eta^2 L \left(\|\nabla_x f (x,y)\| + \|\nabla_y f(x,y)\|\right)\\
&\quad\quad\quad\quad\quad\quad- 384 \eta^4 L^2(\|\nabla_x f (x,y)\|^2+\|\nabla_y f(x,y)\|^2)\Big\}
    \end{align*}

When initializing close to the stationary point, the Lipschitz-continuity of the gradient guarantees that the terms $\left(\|\nabla_x f (x,y)\| + \|\nabla_y f(x,y)\|\right)$ and $\|\nabla_x f (x,y)\|^2+\|\nabla_y f(x,y)\|^2$ are small and we have,

\begin{align*} 
 \lambda_{min}&\geq  \min\Big\{\eta(2\underline{\lambda_{xx}}-2\frac{10\eta+8\alpha}{\eta}\overline{\lambda_{xx}}^2)+k\frac{\underline{\lambda_{xy}}}{1+
\alpha^2\underline{\lambda_{xy}}
},\\
&\quad\quad\quad\quad-\eta(2\underline{\lambda_{yy}}+2\frac{10\eta+8\alpha}{\eta}\overline{\lambda_{yy}}^2)+k\frac{\underline{\lambda_{yx}}}{1+
\alpha^2\underline{\lambda_{yx}}}\Big\}
\end{align*}
which is the statement of our Theorem.

\section{Convergence for $\alpha$-coherent functions}\label{strictalpha}

\subsection{\cgo converges to a saddle point under strictly $\alpha$-coherent functions}
\begin{proof}[Proof of Theorem \eqref{strictalphatheorem}]
We prove the convergence through contradiction. 
Let us assume that the algorithm does not converge to a saddle point. Let $z_n:=(x_n,y_n)$ denote the parameters at the $n$'th iterate of the algorithm. $g_{\alpha,n}:=g_{\alpha}(z_n)$ denote the vector $g_{\alpha}$ evaluated at $z_n$. Let the set of saddle points be $\mathcal{Z}^*$, and let all the iterates of the algorithm lie in a compact set $\mathcal{C}$. Then from the assumption we have $\mathcal{Z}^*\cap \mathcal{C} = \phi$. Now from the definition of strict coherence we have $\langle g_{\alpha,n},z-z^*\rangle \geq a$ for some $a>0$ and $z^*\in\mathcal{Z}^*$ and $,\forall z\in \mathcal{C}$. Such a $z^*$ is guaranteed by definition~\eqref{coherencedef}($2^{nd}$ point).

Recall the proximal map defined in Eq.~\eqref{eq:prox},
    \begin{align} \label{prox2}
           z^+ = P_z(y) = \argmin_{z' \in \mathcal{Z}} \{\langle y,z-z'\rangle+D(z',z)\} = \argmax_{z' \in \mathcal{Z}} \{\langle y+\nabla h(z),z'\rangle-h(z')\}.
        \end{align}

Then for Bregmann Divergence $D_h(x,y)$ with K-strongly convex potential function $h$ and 2-norm $\|.\|$ we have, \cite{mertikopoulos2018optimistic}(Proposition B.3), 
    \begin{align} \label{B3}
           D(p,z^+) \leq D(p,z)+\langle y,z-p\rangle+\frac{K}{2}\|\Delta y\|^2.
    \end{align}

To obtain the \cgo update we substitute $y = -\eta_n g_{\alpha, n},z=z_n,z^+=z_{n+1},p=z^*,h = \frac{\|.\|_2^2}{2}$ in Eq.~\eqref{B3} we get,    
\begin{align*} 
           D(z^*,z_{n+1}) = D(z^*,P_{z_n}(-\eta_ng_{\alpha, n}))\leq D(z^*,z_n)-\eta_{n}\langle g_{\alpha, n},z_n-z^*\rangle+\frac{\eta_n^2\|g_{\alpha, n}\|^2}{2}
\end{align*}
Since the saddle point is $\alpha$-coherent we have $\langle g_{\alpha, n},z-z^*\rangle \geq a$ for some $a>0$.
\begin{align*} 
           D(z^*,z_{n+1}) \leq D(z^*,z_n)-\eta_{n}a+\frac{\eta_n^2\|g_{\alpha, n}\|^2}{2} \leq D(z^*,z_0)- (a-\frac{\sum_{k=1}^n \|\eta_{k}\|^2}{2\sum_{k=1}^n\eta_k})\sum_{k=1}^n\eta_k
\end{align*}

Since we have $\sum_{k=1}^n \eta_k = \infty$ and $\sum_{k=1}^n\|\eta_k\|^2 < \infty$, we obtain $\lim_{n\rightarrow\infty}D_n = -\infty$, which is a contradiction since the divergence is positive. Hence \cgo converges to a saddle point.
\end{proof}
\subsection{\ocgo converges to a saddle point under $\alpha$-coherent functions}

\begin{proof}[Proof of Theorem \eqref{alphatheorem}]\label{lalpha}
Let $P_z(y)$ be as in Eq.~\eqref{eq:prox} and $z^+_1 = P_z(y_1),z^+_2 = P_z(y_2)$. We then have for Bregmann Divergence $D_h(x,y)$ with K-strongly convex potential function $h$, 2-norm $\|.\|$ and a fixed point p \cite{mertikopoulos2018optimistic}(Proposition B.4), 
\begin{align} \label{B4}
    D(p,x^+_2) \leq D(p,x)+\langle y_2,x^+_1-p\rangle + \frac{1}{2K}\|\Delta y_2-y_1\|^2-\frac{K}{2}
\|\Delta x_1^+-x\|^2    .
\end{align}
Let $p^*$ be a solution of the SP problem such that $\alpha$-MVI holds $\forall z\in \X\times\Y$, the existence of such a $p$ is guaranteed via the definition of $\alpha$-coherence Def.~\eqref{coherencedef}($2^{nd}$ point).

In order to obtain the $\ocgo$ update we substitute $y_1 = -\eta_n g_{\alpha,n},y_2 = -\eta_n g_{\alpha,n+\frac{1}{2}},x = z_n,x_1^+ = z_{n+\frac{1}{2}},x_2^+ = z_{n+1},p=p^*$ and set $h = \frac{\|.\|^2}{2}$ (for this $h$ we have $K=1$),
$$D(x^*,z_{n+1})\leq D(x^*,z_n)-\eta_n\langle g_{\alpha,n+\frac{1}{2}},z_{n+\frac{1}{2}}-x^*\rangle+\frac{\eta_n^2}{2}\|g_{\alpha,n+\frac{1}{2}}-g_{\alpha,n}\|^2-\frac{1}{2}\|z_{n+\frac{1}{2}}-z_n\|^2$$
From coherence condition we have,
\begin{align} \label{rel1}
    D(p^*,z_{n+1})\leq D(p^*,z_n)+\frac{\eta_n^2}{2}\|g_{\alpha,n+\frac{1}{2}}-g_{\alpha,n}\|^2-\frac{1}{2}\|z_{n+\frac{1}{2}}-z_n\|^2
\end{align}
Using Eq.~\eqref{form1} we get,
\begin{align*} \|g_{\alpha,n+\frac{1}{2}}-g_{\alpha,n}\|^2 &\leq \|g_{0,n+\frac{1}{2}}-g_{0,n}\|^2+\frac{\alpha^2}{\eta_n^2}(\| \nabla_{xy,n+\frac{1}{2}}f\Delta y_{n+\frac{1}{2}}-\nabla_{xy,n}f\Delta y_{n}\|^2\\
&\quad\quad+\|\nabla_{xy,n+\frac{1}{2}}f^{\top} \Delta x_{n+\frac{1}{2}}-\nabla_{xy,n}f^{\top} \Delta x_n\|^2)
\end{align*}
Where $(\Delta x_n,\Delta y_n) = -\eta_n g_{\alpha,n}$, $(\Delta x_{n+\frac{1}{2}},\Delta y_{n+\frac{1}{2}}) = -\eta_n g_{\alpha,n+\frac{1}{2}}$ and $\nabla_{xy,n}f,\nabla_{xy,n+\frac{1}{2}}f$ are the second order cross terms evaluated at $z_n,z_{n+\frac{1}{2}}$.
We can re-write the above as,
\begin{align*} 
\frac{1}{\eta_n^2}\|(\Delta x_{n+\frac{1}{2}}&-\Delta x_n,\Delta y_{n+\frac{1}{2}}-\Delta y_n)\|^2\\
&\leq  \|g_{0,n+\frac{1}{2}}-g_{0,n}\|^2\\
&\quad\quad+\frac{\alpha^2}{\eta_n^2}\| \nabla_{xy,n+\frac{1}{2}}f\Delta y_{n+\frac{1}{2}}-\nabla_{xy,n+\frac{1}{2}}f\Delta y_{n}+\nabla_{xy,n+\frac{1}{2}}f\Delta y_{n}-\nabla_{xy,n}f\Delta y_{n}\|^2\\
&\quad\quad+\frac{\alpha^2}{\eta_n^2}\|\nabla_{xy,n+\frac{1}{2}}f^{\top} \Delta x_{n+\frac{1}{2}}-\nabla_{xy,n+\frac{1}{2}}f^{\top} \Delta x_n+\nabla_{xy,n+\frac{1}{2}}f^{\top} \Delta x_n-\nabla_{xy,n}f^{\top} \Delta x_n\|^2\\
&\leq \|g_{0,n+\frac{1}{2}}-g_{0,n}\|^2\\
&\quad\quad+\frac{\alpha^2}{\eta_n^2}(\| \nabla_{xy,n+\frac{1}{2}}f\|^2\|\Delta y_{n+\frac{1}{2}}-\Delta y_{n}\|^2+\|\Delta y_{n}\|^2\|\nabla_{xy,n+\frac{1}{2}}f-\nabla_{xy,n}f\|^2)\\
&\quad\quad+\frac{\alpha^2}{\eta_n^2}( \|\nabla_{xy,n+\frac{1}{2}}f^{\top}\|^2\| \Delta x_{n+\frac{1}{2}}-\Delta x_n\|^2+\|\Delta x_n\|^2\|\nabla_{xy,n+\frac{1}{2}}f^{\top}-\nabla_{xy,n}f^{\top} \|^2)
\end{align*}

Using the Lipschitz continuity of the Hessian terms, setting $\alpha^2\| \nabla_{xy,n+\frac{1}{2}}f\|^2 = \alpha^2\| \nabla_{xy,n+\frac{1}{2}}f^{\top}\|^2 =\alpha^2 L_{xy}^2\leq 1$, and rearranging we get,

\begin{align*} 
\frac{1}{\eta_n^2}\|\Delta x_{n+\frac{1}{2}}-\Delta x_n,\Delta y_{n+\frac{1}{2}}-\Delta y_n\|^2 &\leq  \frac{1}{ 1-\| \nabla_{xy,n+\frac{1}{2}}f\|^2\alpha^2}\|g_{0,n+\frac{1}{2}}-g_{0,n}\|^2\\
&\quad\quad+\frac{\alpha^2}{ \eta_n^2(1-\| \nabla_{xy,n+\frac{1}{2}}f\|^2\alpha^2)}(\|\Delta y_{n}\|^2\|\nabla_{xy,n+\frac{1}{2}}f-\nabla_{xy,n}f\|^2)\\
&\quad\quad+\frac{\alpha^2}{ \eta_n^2(1-\| \nabla_{xy,n+\frac{1}{2}}f\|^2\alpha^2)}(\|\Delta x_n\|^2\|\nabla_{xy,n+\frac{1}{2}}f^{\top}-\nabla_{xy,n}f^{\top} \|^2)\\ 
&\leq  \frac{L^2}{ \eta_n^2(1-\| \nabla_{xy,n+\frac{1}{2}}f\|^2\alpha^2)}\|z_{n+\frac{1}{2}}-z_n\|^2\\&\quad\quad+\frac{L_{xy}^2\alpha^2}{ \eta_n^2(1-\| \nabla_{xy,n+\frac{1}{2}}f\|^2\alpha^2)}(\|\Delta x_{n}\|^2+\|\Delta y_{n}\|^2)\|z_{n+\frac{1}{2}}-z_n\|^2
\end{align*}

Finally we have,

\begin{align} 
\|g_{\alpha,n+\frac{1}{2}}-g_{\alpha,n}\|^2\leq & \frac{L^2+L_{xy}^2\alpha^2(\|\Delta x_{n}\|^2+\|\Delta y_{n}\|^2)}{ \eta_n^2(1-\| \nabla_{xy,n+\frac{1}{2}}f\|^2\alpha^2)}\|z_{n+\frac{1}{2}}-z_n\|^2\nonumber\\ 
\leq & \frac{L^2+L_{xy}^2\alpha^2(\alpha+\eta)^2\|g_{0,n}\|^2}{ \eta_n^2(1-\| \nabla_{xy,n+\frac{1}{2}}f\|^2\alpha^2)}\|z_{n+\frac{1}{2}}-z_n\|^2
\end{align}

Substituting in Eq.~\eqref{rel1} we get,

\begin{align} \label{rel2}
    D(p^*,z_{n+1})\leq &D(p^*,z_n)+\|z_{n+\frac{1}{2}}-z_n\|^2 (\frac{\eta_n^2L'^2+L_{xy}^2\alpha^2(\alpha+\eta_n)^2\|g_{0,n}\|^2}{ 2(1-\| \nabla_{xy,n+\frac{1}{2}}f\|^2\alpha^2)}-\frac{1}{2})\nonumber\\ \leq &D(p^*,z_n)+\|z_{n+\frac{1}{2}}-z_n\|^2 ( \frac{\eta_n^2L'^2+L_{xy}^2\alpha^2(\alpha+\eta_n)^2L^2}{ 2(1-\| \nabla_{xy,n+\frac{1}{2}}f\|^2\alpha^2)}-\frac{1}{2})
\end{align}

Hence if $\alpha$ satisfies the following,

\begin{align*}
    \alpha^4L_{xy}^2L^2+\alpha^2L'^2-1 < 0 
\end{align*}
or equivalently we have, 
\begin{align} \label{alpha}
    -\sqrt{\frac{\sqrt{L'^4+4L_{xy}^2L^2}-L'^2}{2L_{xy}^2L^2}}<\alpha < \sqrt{\frac{\sqrt{L'^4+4L_{xy}^2L^2}-L'^2}{2L_{xy}^2L^2}}
\end{align}

and also $\eta_n$ satisfying the following, 

\begin{align} \label{eta}
    0< \eta_n < \frac{\sqrt{\alpha^2 L^2 L_{xy}^2 + L'^2-2\alpha^4L^2 L'^2 L_{xy}^2 - \alpha^2 L'^4 } - \alpha^3 L_{0}^2 L_{xy}^2}{\alpha^2 L^2 L_{xy}^2 + L'^2}
\end{align}

We have,

$$\frac{\eta_n^2L'^2+L_{xy}^2\alpha^2(\alpha+\eta_n)^2L^2}{ 2(1-\| \nabla_{xy,n+\frac{1}{2}}f\|^2\alpha^2)}-\frac{1}{2} < 0$$

and the divergence decreases at each step. 
By telescoping Eq.~\eqref{rel2} we obtain,

\begin{align} \label{rel3}
    \sum_{k=1}^n\|z_{k+\frac{1}{2}}-z_k\|^2 ( 1-\frac{\eta_k^2L'^2+L_{xy}^2\alpha^2(\alpha+\eta_k)^2L^2}{(1-\| \nabla_{xy,k+\frac{1}{2}}\|^2\alpha^2)})\leq &2D(x^*,z_1).
\end{align}

We also know $z_{k+\frac{1}{2}}-z_k = -\eta_k g_{\alpha,k}$, thus for $\alpha$ and $\eta_n$ satisfying Eq.~\eqref{alpha} and Eq.~\eqref{eta}, we have,

\begin{align} \label{rel4}
    \frac{1}{n}\sum_{k=1}^n\|z_{k+\frac{1}{2}}-z_k\|^2 = \frac{1}{n}\sum_{k=1}^n\eta_k^2 \|g_{\alpha,k}\|^2 \leq & \frac{2}{nc}D(x^*,z_1).
\end{align}

Where $ 1-\frac{\eta_k^2L'^2+L_{xy}^2\alpha^2(\alpha+\eta_k)^2L^2}{(1-\| \nabla_{xy,k+\frac{1}{2}}\|^2\alpha^2)} > c,\forall \eta_k$. 
If we assume without loss of generality that $\eta_k$ converges to $\eta$, then we have from Eq.~\eqref{rel4} that the average of $\|g_{\alpha,n}\|$ and $\|z_{n+\frac{1}{2}}-z_n\|$ falls with order $O(\frac{1}{n})$ where $n$ is the iteration count.\\ 
Taking limit of $z_{k+\frac{1}{2}}$ we have,
$$z^* =\lim_{k\rightarrow \infty} z_{k+\frac{1}{2}} = P_{z^*}(-\eta g_{\alpha}(z^*)),$$
this implies $z^*$ satisfies $\alpha$-SVI and is hence a solution of the SP problem via definition of $\alpha$-coherence Def.~\eqref{coherencedef} ($1^{st}$ point ). 

Coherence condition Def.~\eqref{coherencedef}($3^{rd}$ point) implies that $\alpha$-MVI holds locally around $z^*$. 
Thus for, $\alpha$ and $\eta_n$ satisfying Eq.~\eqref{alpha} and Eq.~\eqref{eta} respectively, and sufficiently large $n$, we have,

\begin{align*}
    D(z^*,z_{n+1})\leq  &D(z^*,z_n)+\|z^*-z_n\|^2 ( \frac{\eta_n^2L'^2+L_{xy}^2\alpha^2(\alpha+\eta_n)^2L^2}{ 2(1-\| \nabla_{xy,n+\frac{1}{2}}f\|^2\alpha^2)}-\frac{1}{2}) \stackrel{(a)}{\leq} D(z^*,z_n)
\end{align*}
Where the equality in $(a)$ holds if and only if $z^*=z_n$.  Thus $D(z^*,z_n)$ is non-increasing and $z_n \rightarrow z^*$ which is a saddle point.
\end{proof}
\section{Additional simulations}
We now present some more simulations of \cgo and \ocgo on the function $x^\top A y$ with multiple samples of the matrix $A=(a_{ij}),a_{ij}\sim \N (0,1)$.
\begin{figure}[H]
\centering
\subcaptionbox{\cgo \label{a}}
[.24 \textwidth]{\includegraphics[width=\linewidth]{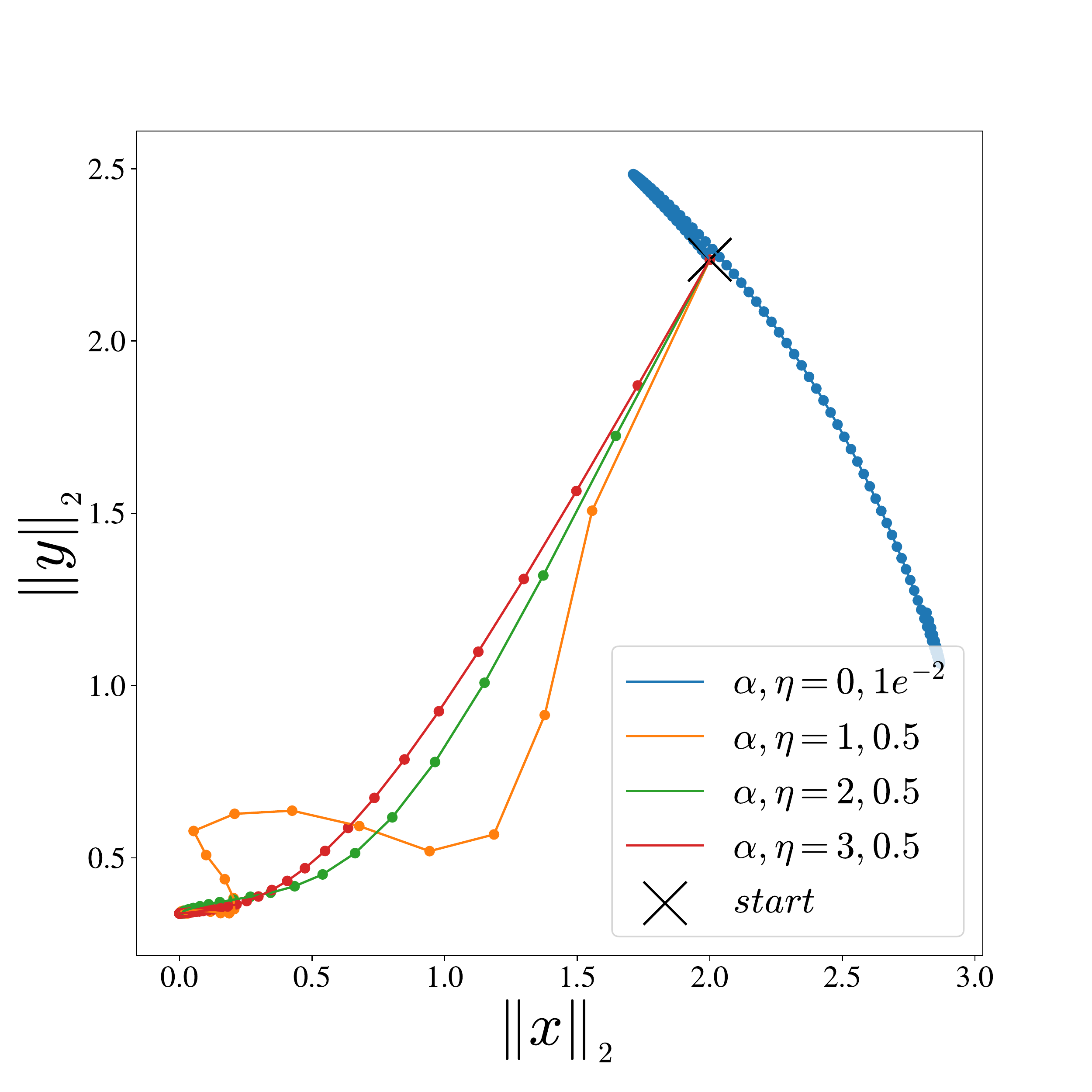}}\subcaptionbox{\ocgo\label{b}}
[.24 \textwidth]{\includegraphics[width=\linewidth]{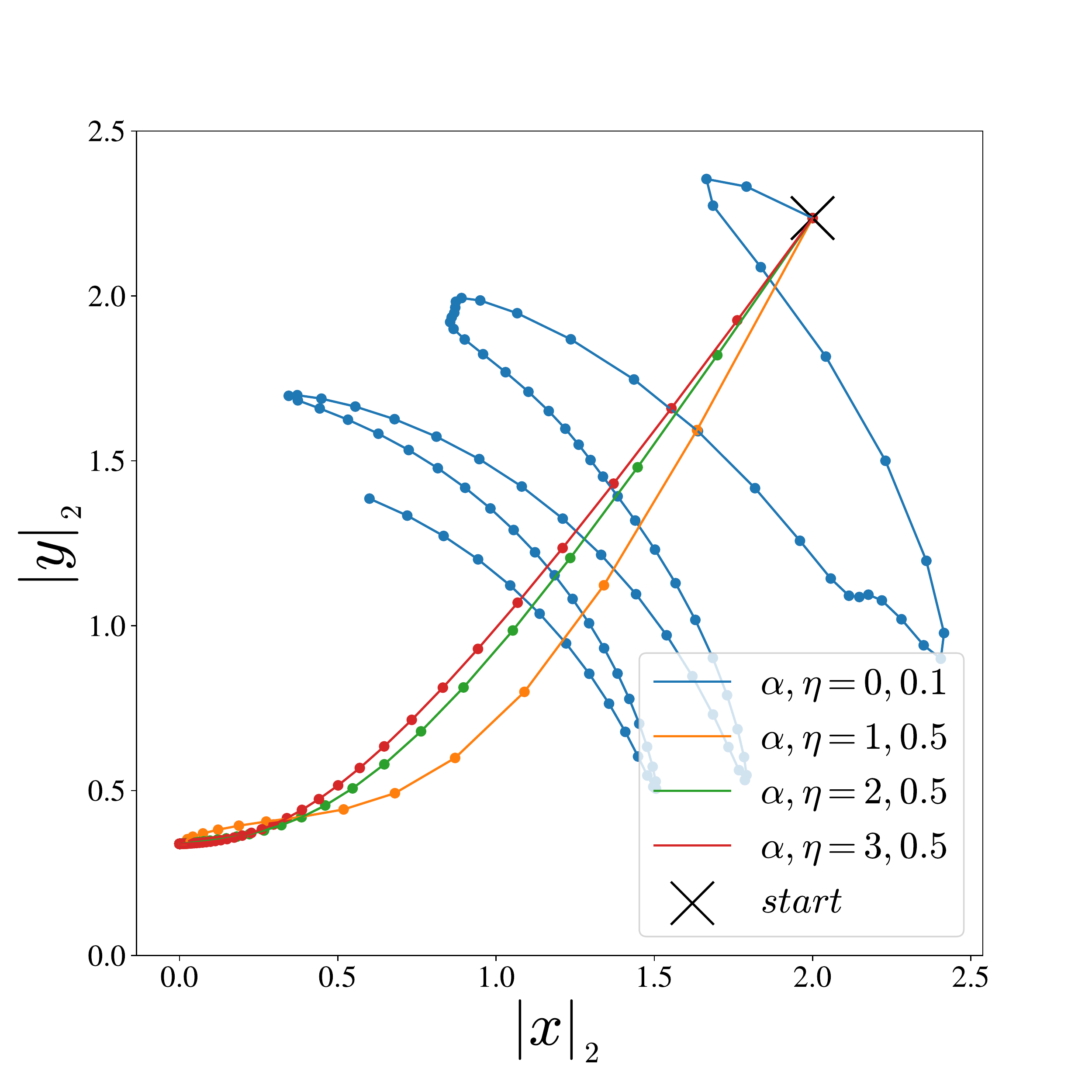}}\subcaptionbox{\cgo\label{c}}
[.24 \textwidth]{\includegraphics[width=1\linewidth]{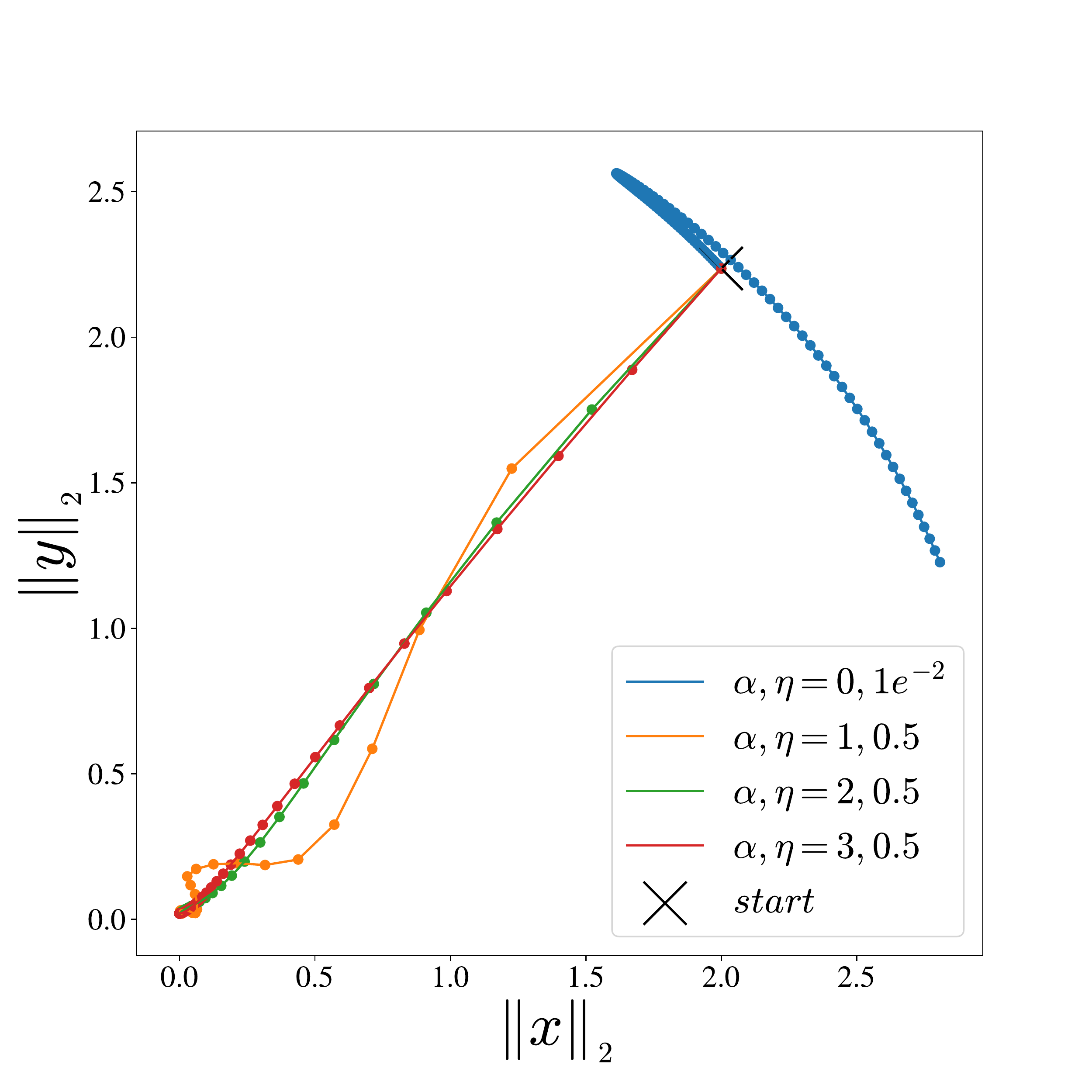}}\subcaptionbox{\ocgo\label{d}}
[.24 \textwidth]{\includegraphics[width=1\linewidth]{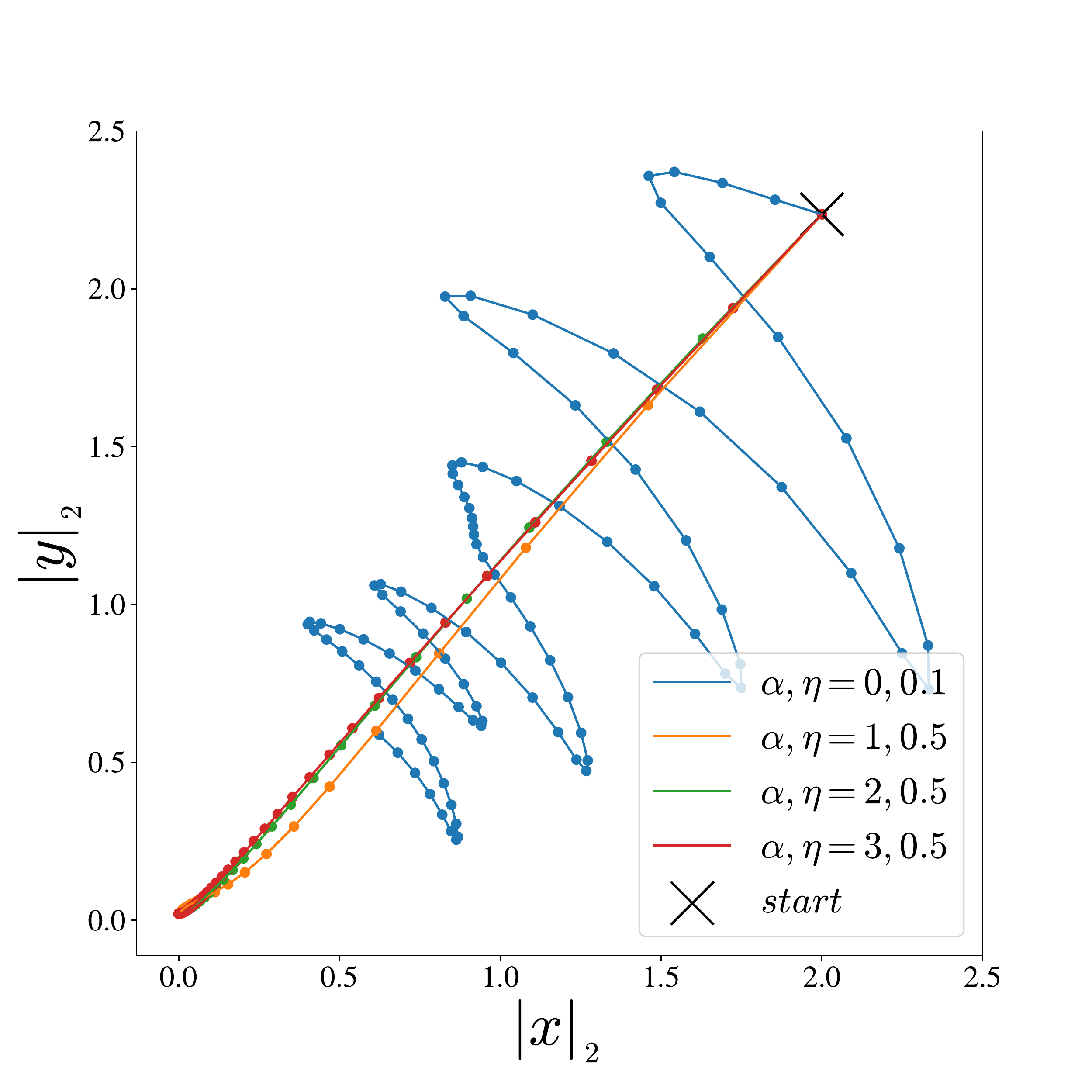}}
\subcaptionbox{\cgo \label{e}}
[.24 \textwidth]{\includegraphics[width=\linewidth]{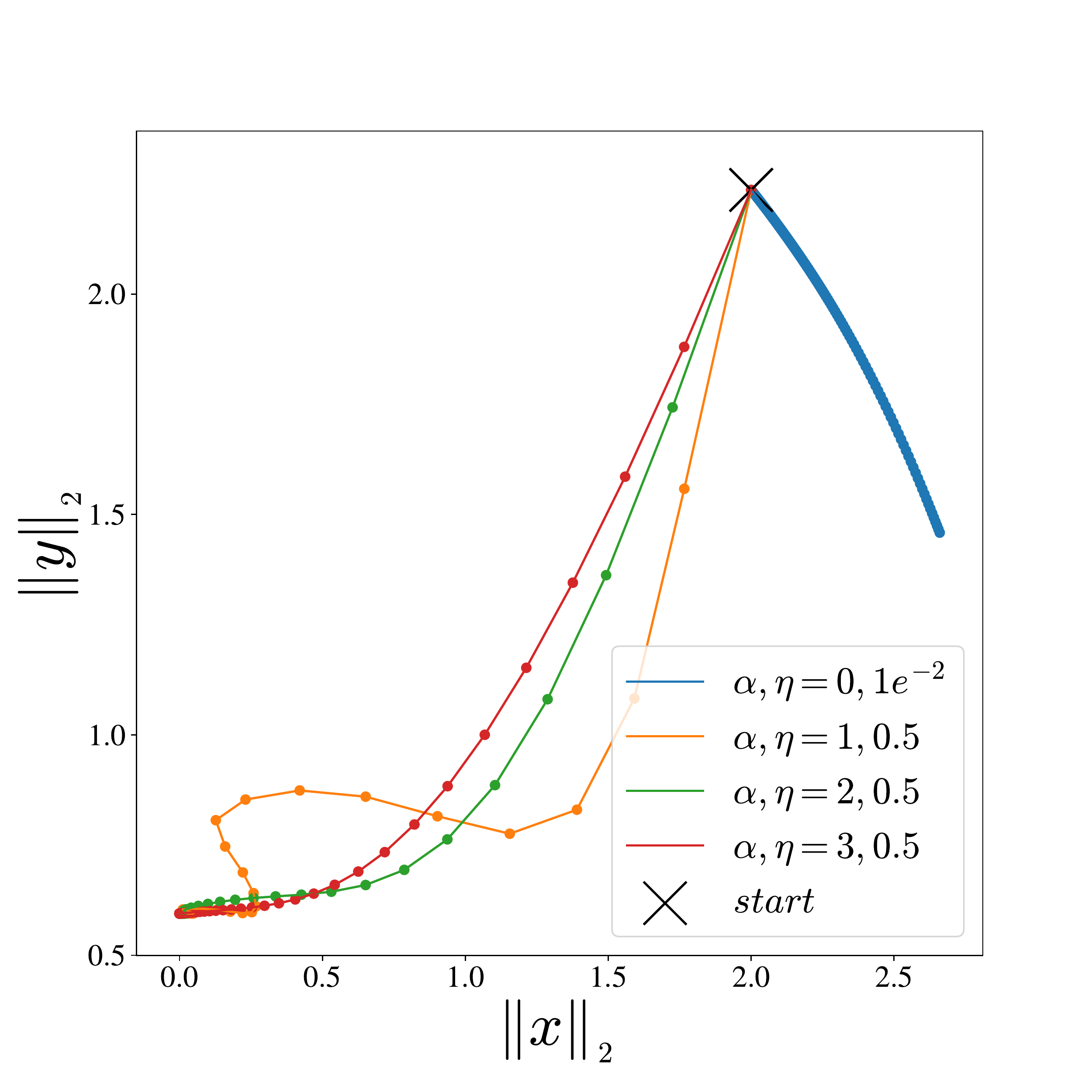}}\subcaptionbox{\ocgo\label{f}}
[.24 \textwidth]{\includegraphics[width=\linewidth]{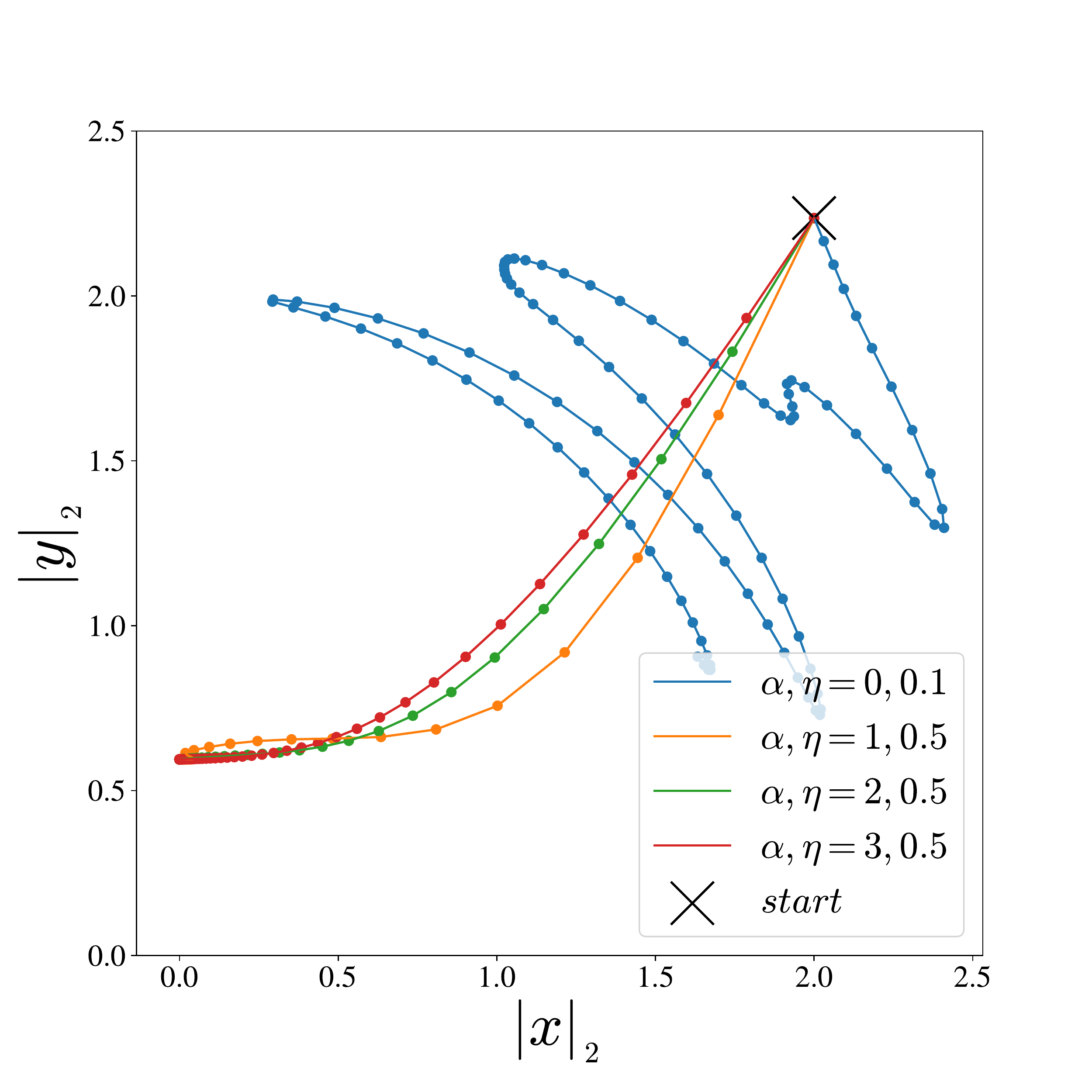}}\subcaptionbox{\cgo\label{g}}
[.24 \textwidth]{\includegraphics[width=1\linewidth]{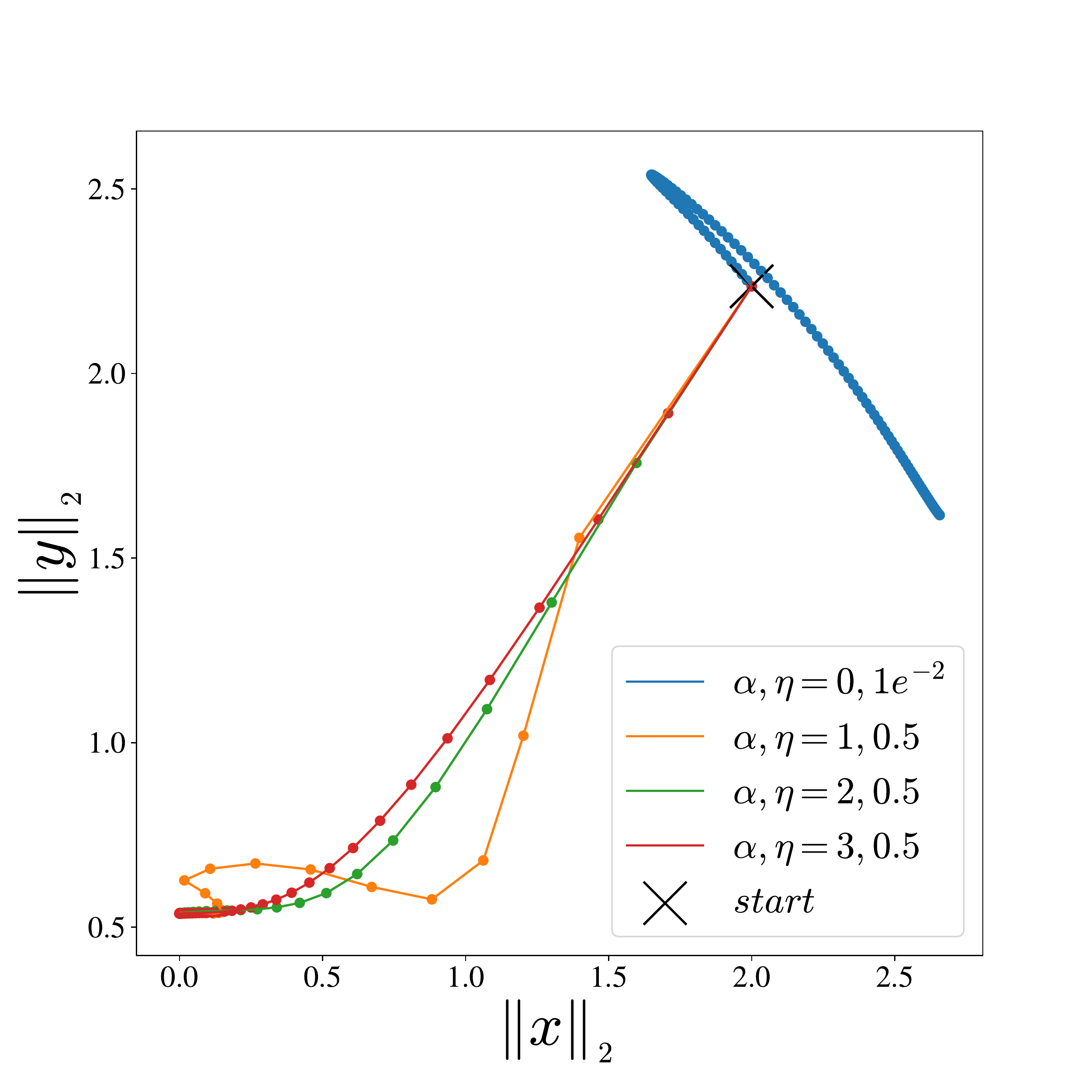}}\subcaptionbox{\ocgo\label{h}}
[.24 \textwidth]{\includegraphics[width=1\linewidth]{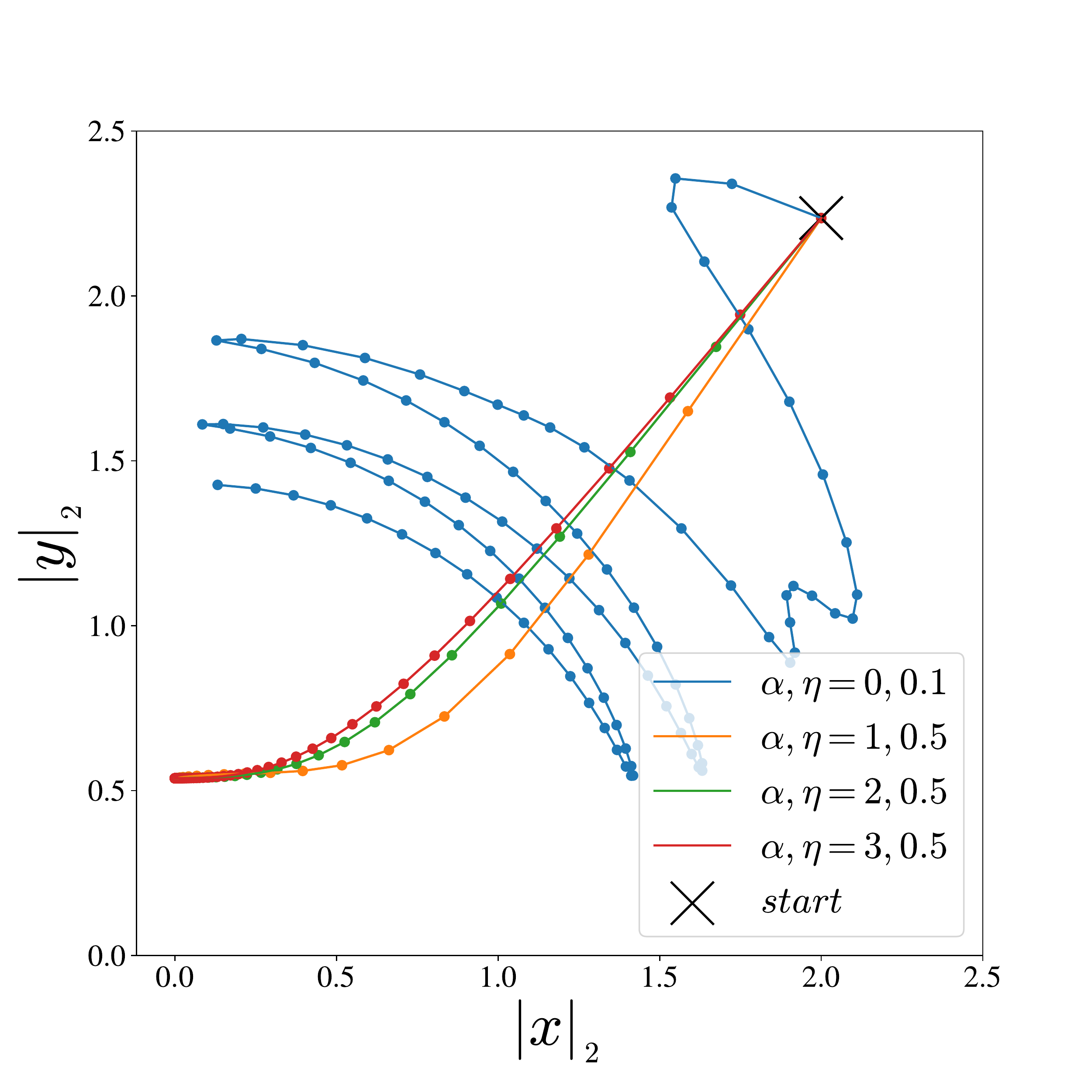}}
\subcaptionbox{\cgo \label{i}}
[.24 \textwidth]{\includegraphics[width=\linewidth]{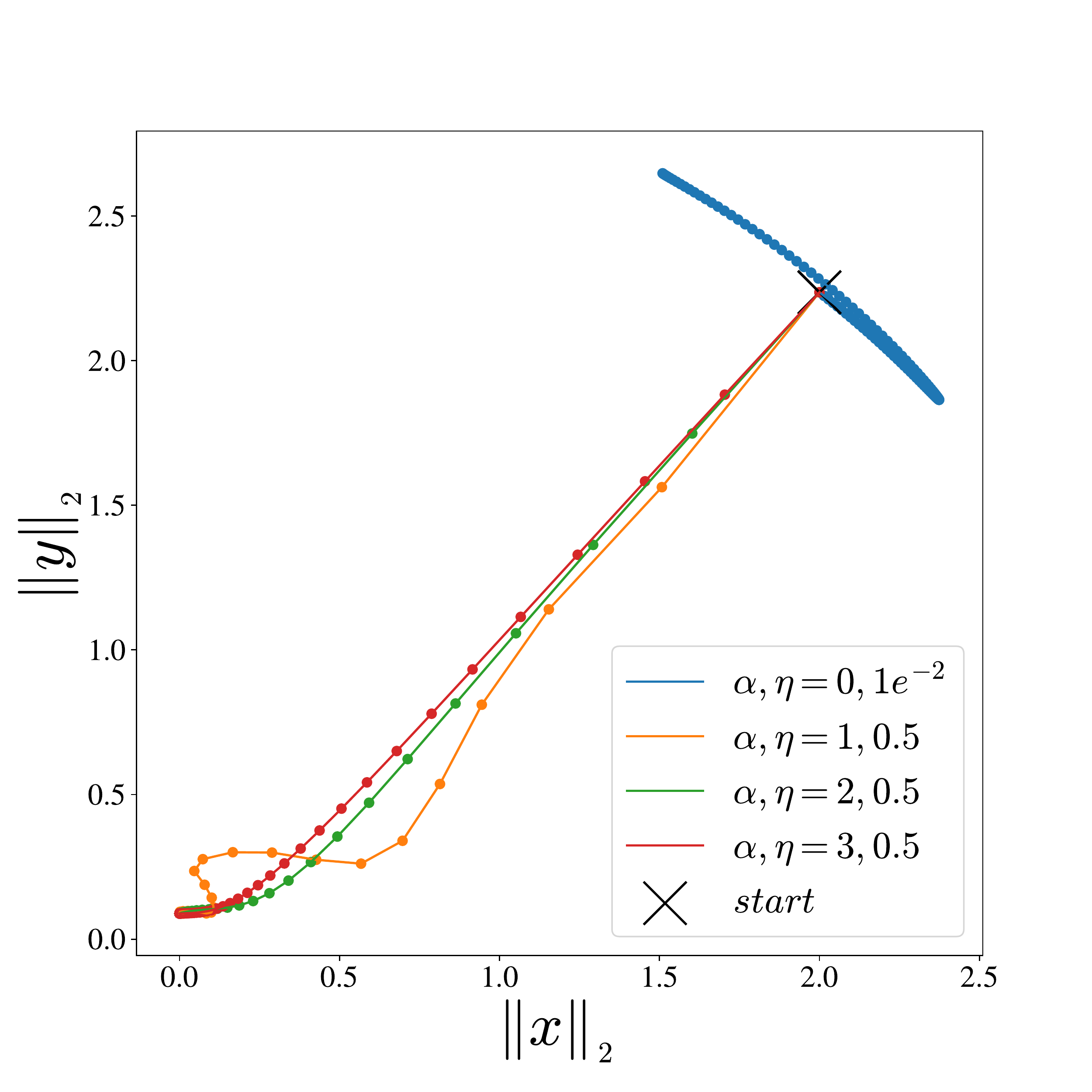}}\subcaptionbox{\ocgo\label{j}}
[.24 \textwidth]{\includegraphics[width=\linewidth]{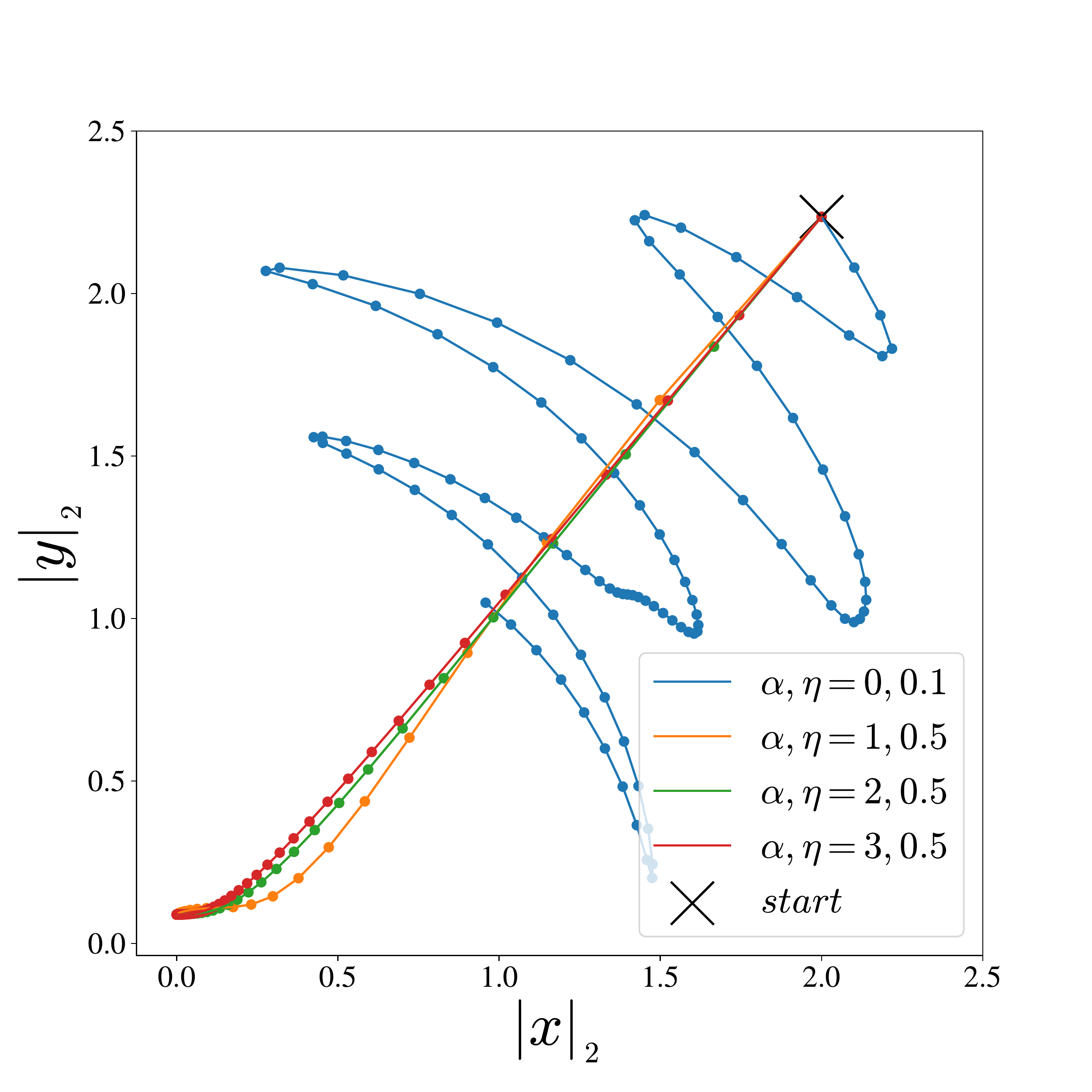}}\subcaptionbox{\cgo\label{k}}
[.24 \textwidth]{\includegraphics[width=1\linewidth]{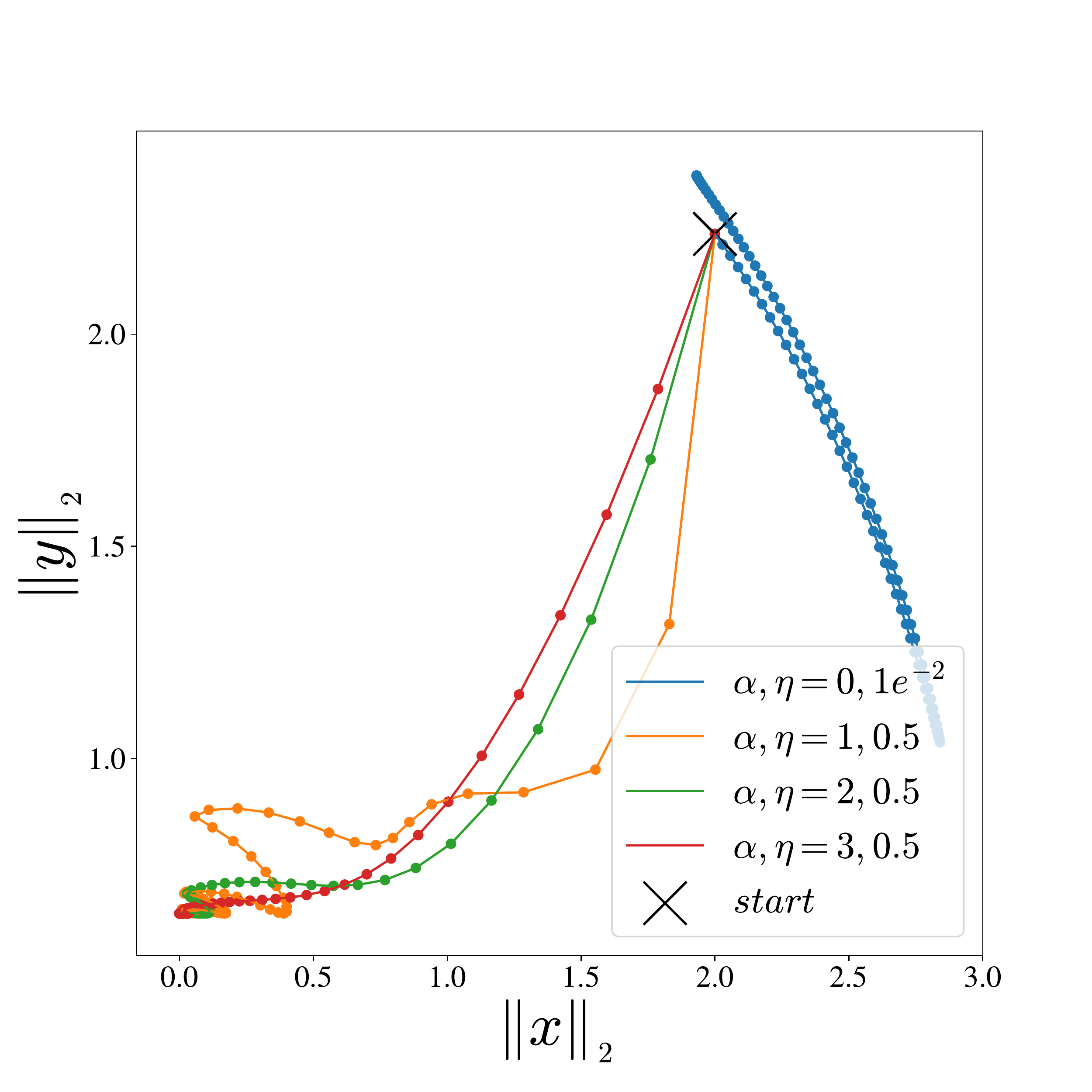}}\subcaptionbox{\ocgo\label{l}}
[.24 \textwidth]{\includegraphics[width=1\linewidth]{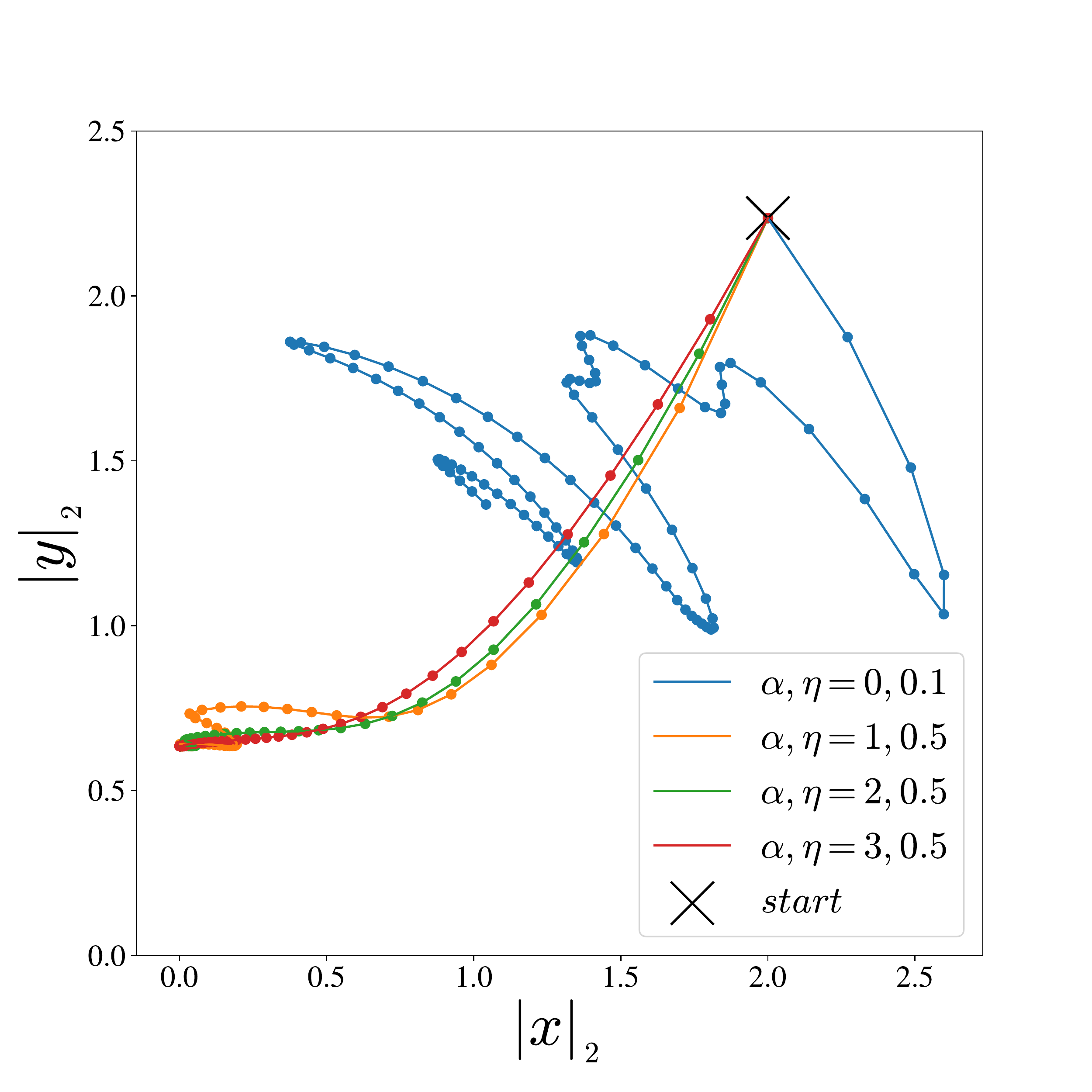}}
\subcaptionbox{\cgo \label{m}}
[.24 \textwidth]{\includegraphics[width=\linewidth]{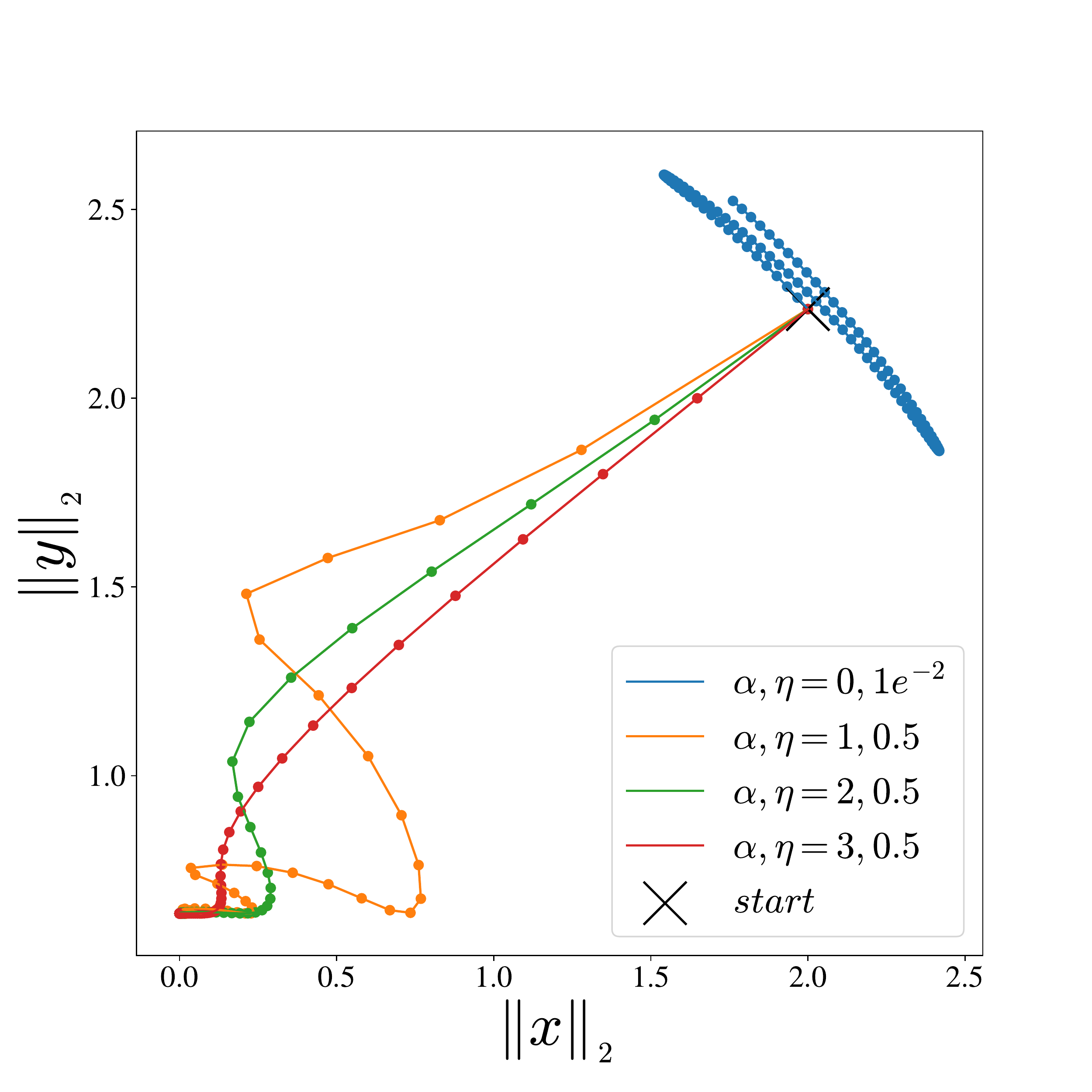}}\subcaptionbox{\ocgo\label{n}}
[.24 \textwidth]{\includegraphics[width=\linewidth]{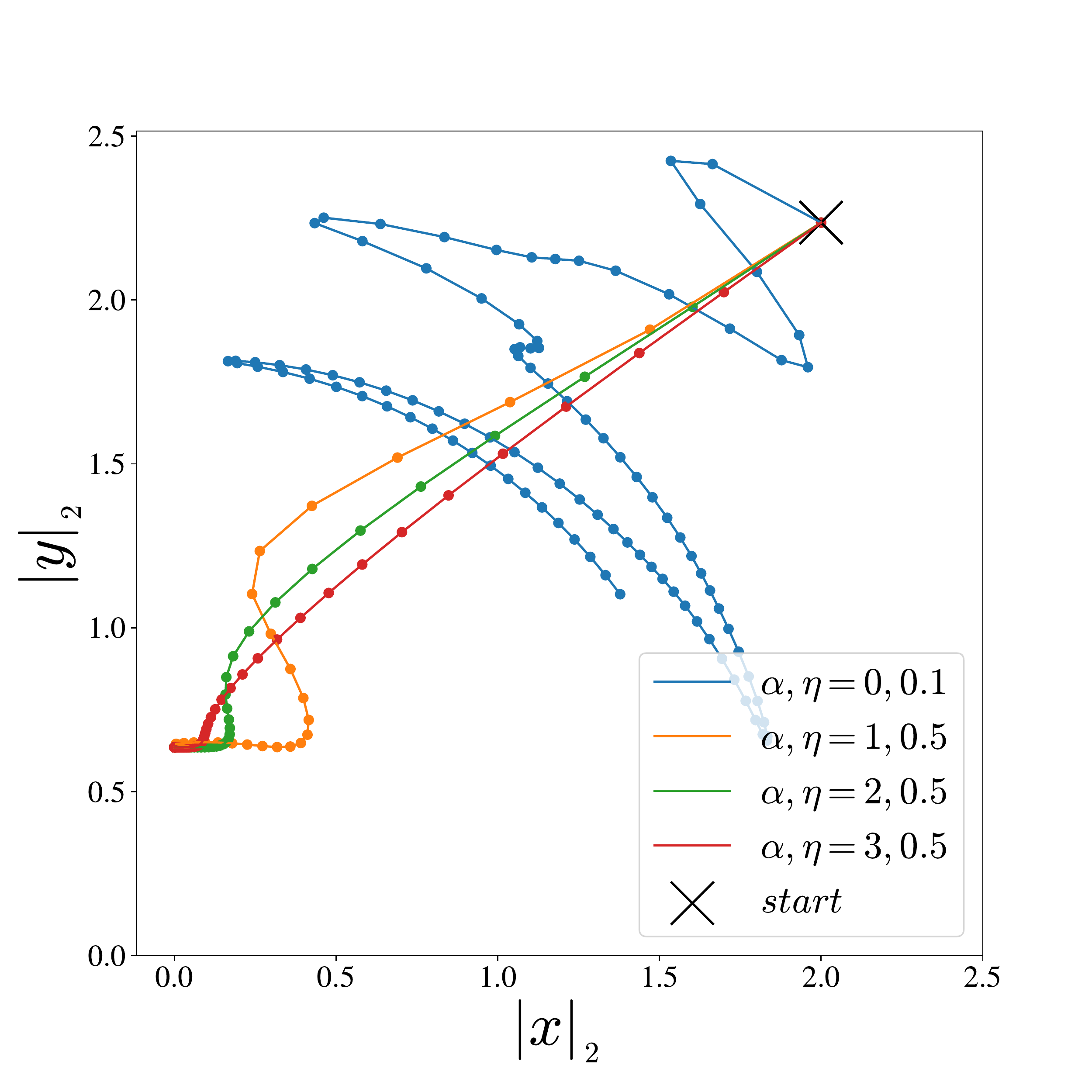}}\subcaptionbox{\cgo\label{o}}
[.24 \textwidth]{\includegraphics[width=1\linewidth]{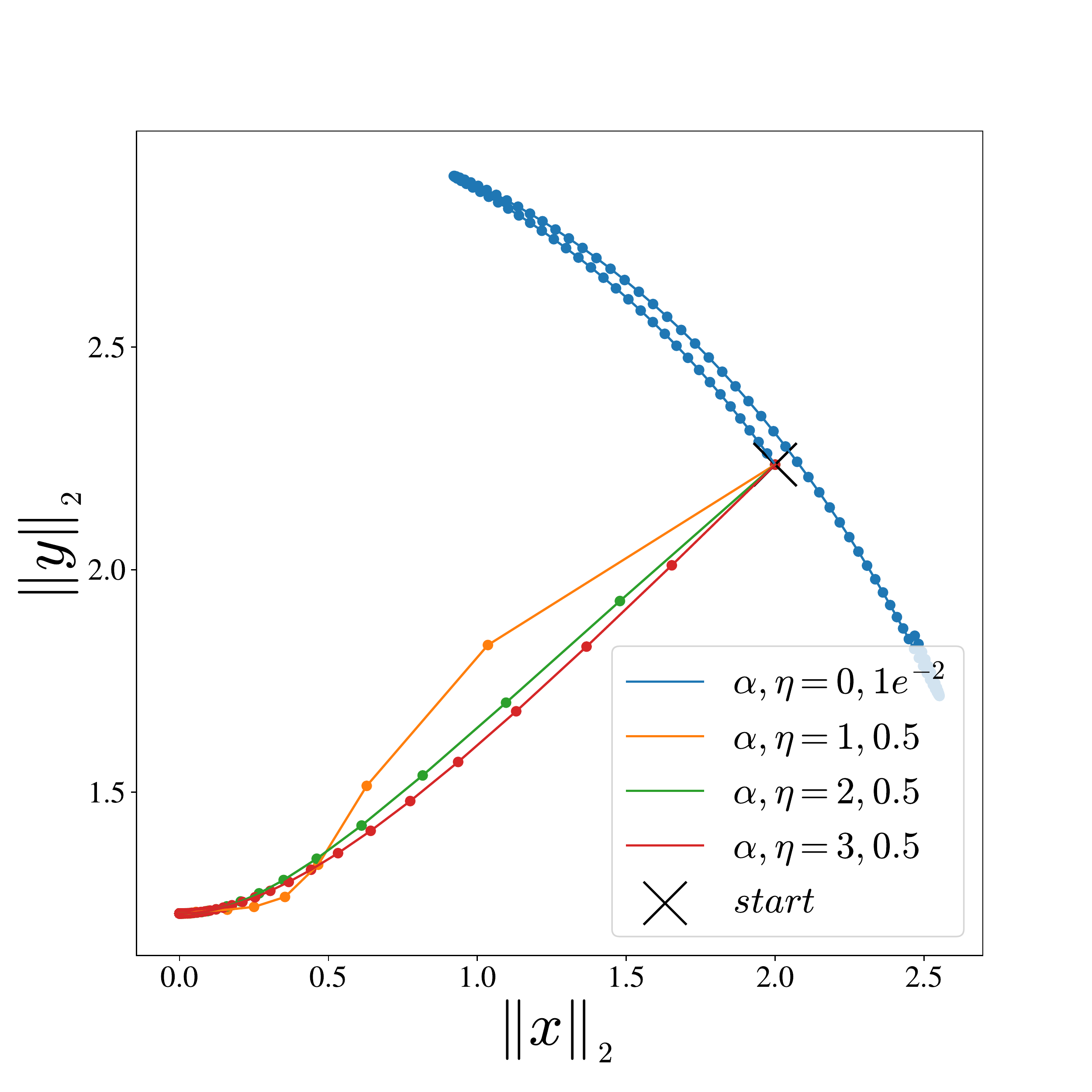}}\subcaptionbox{\ocgo\label{p}}
[.24 \textwidth]{\includegraphics[width=1\linewidth]{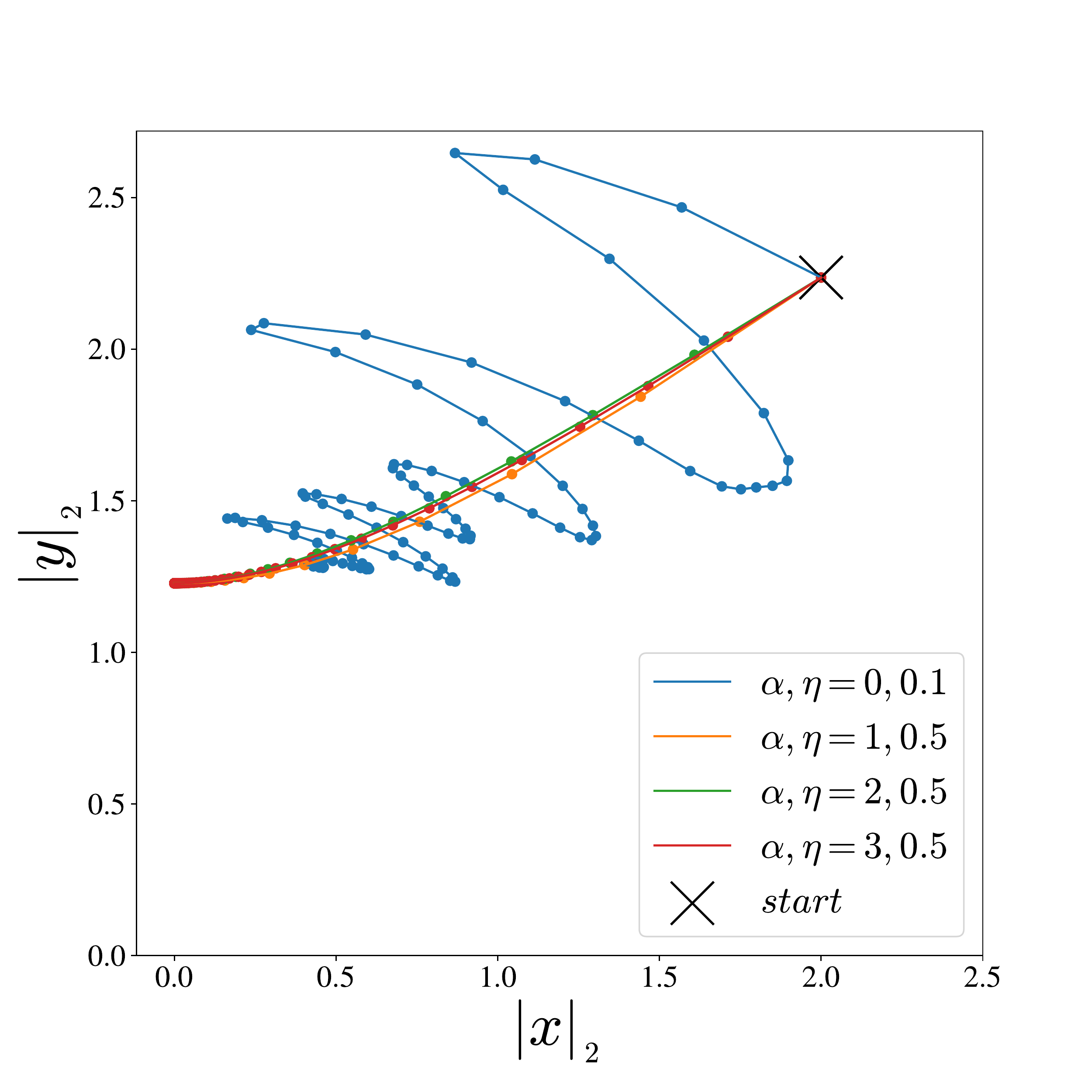}}
\caption{\cgo and \ocgo on bilinear function family $f(x,y) = x^\top A y, x\in \Real^4,y\in \Real^5$ for 100 iterations. In each row, the $1^{st}$ and $2^{nd}$ as well as the $3^{rd}$ and $4^{th}$ figures correspond to the same sample of $A$}
\end{figure}

\end{document}